\documentclass[10pt,english]{amsart}

\usepackage[T1]{fontenc}
\usepackage{tikz}
\usetikzlibrary{fadings}
\usepackage[a4paper]{geometry}
\usepackage{color}
\usepackage{amstext}
\usepackage{amsthm}
\usepackage{amssymb}
\usepackage{varioref}
\usepackage{amssymb}
\usepackage{mathrsfs}
\usepackage{enumerate}
\usepackage[all]{xy}
\usepackage{graphicx}
\usepackage{url}
\usepackage{float}
\usepackage[pagebackref=true]{hyperref}
\usepackage{hyperref}
\usepackage{backref}
\hypersetup{
    colorlinks,
    citecolor=blue,
    filecolor=black,
    linkcolor=blue,
    urlcolor=black
}

\title{Natural operations in Intersection Cohomology}

\date{\today}

\author{David Chataur}
\address{Lamfa\\
Universit\'e de Picardie Jules Verne\\
33, rue Saint-Leu\\
80039 Amiens Cedex~1\\
         France}
\email{David.Chataur@u-picardie.fr}

\author{Daniel Tanr\'e}
\address{D\'epartement de Math{\'e}matiques\\
         UMR-CNRS 8524 \\
         Universit\'e de Lille\\
         59655 Villeneuve d'Ascq Cedex\\
         France}
\email{Daniel.Tanre@univ-lille.fr}

\thanks{The first author was supported by the research project ANR-18-CE93-0002  ``OCHOTO''. 
The second author was partially supported by the MINECO and FEDER research project MTM2016-78647-P
and the ANR-11-LABX-0007-01  ``CEMPI''}



\makeatletter
\renewcommand\l@subsection{\@tocline{2}{0pt}{2pc}{5pc}{}}
\renewcommand\l@subsubsection{\@tocline{3}{0pt}{4pc}{10pc}{}}
\makeatother

\theoremstyle{plain}
\newtheorem{theorem}{Theorem}

\newtheorem{conjecture}{Conjecture}

\newtheorem{proposition}{Proposition}[section]
\newtheorem{theoremb}[proposition]{Theorem}
\newtheorem{lemma}[proposition]{Lemma}
\newtheorem{corollary}[proposition]{Corollary}

\theoremstyle{definition}
\newtheorem{definition}[proposition]{Definition}
\newtheorem{example}[proposition]{Example}

\newtheorem{question}{Questions}

\newtheorem{perspective}{Perspective}

\theoremstyle{remark}
\newtheorem{remark}[proposition]{Remark}

\numberwithin{equation}{section}
   \newcommand{\negro}{\normalcolor} 
       
\newcommand{\appendref}[1]{Appendice~\ref{#1}}
\newcommand{\secref}[1]{Section~\ref{#1}}
\newcommand{\subsecref}[1]{Subsection~\ref{#1}}
\newcommand{\thmref}[1]{Theorem~\ref{#1}}
\newcommand{\propref}[1]{Proposition~\ref{#1}}
\newcommand{\lemref}[1]{Lemma~\ref{#1}}
\newcommand{\corref}[1]{Corollary~\ref{#1}}
\newcommand{\conjref}[1]{Conjecture~\ref{#1}}
\newcommand{\remref}[1]{Remark~\ref{#1}}

\newcommand{\defref}[1]{Definition~\ref{#1}}

\def\ov{\overline}

\def\ti{\widetilde}

\def\cC{{\mathcal C}}

\def\cE{{\mathcal E}}

\def\cL{{\mathcal L}}

\def\cO{{\mathcal O}}

\def\cT{{\mathcal T}}
\def\cR{{\mathcal R}}
\def\cS{{\mathcal S}}
\def\cU{{\mathcal U}}

\def\cZ{{\mathcal Z}}

\def\crH{{\mathscr H}}

\def\crK{{\mathscr K}}

\def\crS{{\mathscr S}}

\def\bQ{{\mathbf Q}}

\def\1{{\mathbf 1}}

\def\tc{{\mathtt c}}

\def\tv{{\mathtt v}}

\def\ttf{{\mathtt f}}

\def\tti{{\mathtt{i}}}
\def\ttn{{\mathtt{n}}}
\def\tts{{\mathtt s}}
\def\ttt{{\mathtt t}}
\def\ttv{{\mathtt v}}
\def\ttM{{\mathtt M}}
\def\ttN{{\mathtt N}}
\def\ttP{{\mathtt P}}
\def\ttQ{{\mathtt Q}}

\def\B{\mathbb{B}}

\def\F{\mathbb{F}}
\def\H{\mathbb{H}}

\def\N{\mathbb{N}}
\def\Q{\mathbb{Q}}

\def\Z{\mathbb{Z}}

\def\Ker{{\rm Ker\,}}

\def\im{{\rm Im\,}}

\def\Hom{{\rm Hom}}
\def\HomD{{\rm Hom}^{\Delta}}
\def\op{{\rm{op}}}

\def\nat{{\rm{Nat}}}
\def\natr{{\mathrm{Nat}}_{R}}

\def\face{{\mathrm{Face}}}

\def\id{{\rm id}}

\def\sing{{\rm Sing}}

\def\singp{{\rm Sing}^{\ttP}}

\def\tN{{\widetilde{N}}}

\def\tG{{\widetilde{G}}}

\def\reg{{\rm reg}}
\def\ffs{{filtered face set}}
\def\ffss{{filtered face sets}}
\def\FFS{{{\mathrm{Ffs}}_{\ttP}}}

\def\ev{{\tt eval}}

\def\sq{{\rm Sq}}

\def\rc{{\mathring{\tc}}}

\def\tDelta{{\widetilde{\Delta}}}
\def\dDelta{{{\pmb\Delta}}}

\def\ttK{{\mathtt{K}}}

\def\emlp{{\text{\sc{EML}}}_{\ov{p}}}
\def\menos{\backslash}

\newcounter{ejemplo}

\newcounter{figura}

\def\set{{\mathrm{Set}}}
\def\top{{\mathrm{Top}}}
\def\sset{{\mathrm{Sset}}}
\def\ssetp{{\mathrm{Sset}}_{\ttP}}
\def\ssets{{\mathrm{Sset}}_{[1]}}
\def\topp{{\mathrm{Top}}_{\ttP}}

\def\mdg{\mathbf{M_{dg}}}

\def\tkpn{{\mathtt K}(R,n,\ttP,\ov{p})}

\def\tkqm{{\mathtt K}(R,m,\ttP,\ov{q})}
\def\tkqn{{\mathtt K}(R,n,\ttP,\ov{q})}
\def\tkpp{{\mathtt K}(R,n-1,\ttP,\ov{p})}

\begin{document}

\begin{abstract}
Eilenberg-MacLane spaces, that classify the singular cohomology groups of topological spaces,
admit natural constructions in the framework of simplicial sets.
The existence of  similar spaces for the intersection cohomology groups of a stratified space
is a long-standing open problem asked by M. Goresky and R. MacPherson. 
One feature of this work is a construction of such simplicial sets.
From works of R. MacPherson, J. Lurie and others, it is now commonly accepted that 
the simplicial set of singular simplices associated to a topological space has to be replaced by 
the simplicial set  of singular simplices that respect the stratification.  
This is encoded in the category of simplicial sets over the nerve of the poset of  strata. For each perversity, 
we define a functor from it, with values in 
the category of cochain complexes over a commutative ring.  
This construction is based upon a simplicial blow up and the associated
cohomology is the intersection cohomology as it was defined by M. Goresky and R. MacPherson.
This functor admits an adjoint and we use it to get classifying spaces for intersection cohomology.
Natural intersection cohomology operations are understood 
in terms of intersection cohomology of these classifying spaces.
As in the classical case, they form infinite loop spaces.
In the last section, we examine the depth one case of stratified spaces with only one singular stratum. 
We observe that the classifying spaces are
Joyal's projective cones over  classical Eilenberg-MacLane spaces.  We establish some of their properties  and
conjecture that, for Goresky and MacPherson perversities, all intersection cohomology operations are
induced by classical ones. 
\end{abstract}

\maketitle

\tableofcontents

\section*{Introduction}
M. Goresky and R. MacPherson introduced intersection homology which extend Poincar\'e duality 
from smooth manifolds to some singular spaces, the pseudomanifolds, 
admitting a decomposition into manifolds of different dimensions, called strata,
assembled so that each point has a conical neighbourhood.
Intersection homology relies on the notion of  perversity,  a parameter denoted $\ov{p}$
which measures the tangential degree of the component of chains along the strata.
For the intersection homology with rational coefficients, one gets the same picture as for topological manifolds 
\cite{MR440533, GM1}.
For instance, there exists a signature, which is a bordism invariant, when one applies the theory to Witt spaces
\cite{Sieg}. But when working over a commutative ring, subtle and important differences occur.

\medskip
Let us focus on Poincar\'e duality as an isomorphism between cohomology and homology
given by a cap product with the fundamental class. If the intersection cohomology is defined from a linear dual
of the intersection chain complex, we do not recover such Poincar\'e duality isomorphism for a general
commutative ring $R$ of coefficients,
without restriction on the torsion part of the intersection homology of the links of some singular strata, \cite{GS}.
We refer the reader to the monographs
\cite{MR2286904, Bor, Greg,MR2207421}
for a detailed account of these results and their applications.

\medskip
In previous works \cite{CST1,CST2}, we have introduced a cohomology obtained from a process of blow up
of singularities at the level of simplices. We call it blown up cohomology (or TW-cohomology) and denote it
$\crH_{\ov{p}}^*(-;R)$.
This cohomology coincides with the intersection cohomology obtained from the dual chain complex
if coefficients are in a field but differs in general.
One of its  main features is the existence of cup products of classes and  of cap products with intersection
homology classes.
In particular, the cap product with the fundamental class of a compact oriented pseudomanifold
gives a Poincar\'e isomorphism between the blown up cohomology and the intersection homology
\cite[Theorem B]{CST2}.
Versions for the non compact case also exist in \cite{CST2,ST1, ST2,Greg}.

\medskip
In their second main paper (\cite{GM2}), Goresky and MacPherson 
define  and characterize complexes of sheaves whose hypercohomology coincides with intersection homology.
The prototype  is named Deligne sheaf and denoted $\bQ_{\ov{p}}$. In \cite{CST5},
we prove that the sheafification of the blown up cochains is isomorphic to the Deligne sheaf 
in the derived category of complexes of sheaves of the space in consideration.
Thus the construction that we develop in this work also applies  to cohomological operations for the
hypercohomology associated to $\bQ_{\ov{p}}$.

\medskip
 \emph{One purpose of this work} is the definition of a perverse analog of Eilenberg-MacLane spaces,
and prove that their blown up cohomology is isomorphic to the set of intersection cohomological operations
(see \thmref{thm:main3} below), an exact duplicate of the topological situation.
In particular, our result answers the long standing question asked by Goresky and MacPherson as Problem 11 in \cite{Bor}:
\emph{
Is there a category of spaces, maps and
homotopies, and a ``classifying space'' $B$ so that intersection cohomology of $X$
can be interpreted as homotopy classes of maps from $X$ to
$B$?
}
Let us also mention that the existence of Steenrod squares in intersection cohomology was established by M. Goresky
(\cite{MR761809}) and adapt to the blown up cohomology in \cite{CST6}. (The main interest of \cite{CST6}
lies in the proof of a conjecture made in \cite{MR1014465}, see \conjref{conjecture2} below.)

\subsection*{Operations in singular cohomology}
Let us summarize the situation. If  $X$ is a topological space, we can  use  the simplicial set $\sing \,X$,
formed of the singular simplices, and move the problem into the simplicial paradigm. 
Thus, let $\sset$ be the category of simplicial sets and $\mdg$ be
 the category of cochain complexes of $R$-modules. 
 Cohomology of simplicial sets can be defined as the homology of the normalized cochain funtor,
 $$N^*(-;R)\colon \sset^\op\to \mdg.$$
 If we apply it to the simplicial set $\sing\, X$, we recover the singular cohomology of the 
 topological space $X$.
 
 \medskip
Let $K\in\sset$ and $M\in\mdg$. The functor $N^*$ admits an adjoint
 $\langle -\rangle$, defined by $\langle M\rangle_{k}=\Hom_{\mdg}(M,N^*(\Delta[k]))$. 
 Taking $M=R(n)$ with $R(n)^k=0$ if $k\neq n$ and $R(n)^n=R$, we obtain an Eilenberg-MacLane space
$K(R,n)=\langle R(n)\rangle$, \cite[Corollary III.2.7]{MR1711612}, giving an isomorphism between
homotopy classes and cohomology,
\begin{equation}\label{equa:introeml}
[K,K(R,n)]_{\sset}\cong H^n(K;R).
\end{equation}
 The category $\sset$ is  simplicially enriched and  
 we can define a simplicial set by,
 $$
 \HomD_{\sset}(K,\langle M\rangle)_{k}=\Hom_{\sset}(K\times \Delta[k],\langle M\rangle).
 $$
For $M=R(n)$, this simplicial set  is an abelian simplicial group and thus 
of the homotopy type of a  product of Eilenberg-MacLane space
(\cite[Th\'eor\`eme 6]{MooreInvariant}),
\begin{equation}\label{equa:introgem}
 \HomD_{\sset}(K,K(R,n))\simeq
\prod_{k=0}K(H^{n-k}(K;R),k),
\end{equation}
the determination \eqref{equa:introeml} corresponding to the image by $\pi_{0}$ of \eqref{equa:introgem}.
Mention also  that the family of Eilenberg-MacLane spaces 
$K(R,n)_{n}$
is an infinite loop space, the based loop space
$\Omega K(R,n)$
being homotopy equivalent to $K(R,n-1)$.

\medskip
By definition,  a  cohomological operation of type $(R,n,m)$
is a natural transformation
between the  functors
$H^n(-;R)$ and $H^m(-;R)$,
from $\sset$ to the category of $R$-modules.
We denote $\natr(H^n,H^m)$
the set of cohomological operations of this type. 
The representability theorem stated in  \eqref{equa:introeml} reveals crucial in the determination of  cohomology operations
since, as a direct consequence of Yoneda's lemma, there is an isomorphism
\begin{equation}\label{equaintroyonedaeml}
\natr(H^n,H^m)
\cong
H^m(K(R,n);R),
\end{equation}
between the set of operations and  the cohomology of Eilenberg-MacLane spaces.

\medskip
These previous notions constitute a well known material and most of them are in a talk of Jean-Pierre~Serre 
 \cite{Serrekpin}  in the  Cartan seminar.
However, they meet also recent progress and deep results in homotopy theory. The singular set
$\sing\, X$ enters in the framework of quasi-categories developed by A.~Joyal (\cite{joyalbarcelona, zbMATH05219541})
and J.~Lurie (\cite{LurieSecondBook,MR2522659}), where $\sing\, X$ has for 0-morphisms the points of $X$,
1-morphisms the paths in $X$, 2-morphisms the homotopy of paths,...
Its first truncation gives the classical Poincar\'e groupoid.
As we develop it below, our results contain the extension of this material to stratified spaces,
with a presentation taking in account  higher categorical structures.

\subsection*{Stratified spaces}
It is time to specify the objects of our study.
The prototype comes from the notion of pseudomanifolds (\defref{def:pseudomanifold}), 
that covers many notions of interest as (\cite{GM2})
real analytic varieties, Whitney stratified sets, Thom-Mather stratified spaces,...
A more general situation is a Hausdorff topological space together with a  partition  
$$X=\sqcup_{\tts\in \ttP}S_{\tts}$$ 
whose elements are  
non-empty, locally closed subsets of $X$, called strata.
If the partition is locally finite and any closure of a stratum is a union of strata,
then we say that $X$, with its partition, is \emph{a stratified space}, see \defref{def:decomposed}.
A crucial property of stratified spaces is the existence of a structure of poset on $\ttP$ for the relation
$\tts\preceq \ttt$ if $S_{\tts}\subset \ov{S}_{\ttt}$.
Endowing $\ttP$ with the associated Alexandrov topology, a structure of stratified space can be encoded
in a continuous map $X\to \ttP$, with some additional properties, see \defref{def:stratification}.

\medskip
To obtain a simplicial representation of a stratified space $X$, we consider the nerve  $\ttN(\ttP)$ of the poset $\ttP$
and define a simplicial set over $\ttN(\ttP)$, $\singp X$, as the pullback of
$$\sing\, X\to \sing\, \ttP \leftarrow \ttN(\ttP).$$
This definition coincides with those of filtered simplices in \cite[Definition A.3]{CST1} and \cite{MacPherson90}. 
In terms of quasi-categories, $\singp X$ meets an environment similar to that of $\sing\, X$. 
Remember that the first truncation of $\sing\, X$ is equivalent to the Poincar\'e groupoid.
Here, the first truncation of $\singp X$
is equivalent to the category of exit paths up to stratified homotopy, introduced in an unpublished work of MacPherson.
(An exit path is a path which is stratum-increasing.)
This similitude between $\sing\, X$ and $\singp X$  goes farther: the category of locally constant sheaves on a connected, locally contractible, topological space
is well-known to be equivalent to the category of $\pi_{1}X$-sets.
On a stratified space, MacPherson  proves that the  category of sheaves that are
locally constant on each stratum (also called constructible sheaves) is equivalent to the category 
of set valued functors on the category  of exit paths up to stratified homotopy.
 In \cite{MR2575092}, D. Treumann extends this result into a 2-categorical framework.
 Finally,  this has been  generalized by Lurie (\cite[Theorem A.9.3]{LurieSecondBook})  as an equivalence 
 between quasi-categories involving the quasi-category of simplicial set-valued functors defined on $\singp X$, 
 recovering the results of MacPherson and
 Treumann from the first  and the second truncations.
 Let us mention also the work (\cite{MR2591969}) of J. Woolf who refines \cite{MR2575092} for the
 homotopically stratified spaces of F. Quinn (\cite{Qui}).
 
 \medskip
 Thus,  moving the  notion of stratified space in  the simplicial paradigm by considering the simplicial map
 $\singp X\to \ttN(\ttP)$ defined above is an exact replica of the situation with $\sing\,X$. In view of Lurie's
 theorem, one can ask the following questions.
 
 \begin{question}
 Let $X\to \ttP$ be a stratified space, $\singp X\to \ttN(\ttP)$ the associated simplicial set over the
 nerve of $\ttP$ and $\ov{p}$ be a perversity of associated Deligne sheaf $\bQ_{\ov{p}}$.
 \begin{enumerate}[A)]
 \item Can we define the  intersection cohomology groups of simplicial sets over $\ttP$ so that
  $\H^*(X;\bQ_{\ov{p}})$  is the intersection cohomology of $\singp X$?
 \item Does there exist a simplicial set over $\ttP$, $\ttK(R,n,\ttP, {\ov{p}})$, so that the intersection
 cohomology is recovered as the homotopy classes of the simplicial set 
 $\HomD_{\ssetp}(-, \ttK(R,n,\bQ_{\ov{p}}))$?
 \item If (A) holds, does the intersection cohomology of $\ttK(R,n,\ttP, {\ov{p}})$ correspond to cohomological operations on
  intersection cohomology?
 \end{enumerate}
 \end{question}

\subsection*{The results}
Let $\topp$ be the category of topological spaces over a poset $\ttP$ and 
$\ssetp$ be the category of simplicial sets over $\ttN(\ttP)$.
The blown up cochain complex, $\tN^*_{\ov{p}}(-;R)\colon \topp\to \mdg$, already introduced
in \cite{CST4, CST3, CST5}, factorizes through $\ssetp$ as
\begin{equation}\label{equa:introboum}
\tN^*_{\ov{p}}(-;R)\colon
\xymatrix@1{
\topp\ar[r]^{\singp}
&
\ssetp\ar[r]
&
\mdg.
}
\end{equation}
By abuse of notation, we denote also $\tN^*_{\ov{p}}(-;R)\colon\ssetp\to\mdg$ 
the functor that appears in \eqref{equa:introboum} and by $\crH_{\ov{p}}^*(-;R)$ its homology. 
The first question is answered in \eqref{equa:deligneforever}
as follows.

\begin{theorem}\label{thm:main1}
Let $X\to \ttP$ be a pseudomanifold over the poset $\ttP$ and $\ov{p}$ be a perversity on $\ttP$.
Then, there are isomorphisms
$$\crH_{\ov{p}}^*(X;R)\cong
H^*(\tN^*_{\ov{p}}(\singp X);R)
\cong
\H^*(X;\bQ_{\ov{p}}).$$
\end{theorem}

For the introduction of perverse Eilenberg-MacLane spaces, we show the existence of  an adjunction
$$\xymatrix{\ssetp&& \mdg \ar@<1ex>[ll]^(.46){\langle-\rangle_{\ov{p}}}
\ar@<1ex>[ll];[]^(.50){\tN^*_{\ov{p}}   }\\}
$$
We take over the classical construction by setting $\tkpn =\langle R(n)\rangle_{\ov{p}}$.
The following statement (see \corref{cor:productEML}) is the answer to the second question.

\begin{theorem}\label{thm:main2}
Let $X\to \ttP$ be a pseudomanifold over the poset $\ttP$ and $\ov{p}$ be a perversity on $\ttP$.
Then, there is a homotopy equivalence
$$\HomD_{\ssetp}(\singp X, \tkpn)
\simeq
\prod_{k\geq 0} K(\H^{n-k}(X;\bQ_{\ov{p}}),k).$$
In particular, there are isomorphisms
$$\pi_{0}(\HomD_{\ssetp}(\singp X, \tkpn)
\cong
[\singp X,\tkpn]_{\ssetp}
\cong
\H^{n}(X;\bQ_{\ov{p}}).
$$
\end{theorem}

We also show in \thmref{thm:EMLinfinitloop} that the family $(K(R,n,\ttP,\ov{p}))_{n}$ 
is an infinite loop object in the category $\ssetp$. Finally, the behavior of cohomological operations on
the hypercohomology of Deligne's sheaves is deduced from  \propref{prop:emloperation} as follows. 

\begin{theorem}\label{thm:main3}
Let  $\ov{p}$ and $\ov{q}$ be perversities on a poset $\ttP$.
For any couple of integers $(n,m)$, there is an isomorphism
$$\natr(\crH^{n}_{\ov{p}},\crH^{m}_{\ov{q}})\cong
\crH^m_{\ov{q}}(\tkpn;R).
$$
\end{theorem}

Thus, an important task is the computation of intersection cohomology of the perverse 
\break
Eilenberg-MacLane  spaces.
For that, we begin with the perversities introduced by Goresky and MacPherson in 
\cite{GM1,GM2}, that we call here GM-perversities. They have the particularity of depending only on the codimension
of the strata and we may choose subspaces of $\N$ as posets. We investigate  the simplest case
of isolated singularities  for which we can choose $\ttP=[1]=\{0,1\}$.
We notice that the corresponding perverse Eilenberg-MacLane  spaces are Joyal's cylinders in the sense
of \cite[Section~7]{joyalbarcelona}. More precisely, with the terminology of \cite{joyalbarcelona}, 
they are projective cone over classical Eilenberg-MacLane spaces. 
For instance, we can consider the two constant perversities, $\ov{0}$ and $\ov{\infty}$, with value 0 and $\infty$
respectively. For them, we get:
\begin{itemize}
\item $\ttK(R,n,[1],\ov{\infty})=\Delta[0]\ast K(R,n)$,
\item $\ttK(R,n,[1],\ov{0})=\Delta[1]\times K(R,n)/\Delta[0]\times K(R,n).$
\end{itemize}

Our actual knowledge of perverse Eilenberg-MacLane  spaces and their comparison with
the classical ones leads us to the following conjecture. 
Recall first the existence (\cite[Proposition 3.1]{CST7} of a 
natural chain isomorphism, $N^*(\Delta)\to \tN^*_{\ov{0}}(\Delta)$, which induces a natural chain injection
$N^*(\Delta)\to \tN^*_{\ov{p}}(\Delta)$ for any positive perversity $\ov{p}$.

\begin{conjecture}\label{conjecture1}
 Let  $\ttP=[n]$ and $\ov{p}$, $\ov{q}$ be two perversities of Goresky and MacPherson. Then,  all perverse cohomological operations come from the classical cohomology situation; i.e., the previous natural chain injection induces an 
 injective map
 $$\crH^m_{\ov{q}}(\tkpn;R)\longrightarrow H^m(K(R,n);R).$$
 \end{conjecture}

For the poset $\ttP=[1]$, we prove this conjecture in low degrees, for any ring $R$ and any perversities
$\ov{p}$, $\ov{q}$ when $m\leq n$, cf. Propositions~\ref{prop:kskeletonofcone}, \ref{prop:kskeleton}, \ref{prop:emlk}
and \thmref{thm:conjen0}.
If \conjref{conjecture1} is true, we can state a more precise conjecture, based on
computations in the rational case in  \cite{CST1}  and from
\cite{MR761809,MR1014465,CST6} for $\F_{2}$. 

\begin{conjecture}\label{conjecture2}
When $\ttP=[n]$, we have the following isomorphisms of perverse algebras.
\begin{itemize}
\item $\crH^\ast_{\ov{\bullet}}(K(\Q,m,[n],\ov{p});\Q)\cong \land_{\ov{\bullet}} \,x$, 
where $\land_{\ov{\bullet}} \,x$ is the free  rational commutative graded perverse algebra over one generator $x$
of differential degree $m$ and perverse degree $\ov{p}$.
\item $\crH^\ast_{\ov{\bullet}}(K(\F_{2},m,[n],\ov{p});\F_{2})\cong 
\crK_{\ov{\bullet}}^\ast (x)$, where
$\crK_{\ov{\bullet}}^\ast (x)$ is the free unstable perverse algebra over one generator $x$
of differential degree $m$ and perverse degree $\ov{p}$.
\end{itemize}
\end{conjecture}

The perverse algebra $\crK_{\ov{\bullet}}^\ast (x)$ is a polynomial algebra $\F_{2}[\{\sq^I x\}_{I\in \mathrm{Ad}}]$
generated by admissible sequences of Steenrod squares
$$\sq^I x=\sq^{i_{1}}\dots \sq^{i_{k}} x$$
where $i_{j}\geq 2 i_{j+1}$ and
$\sum_{j}i_{j}-2i_{j+1}<m$. The perverse degree of $\sq^I$ is computed from the following formula
(conjectured in \cite{MR1014465} and settled in \cite{CST6}):
``If $a$ is of perverse degree $\ov{p}$ then $\sq^j a$ is of perverse degree
$\min(2\ov{p},\ov{p}+j)$.''

\begin{perspective}
From our  simplicial constructions, we can introduce generalized intersection cohomology theories.
Taking an infinite loop space, $(S(n))_{n}$ in $\ssetp$, we  can define, for any stratified topological space
$X\to\ttP$, an abelian group by
$$\crS^n(X)=\pi_{0}\HomD_{\ssetp}(\singp X,S(n)).
$$
Given a perversity $\ov{p}\colon \ttP\to \ov{\Z}$ and an infinite loop space $(L(n))_{n}$ in $\sset$, can we find
an infinite loop space $(L(n,\ov{p}))_{n}$ in $\ssetp$ which brings a generalized intersection cohomology theory?
For this last part, a good understanding of the situation developed in \secref{sec:operations} for  $\ttP=[1]$
 will be a first step.
 We also plan to treate the question of topological invariance of intersection cohomology within the framework
 developed here in future work.
\end{perspective}

\subsection*{Outline}

In \secref{sec:topstrat}, we present the relations between stratified spaces and the category $\ssetp$ of simplicial sets 
over the nerve of a poset, $\ttN(\ttP)$. The structure of simplicial category on $\ssetp$ is detailed in 
\secref{sec:spaceposet}; we also introduce the simplicial category $\ssetp^+$ of restricted  simplicial sets over $\ttN(\ttP)$
and build an adjunction between these two categories. The restricted  simplicial sets over $\ttN(\ttP)$ are  crucial
objects in the blown up process that we introduce in \secref{sec:boum}. From this construction, we define an adjunction
between $\ssetp$ and the category of cochain complexes over a commutative ring in
\secref{sec:functors}. We extend this adjunction to homotopy classes of morphisms to prepare the
construction of Eilenberg-MacLane spaces done in
\secref{sec:perversEML}, where we prove the results stated in Theorems~\ref{thm:main1}, \ref{thm:main2} and \ref{thm:main3}. Finally in \secref{sec:operations}, we analyze the case $\ttP=[1]$, as stated before.

\medskip
Two appendices complete this work. In \cite{CST1}, we  defined a blown up cohomology 
for filtered face sets. 
These latter are simplicial sets over the poset $\N$ without degeneracies,
a role similar to that of $\Delta$-sets \cite{MR0300281} for simplicial sets. 
In  \appendref{sec:ffsareback}, we introduce the category $\FFS$ of filtered face sets over a poset $\ttP$
and  show that the concepts
introduced in the present work are compatible with that of \cite{CST1}; 
this allows the use of results of \cite{CST1} in \secref{sec:operations}.
Finally,
\appendref{sec:simplicialcat} is a brief  reminder on homotopy and loop spaces in simplicial categories, 

\medskip
As  a guide for the reader, we summarize in the following diagram the connections used between 
the category $\ssetp$ of simplicial sets over $\ttN(\ttP)$ and its surroundings.
$$  
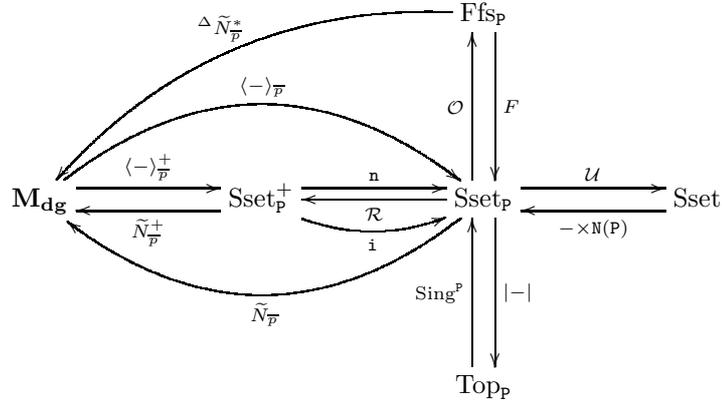
\xymatrix{
 &&&&
 \FFS
 \ar@/_2pc/[ddllll]_-{^\Delta \tN^*_{\ov{p}}}
 \ar@<1ex>[dd]^-{F}
  &&\\
 &&&&&&\\
\mdg
\ar@<1ex>[rr]^-{\langle -\rangle^+_{\ov{p}}}
 \ar@/^3pc/[rrrr]^-{\langle -\rangle_{\ov{p}}}
&& 
\ssetp^+\ar@<1ex>[rr]^-{\ttn}
\ar@<1ex>[ll]^-{\tN_{\ov{p}}^{+}}
\ar@/_1pc/[rr]_{\tti}
&&
\ssetp
\ar@/^3pc/[llll]^-{\tN_{\ov{p}}} 
\ar[ll]^-{\cR}\ar@<1ex>[rr]^-{\cU}
\ar@<1ex>[dd]^-{|-|}
\ar@<1ex>[uu]^{\cO}
&&
\sset\ar@<1ex>[ll]^-{-\times\ttN(\ttP)}
\\
&&&&&&\\
&&&&\topp
\ar@<1ex>[uu]^{\singp}
&&
}
$$

\medskip
The authors would like to thank Martintxo Saralegi-Aranguren
for insightful advice on this work.
%

\section{Stratified topological spaces}\label{sec:topstrat}

\begin{quote}
Stratified topological spaces and maps offer a geometric setting for the definition of intersection cohomology
and the existence of morphisms between them, having regard to the level of cochain complexes. 
After a brief reminder, we present them as topological spaces
 over the  poset of their strata, endowed with the Alexandrov topology.
We also recall the notion of filtered simplices which prepares the study of intersection cohomology 
from simplicial objects.
\end{quote}

\subsection{Stratified spaces and maps}\label{subsec:stratifiedspacesdef}

Let us introduce the stratified spaces, corresponding to the $\cS$-decomposition of 
\cite[\& I.1.1]{MR932724}.

\begin{definition}\label{def:decomposed}
A \emph{stratified space} is a Hausdorff topological space endowed with a  partition  
$$X=\sqcup_{\tts\in \ttP}S_{\tts}$$ 
whose elements are  
non-empty, locally closed subsets of $X$, called strata, and  satisfying the following properties:
\begin{enumerate}[(i)] 
\item\emph{the Frontier condition:}
{for any pair of strata, $S$ and $S'$ with
 $S\cap \overline{S'}\neq \emptyset$,  one has
$S\subset \ov{S'}$,}
\item for any subset $J\subset\ttP$, one has
$\cup_{\tts\in J} \ov{S_{\tts}}=
\ov{\cup_{\tts\in J} S_{\tts}}$.
\end{enumerate}
\noindent
A  stratum is \emph{regular} if it is an open subset of $X$.
A stratified space is said \emph{regular} if it owns regular strata.
\end{definition}

Property (ii) is satisfied if the family of strata $(S_{\tts})_{\tts\in\ttP}$
is locally finite (as in \cite[\& I.1.1.]{MR932724}) and, a fortiori, if $\ttP$ is finite.
By definition, a subset $S$ is  locally closed  if $S=U\cap C$ with $U$ open and $C$ closed in $X$,
or, equivalently, if $S=U\cap \ov{S}$.
Recall also that the frontier condition is equivalent to
\begin{equation}\label{equa:frontier}
\ov{S}=\sqcup_{S'\cap\ov{S}\neq \emptyset}S'.
\end{equation}

\begin{proposition}\label{prop:stratifieetordre}
Let $X=\sqcup_{\tts\in \ttP}S_{\tts}$ be a stratified space.
Then the set $\ttP$ is a poset for the relation
$\tts \preceq \ttt$ if $S_{\tts}\subseteq \ov{S_{\ttt}}$. 
(We write $\tts\prec \ttt$ if  $\tts\preceq \ttt$ and $\tts\neq \ttt$.)
\end{proposition}

\begin{proof}
Let $S_{\tts}$ and $S_{\ttt}$ be two strata of $X$ such that $\tts\preceq \ttt$ and $\ttt\preceq \tts$, we have to prove ${\tts}=\ttt$.
Let $x\in S_{\tts}$. Since $S_{\tts}$ is  locally closed, there exists an open subset $U$ of $X$, such that
$S_{\tts}=U\cap \ov{S_{\tts}}$.
As $x\in S_{\tts}\subset \ov{S_{\ttt}}$, we have $U\cap S_{\ttt}\neq \emptyset$.
On the other hand, from $S_{\ttt}\subset \ov{S_{\tts}}$, we get
$S_{\tts}=U\cap \ov{S_{\tts}}\supset U\cap S_{\ttt}\neq \emptyset$,
which implies $S_{\tts}\cap S_{\ttt}\neq \emptyset$
and $S_{\tts}= S_{\ttt}$ since they are members of a partition.
\end{proof}

\begin{remark}\label{rem:alexandrov}
Let $(\ttP,\preceq)$ be a  poset. We endow $\ttP$ with the \emph{Alexandrov topology.} 
The open sets are the subsets $U$ such that
if $\tts\in U$ and $\tts\preceq \ttt$ then $\ttt\in U$.
The corresponding closed sets are the subsets $F$ such that
if $\tts\in F$ and $\ttt\preceq \tts$ then $\ttt\in F$.
Therefore, any singleton $\{\tts\}$ is locally closed as the intersection
$]-\infty, \tts]\cap [\tts,\infty[$.
Recall also that, in $\ttP$, any union of closed subsets is closed.
\end{remark}

If $X$ is a stratified space, we define a surjective map $\psi_{X}\colon X\to \ttP$ by sending  a point
$x\in S_{\tts}$ to $\tts\in\ttP$. In particular, $\psi_{X}$ sends a regular stratum on a maximal element of
the poset.

\medskip
We encode now the requirements of \defref{def:decomposed} as properties of the map $\psi_{X}$,
see  \cite{LurieSecondBook} or \cite{2018arXiv180411274T} for a similar approach.

\begin{definition}\label{def:stratification}
\emph{A stratification of a topological space} $X$ by a poset $\ttP$ is an open continuous surjective map 
$\psi_{X}\colon X\to\ttP$, where $\ttP$ is equipped with the Alexandrov topology.
\end{definition}

\begin{proposition}\label{prop:stratandstrat}
A Hausdorff topological space is stratified if, and only if, it admits a stratification.
\end{proposition}

In the proof, we use the following characterization of open maps.

\begin{lemma}{\cite[Lemma 3.3]{2018arXiv180411274T}}\label{lem:open}
A map $f\colon X\to Y$ between topological spaces is open if, and only if, 
$f^{-1}(\ov{B})\subset \ov{f^{-1}(B)}$, for any $B\subset Y$.
In particular, $f$ is open and continuous if, and only if,
$f^{-1}(\ov{B})= \ov{f^{-1}(B)}$, for any $B\subset Y$.
\end{lemma}

\begin{proof}[Proof of \propref{prop:stratandstrat}]
Suppose first that $X=\sqcup_{\tts\in \ttP}S_{\tts}$ is a stratified space. 
The frontier condition of \defref{def:decomposed}   implies that $\ttP$ is a poset and that
$\psi_{X}^{-1}(\ov{\tts})=\psi_{X}^{-1}(]-\infty,\tts])=\ov{\psi_{X}^{-1}(\tts)}$.
 Let $B=\cup_{j\in J}\{\tts_{j}\}\subset \ttP$.
We have
\begin{eqnarray*}
\psi_{X}^{-1}(\ov{B})
&=&
\psi_{X}^{-1}(\ov{\cup_{j\in J}\{\tts_{j}\}})
=
\psi_{X}^{-1}(\cup_{j\in J}\ov{\tts_{j}})
=
\cup_{j\in J}\psi_{X}^{-1}(\ov{\tts_{j}})
=
\cup_{j\in J}\ov{\psi_{X}^{-1}(\tts_{j})}\\
&=&
\ov{\cup_{j\in J}\psi_{X}^{-1}(\tts_{j})}
=
\ov{\psi_{X}^{-1}(B)}.
\end{eqnarray*}
Thus, with \lemref{lem:open}, $\psi_{X}$ is stratification.

\medskip
Reciprocally, suppose that $\psi_{X}\colon X\to\ttP$ is a stratification.
We obtain a decomposition
$X=\sqcup_{\tts\in\ttP}\psi_{X}^{-1}(\tts)$,
with $\psi_{X}^{-1}(\tts)\neq \emptyset$ and locally closed.
Since $\psi_{X}$ is open and continuous, we have
$\psi_{X}^{-1}(\ov{\tts})=\ov{\psi_{X}^{-1}(\tts)}$ for any $\tts\in \ttP$. Thus,
$\psi_{X}^{-1}(\tts)\cap \ov{\psi_{X}^{-1}(\ttt)}=\psi_{X}^{-1}(\tts)\cap \psi_{X}^{-1}(\ov{\ttt})=
\psi_{X}^{-1}(\tts)\cap \psi_{X}^{-1}(]-\infty,\ttt)=
\psi_{X}^{-1}(\tts \cap ]-\infty, \ttt])$.
With this equality, from
$\psi_{X}^{-1}(\tts)\cap \ov{\psi_{X}^{-1}(\ttt)}\neq\emptyset$, we deduce $\tts\preceq\ttt$, which is the Frontier condition.
Let $J\subset \ttP$. Using \lemref{lem:open}, we have
$$
\cup_{j\in J}\ov{\psi_{X}^{-1}(\tts_{j})}
=
\cup_{j\in J}\psi_{X}^{-1}(\ov{\tts_{j}})
=
\psi_{X}^{-1}(\cup_{j\in J}\ov{\tts_{j}})
=
\psi_{X}^{-1}(\ov{\cup_{j\in J}\{\tts_{j}\}})
=
\ov{\cup_{j\in J}\psi_{X}^{-1}(\tts_{j})}.
$$
\end{proof}

  We introduce now the morphisms between stratified spaces.

\begin{definition}\label{def:stratifiedmap}
A \emph{stratified map,} $f\colon X=\sqcup_{s\in\ttP_{X}} S_{s}\to Y=\sqcup_{t\in\ttP_{Y}}T_{\ttt}$, is a continuous map between stratified spaces
such that,
for each stratum  $S_{\tts}$ of $X$,
there exists a unique stratum $T_{\ttt_{\tts}}$ of $Y$ with $f(S_{\tts})\subset  T_{\ttt_{\tts}}$.
We denote $\ttf\colon \ttP_{X}\to \ttP_{Y}$, the  map $\ttf(\tts)=\ttt_{\tts}$.
\end{definition}
  
\begin{proposition}\label{prop:stratifiedmap}
If $f\colon X\to Y$ is a stratified map,  the map,
$\ttf\colon \ttP_{X}\to\ttP_{Y}$ is increasing. 
\end{proposition}

\begin{proof}
Let $\tts_1 \preceq \tts_2$ in $\ttP_{X}$. From the  continuity of $f$, 
we deduce $f(S_{\tts_{1}}) \subset \ov{f(S_{\tts_{2}})}$ 
and 
$T_{\ttf(\tts_{1})}\cap \ov{T_{\ttf(\tts_{2})}}\neq \emptyset$.
The frontier condition implies $\ttf(\tts_{1})\preceq \ttf(\tts_{2})$.
\end{proof}

The next result is a direct consequence of Propositions~\ref{prop:stratandstrat} and \ref{prop:stratifiedmap}.

\begin{corollary}\label{cor:mapstratification}
A continuous map, $f\colon X=\sqcup_{s\in\ttP_{X}} S_{s}\to Y=\sqcup_{t\in\ttP_{Y}}T_{\ttt}$, 
between stratified spaces is a stratified map if, and only if, there exists a commutative diagram
of continuous maps between the associated stratifications,
$$\xymatrix{
X\ar[r]^-{f}\ar[d]_{\psi_{X}}&
Y\ar[d]^{\psi_{Y}}\\
\ttP_{X}\ar[r]^-{{\ttf}}&\ttP_{Y}.
}$$
\end{corollary}

\subsection{Filtered simplices}\label{subsec:fileterd}

Let us introduce the filtered simplices to prepare the simplicial approach developed in \secref{sec:spaceposet}.

\medskip
Let $\dDelta$ be the simplicial category whose objects are the nonnegative integers $[n]=\{0,\dots,n\}$
and whose morphisms are the order preserving maps. 
Among them, we quote  the  cofaces and codegeneracies,
$d^i\colon [n-1]\to [n]$
and
$s^i\colon [n+1]\to [n]$,
for $0\leq i\leq n$.
The \emph{standard simplicial $n$-simplex}  is defined by
$\Delta[n]=\Hom_{\dDelta}(-,[n])\colon \dDelta^{\op} \to \set$.
Its realization is the geometric $n$-simplex, $|\Delta[n]|=\Delta^n$.

\medskip
A \emph{simplicial set} is a functor
$K\colon \dDelta^{\op}\to \set$,
where $\set$ is the category of sets. 
The elements of $K_{n}=K([n])$ are called $n$-simplices of $K$.
The image, by the functor $K$, of the cofaces and the codegeneracies are, respectively, 
the faces  $d_{i}\colon K_{n}\to K_{n-1}$ 
and 
the degeneracies $s_{i}\colon K_{n}\to K_{n+1}$,
for $0\leq i\leq n$.
A \emph{simplicial map} is a natural transformation between simplicial sets
and  $\sset$ is the associated category.

\medskip
Filtered simplices are the crucial notion for the existence of a simplicial blow up associated to a stratified space.
We already have expressed them in the case of filtered spaces in \cite{CST4}, \cite{CST1}, \cite{CST3}.

\begin{definition}\label{def:filteredsimplex}
Let $\psi_{X}\colon X= \sqcup_{s\in\ttP} S_{s} \to \ttP$ be a stratification. 
 A  \emph{filtered simplex (over $\ttP$)} of $X$ is a continuous map,
$\sigma\colon \Delta^m=[e_{0},\dots,e_{m}]\to X$, such that, for any $\tts\in \ttP$,
 $\sigma^{-1}S_{\tts}$ is the empty set, or a face  $A=[e_{0},\dots,e_{a}]$ or a difference of two faces 
 $B\menos A$ with $B=[e_{0},\dots,e_{b}]$ and $a<b$. 
We denote
$\singp \psi_{X}$ the simplicial set of the filtered simplices of the stratification $\psi_{X}$.
\end{definition}

The notion of nerve of a poset allows a simplicial presentation of filtered simplices. Recall 
that the \emph{nerve of a poset $\ttP$}  is the simplicial set, $\ttN(\ttP)$, whose n-simplices are
the chains  of increasing elements of $\ttP$, $\tts_{0}\preceq\dots\preceq \tts_{n}$.

\begin{proposition}\label{prop:propfiltered}
Let $\psi_{X}\colon X= \sqcup_{s\in\ttP} S_{s} \to \ttP$ be a stratification and
$\sigma\colon \Delta^m=[e_{0},\dots,e_{m}]\to X$
be a filtered simplex. Then, 
the association $e_{i}\mapsto \psi_{X}(\sigma(e_{i}))$ defines a simplicial map
$$\Psi_{X}\colon \singp\psi_{X}\to \ttN(\ttP).$$
\end{proposition}

\begin{proof}
First, we prove that $\psi_{X}({\sigma}(e_{i}))=\tts$ and $\psi_{X}({\sigma}(e_{i+1}))=\ttt$
imply  $\tts\preceq \ttt$.
We may suppose $\tts\neq \ttt$. From $\sigma(e_{i})\in S_{\tts}$ and $\sigma(]e_{i},e_{i+1}])\subset S_{\ttt}$, we deduce
$\sigma(e_{i})\in \ov{S_{\ttt}}$. We therefore  have
$S_{\tts}\cap \ov{S_{\ttt}}\neq\emptyset$.
The Frontier condition implies
$S_{\tts}\subset \ov{S_{\ttt}}$ and thus $\tts\prec \ttt$.  
By grouping the identical strata, an
iteration of this process along the vertices of $\Delta^m=[e_{0},\dots,e_{m}]$ gives a decomposition,
\begin{equation}\label{equa:deltajoint}
\Delta^m=[e_{0},\dots,e_{q_{0}},\mid e_{q_{0}+1},\dots,e_{q_{0}+q_{1}+1},\mid\dots\mid e_{m-q_{\ell}},\dots,e_{m}],
\end{equation}
such that 
$\sigma(e_{\alpha})\in S_{i}$ for $\alpha=\sum_{k< i}q_{k}+i+j$
with $j\in\{0,\dots,q_{i}\}$, $q_{-1}=0$ and
$\tts_{i}=\psi_{X}(S_{i})\prec \tts_{i+1}=\psi_{X}(S_{i+1})$ for all $i\in\{0,\ldots,\ell\}$.
Thus,  $\psi_{X}\circ {\sigma}$ is a simplex of $\ttN(\ttP)$, that we write
\begin{equation}\label{equa:decomposition}
\psi_{X}\circ{\sigma}=
\begin{array}{ccc}
\underbrace{\tts_{0}\preceq\dots\preceq \tts_{0}}&\prec\dots\prec &
\underbrace{\tts_{\ell}\preceq\dots\preceq \tts_{\ell}}\\
q_{0}+1& &q_{\ell}+1
\end{array}
=\tts_{0}^{[q_{0}]}\prec\dots\prec \tts_{\ell}^{[q_{\ell}]}.
\end{equation}
The compatibility with faces and degeneracies
is immediate and  $\Psi_{X}$ is a simplicial map. 
\end{proof}

\begin{remark}
From the decomposition
$\Delta^m=\Delta^{q_{0}}\ast\dots\ast\Delta^{q_{\ell}}$, 
established in \eqref{equa:deltajoint}, we observe  that, for any $i\in\{0,\dots,\ell\}$,
\begin{equation}\label{equa:petitestrate}
\sigma^{-1}(S_{0}\sqcup\dots\sqcup S_{i})=\Delta^{q_{0}}\ast\dots\ast\Delta^{q_{i}}.
\end{equation}
With this filtration, a filtered simplex, $\sigma\colon \Delta\to X$, is a stratified map. Thus,
stratified maps preserve  filtered simplices.
\end{remark}

\begin{proposition}\label{prop:stratifdegree}
Let 
$f\colon \psi_{X}\to \psi_{Y}$
be a stratified map.  
Then
there is a commutative diagram in $\sset$,
$$\xymatrix{
\sing^{\ttP_{X}}\psi_{X}\ar[rr]^-{\sing(f)}\ar[d]_{\Psi_{X}}&&
\sing^{\ttP_{Y}}\psi_{Y}\ar[d]^{\Psi_{Y}}\\
\ttN(\ttP_{X})\ar[rr]^-{\ttN(\ttf)}&&\ttN(\ttP_{Y}),
}$$
where $\Psi_{X}$, $\Psi_{Y}$ are defined in \propref{prop:propfiltered} and $\ttf$
in \defref{def:stratifiedmap}. 
\end{proposition}

\section{Simplicial sets  over a poset}\label{sec:spaceposet}

\begin{quote} 
We introduce the  category $\ssetp$ of simplicial sets over a poset $\ttP$. 
We endow  it with a structure of simplicial category and connect it to  classical simplicial sets
and to topological spaces over a poset.
In the construction of a simplicial blow up, the notion of regular simplices turns out to be crucial. So, we introduce
the category of restricted  simplicial sets over a poset, $\ssetp^+$, and describe three functors between
$\ssetp$ and $\ssetp^+$.
\end{quote}


Let $\ttP$ be a \emph{fixed} poset, of associated nerve $\ttN(\ttP)$.
Let  $\Delta[\ttP]$ be the category of simplices of $\ttN(\ttP)$, whose
objects are the simplicial maps $\Delta[k]\to \ttN(\ttP)$,
and morphisms  the commutative triangles of simplicial maps,
\begin{equation}\label{equa:morphisms}
\xymatrix{
\Delta[k]\ar[rr]\ar[dr]&&\Delta[\ell]\ar[ld]\\
&\ttN(\ttP).&
}\end{equation}

\begin{definition}\label{def:overposet}
A \emph{simplicial set over $\ttP$} is a presheaf on the category $\Delta[\ttP]$.
We denote $\ssetp$ the category of natural transformations between simplicial
sets over $\ttP$. 
\end{definition}

If $\Phi_{K}\colon (\Delta[\ttP])^{\op}\to \set$ is a presheaf, we define a simplicial set
$$K=\left\{(\sigma,x)\mid \sigma\in\ttN(\ttP) \text{ and } x\in \Phi_{K}(\sigma)\right\}.$$
The projection $(\sigma,x)\mapsto \sigma$ is a simplicial map 
$\Psi_{K}\colon K\to \ttN(\ttP)$. 
Conversely, if $\Psi_{K}\colon K\to \ttN(\ttP)$ is a simplicial map, we define a presheaf
$\Phi_{K}\colon K\to \ttN(\ttP)$ as follows: to any $\sigma\in\ttN(\ttP)$, we associate the set
$\Phi_{K}(\sigma)$ of the lifting simplicial maps
$$\xymatrix{
&K\ar[d]^-{\Psi_{K}}\\
\Delta[n]\ar[r]^-{\sigma}\ar@{-->}[ru]&
\ttN(\ttP).
}$$
Therefore, an object of $\ssetp$ can  be seen as a simplicial map
$\Psi_{K}\colon K\to \ttN(\ttP)$. 
With this point of view,  a morphism of $\ssetp$ is a commutative triangle
$$\xymatrix{
K\ar[rr]^-{f}\ar[dr]_{\Psi_{K}}&&
L\ar[dl]^{\Psi_{L}}\\
&\ttN(\ttP).&
}$$

\begin{remark}\label{rem:overposets}
The main result of this work is a presentation of perverse cohomological operations 
by the use of the representability of intersection cohomology. 
This can be achieved by considering simplicial sets over one fixed
poset. However, the category of simplicial sets over posets can easily be defined as it occurs in
\propref{prop:stratifdegree}.

Let $\ttP$ and $\ttQ$ be posets of associated nerves $\ttN(\ttP)$ and $\ttN(\ttQ)$. 
 A morphism $(f,\ttf)$ between two simplicial sets over posets, $(\Psi_{K},\ttP)$ and $(\Psi_{L},\ttQ)$, 
is a commutative diagram of simplicial maps,
$$\xymatrix{
K\ar[r]^-{f}\ar[d]_{\Psi_{K}}&
L\ar[d]^{\Psi_{L}}\\
\ttN(\ttP)\ar[r]^-{\ttf}&\ttN(\ttQ).
}$$
In \propref{prop:mapinterseccohomology}, we place in this more general context
the existence of  homomorphisms between intersection cohomology groups, induced by simplicial maps.
\end{remark}

Appendix~B contains basic recalls on simplicial categories, sometimes also called simplicially enriched category.
The category $\sset$ is a closed simplicial model category (\cite{MR0223432}) with the following classes of maps:
the cofibrations are the monomorphisms, the fibrations are the Kan fibrations and the
weak-equivalences are the simplicial maps whose realization is a weak equivalence in the category of topological
spaces. The simplicial structure comes from $K\otimes L= K\times L$ and 
$\left(\HomD_{\sset}(K,L)\right)_{n}=\Hom_{\sset}(K\times \Delta[n],L)$.

\begin{remark}\label{rem:ssetpsimplicial}
\emph{ The category $\ssetp$ inherits a structure of simplicial category} with
 $\Psi_{K}\otimes L=K\times L\xrightarrow{\rm proj} K\xrightarrow{\Psi_{K}} \ttN(\ttP)$,
 for $\Psi_{K}\in \ssetp$ and $L\in\sset$.
 If $\Psi_{K_{1}}, \,\Psi_{K_{2}}\in \ssetp$, the simplicial set $\HomD_{\ssetp}(\Psi_{K_{1}},\Psi_{K_{2}})$ is the simplicial subset of $\HomD_{\sset}(K_{1},K_{2})$
 formed of simplicial maps which commute over $\ttN(\ttP)$.
 Following \defref{def:sfibrant}, an object $\Psi_{K}$ of $\ssetp$ is s-fibrant if the simplicial set
 $\HomD_{\ssetp}(\Psi_{L},\Psi_{K})$ is Kan for any $\Psi_{L}\in \ssetp$.
Moreover a simplicial map over $\ttP$, $f\colon \Psi_{K_{1}}\to \Psi_{K_{2}}$, is a weak-equivalence, if
 it induces an isomorphism 
 $\pi_{0}\HomD_{\ssetp}(\Psi_{L},\Psi_{K_{1}})\xrightarrow{\cong}\pi_{0} \HomD_{\ssetp}(\Psi_{L},\Psi_{K_{2}})$
 for any $\Psi_{L}\in \ssetp$.
 \end{remark}

\subsection{Connection with  topological spaces over a poset}\label{sunsec:realization}
Let $\ttP$ be \emph{a fixed poset.} We consider a slight more general situation than that of stratified space,
called $\ttP$-stratification in \cite[Definition~A.5.1]{LurieSecondBook}.

\begin{definition}\label{def:topoverP}
A \emph{topological space over $\ttP$} is a map,
$\psi_{X}\colon X\to \ttP$, which is continuous for the Alexandrov topology on $\ttP$.
We denote $\topp$ the category of topological spaces over $\ttP$ with morphisms, 
$f\in\Hom_{\topp}( \psi_{X}, \psi_{Y})$, the continuous maps $f\colon X\to Y$
 such that $\psi_{Y}\circ f=\psi_{X}$.
\end{definition}

Let $\psi_{X}\colon X\to \ttP$ be an object of $\topp$. We define the
simplicial set $\singp\psi_{X}$ as the following pullback in $\sset$,
 \begin{equation}\label{equa:filteredpullback}
 \singp\psi_{X}=\ttN(\ttP)\times_{\sing\;\ttP}\sing \,X,
 \end{equation}
whose elements  are pairs $(\psi_{X}\circ{\sigma}, \sigma)$  with
 $\sigma\in\sing\,X$ 
 and $\psi_{X}\circ\sigma \in\ttN(\ttP)$.
This construction coincides with the simplicial set of filtered simplices, introduced in \defref{def:filteredsimplex}.
 
 \medskip
 We extend the classical adjunction between the categories $\top$ and $\sset$ as:
 $$\xymatrix{
\ssetp
\ar@<1ex>[rr]^-{|-|}
&&\topp.
\ar@<1ex>[ll]^-{\singp}
}$$
The functor $\singp$ is defined in \eqref{equa:filteredpullback}.
 The realization functor  
$|-|$  sends  the object
$\Psi_{K}\colon K\to \ttN(\ttP)$ of $\ssetp$ to the composite
$$\psi_{|K|}\colon |K|\xrightarrow{|\Psi_{K}|} |\ttN(\ttP)|\xrightarrow{\chi_{\ttP}}\ttP\in\topp.$$
Recall that the map $\chi_{\ttP}$ is a natural weak homotopy equivalence called the last vertex map,
see \cite{MR0196744}. 
For instance, if $\ttP=[n]$, $\chi_{\ttP}$ associates to
$(t_{0},\dots,t_{n})\in \Delta^n=|\ttN([n])|$ the greatest index $i\in[n]$ such that $t_{i}\neq 0$.

\begin{remark}\label{rem:toppsimplicial}
The category $\top$ endows a structure of closed simplicial model category with
 $X\otimes K= X\times |K|$ and 
$\left(\HomD_{\top}(X,Y)\right)_{n}=\Hom_{\top}(X\times \Delta^n,Y)$.
With constructions similar to those of \remref{rem:ssetpsimplicial}, one can equip the category $\topp$ with a structure of
simplicial category.
\end{remark}

The adjunction between simplicial sets and topological spaces over $\ttP$ can be enriched to a
simplicial adjunction, see \cite[Proposition 5.1.12]{2018arXiv180104797D} for instance.
\begin{proposition}\label{prop:isohomotopyclassestop}
Let $\Psi_{K}\colon K \to \ttN(\ttP)$ be an  object of $\ssetp$
and $\psi_{X}\colon X\to \ttP$ an object of $\topp$.
 Then, there is an isomorphism of simplicial sets,
 \begin{equation}\label{equa:homdeltatop}
 \HomD_{\topp}(|\Psi_{K}|,\psi_{X})\cong \HomD_{\ssetp}(\Psi_{K},\singp\psi_{X}).
 \end{equation}
\end{proposition}

An application of $\pi_{0}$ to \eqref{equa:homdeltatop}
gives the following property.

\begin{corollary}
There is  an isomorphism  between the homotopy classes,
$$[|\Psi_{K}|,\psi_{X}]_{\topp} 
\cong [\Psi_{K},\singp \psi_{X}]_{\ssetp}.$$
\end{corollary}

\begin{remark}
Let $\psi_{X}\colon X\to \ttP$ be a stratified space with $X$ conically stratified
(\cite[Definition A.5.5]{LurieSecondBook}).
In  \cite[Theorem A.6.4]{LurieSecondBook}, Lurie proves that
the associated simplicial set over $\ttP$,
$ \Psi_{X}\colon \singp \psi_{X} \to \ttN(\ttP)$,
is a fibrant object in the Joyal closed model structure on $\ssetp$.
 This is not true in  general for  stratified spaces,
see \cite[Example 4.13]{2018arXiv180104797D}.
(Fibrant objects for the Joyal structure have the lifting property for the horns
$\Lambda^i[n]$ such that $0<i<n$, $n\geq 1$.)
In this work, for having a representability theorem, we use the simplicial structure on $\ssetp$ and 
the following result is sufficient for our purpose. 
\end{remark}

\begin{proposition}\label{prop:singsfibrant}
Let $\psi_{X}\colon X\to \ttP$ be a topological space over the poset $\ttP$.
The  associated simplicial set over $\ttP$, $\Psi_{X}\colon \singp \psi_{X}\to \ttN(\ttP)$, is s-fibrant in $\ssetp$.
\end{proposition}

\begin{proof}
Let $\Psi_{K}\colon K\to \ttN(\ttP)$ be any object of $\ssetp$. 
With \defref{def:sfibrant}, we have to prove that
the simplicial set
$\HomD_{\ssetp}(\Psi_{K},\Psi_{X})$ is a fibrant simplicial set.
This is equivalent  to prove the 
surjectivity of
$$\Hom_{\sset}(\Delta[m],\HomD_{\ssetp}(\Psi_{K},\Psi_{X}))
\to
\Hom_{\sset}(\Lambda^k[m],\HomD_{\ssetp}(\Psi_{K},\Psi_{X})),
$$
for any $m\geq 1$ and any $0\leq k\leq m$.
This  amounts (cf. \defref{def:simplicialcat}) to the surjectivity of
$$\Hom_{\ssetp}(\Psi_{K}\otimes \Delta[m],\Psi_{X})
\to
\Hom_{\ssetp}(\Psi_{K}\otimes \Lambda^k[m],\Psi_{X}).$$
Using the adjunction $(|-|,\singp)$ and the compatibility of the realization functor with products, 
the previous surjectivity is equivalent to the surjectivity of
$$\Hom_{\topp}(|\Psi_{K}|\times  |\Delta[m]|,\psi_{X})
\to
\Hom_{\topp}(|\Psi_{K}|\times  |\Lambda^k[m]|,\psi_{X}),$$
which arises from the existence of a retraction to the canonical  injection
$|\Lambda^k[m]|\to |\Delta[m]|$.
\end{proof}

\subsection{Connection with simplicial sets}\label{subsec:ssetpsset}
There is an adjunction,
$$\xymatrix{
\ssetp
\ar@<1ex>[rr]^-{\cU}
&&\sset,
\ar@<1ex>[ll]^-{\ttN(\ttP)\times-}
}$$
where $\cU$ is a forgetful functor sending  $\Psi_{L}\in \ssetp$ to $L\in \sset$,
and the functor $ \ttN(\ttP)\times-$ sends the simplicial set $Z$ to
 the  projection $p_{\ttP}\colon  \ttN(\ttP) \times Z \to \ttN(\ttP)$.

\begin{proposition}\label{prop:forgetpair}
The pair of functors $(\cU, \ttN(\ttP)\times-)$ forms a simplicial adjunction: for each $Z\in \sset$ and 
each $\Psi_{L}\in\ssetp$, there is an isomorphism
$$\HomD_{\sset}(\cU(\Psi_{L}),Z)\cong\HomD_{\ssetp}(\Psi_{L},\ttN(\ttP)\times Z).$$
\end{proposition}

\begin{proof}
To $f\in \Hom_{\sset}(\cU(\Psi_{L}),Z)$, one associates
$(f,\Psi_{L})\in \Hom_{\ssetp}(\Psi_{L}, \ttN(\ttP)\times Z)$, and
this correspondence is clearly a bijection.
The simplicial adjunction follows from the following isomorphisms between the sets of $n$-simplices,
\begin{eqnarray*}
\left(\HomD_{\sset}(\cU(\Psi_{L}),Z)\right)_{n}
&=&
\Hom_{\sset}(\cU(\Psi_{L})\otimes \Delta[n],Z)\\
&=&
\Hom_{\sset}(\cU(\Psi_{L}\otimes \Delta[n]),Z)\\
&\cong&
 \Hom_{\ssetp}(\Psi_{L}\otimes \Delta[n], \ttN(\ttP)\times Z)\\
&=&
\left(\HomD_{\ssetp}(\Psi_{L}, \ttN(\ttP)\times Z)\right)_{n}.
\end{eqnarray*}
\end{proof}

\subsection{Restricted  simplicial sets over a poset}\label{subsec:ssetp+}

\begin{definition}\label{def:simplexregular}
Let $\ttP$ be a poset.
A simplex $\tts_{0}\preceq \tts_{1}\preceq\dots\preceq \tts_{\ell}$
of $\ttN(\ttP)$ is \emph{regular} if $\tts_{\ell}$ is a maximal element of $\ttP$.
We denote $\Delta[\ttP]^+$ the full sub-category of $\Delta[\ttP]$ whose objects are the \emph{regular simplices}
of $\ttN(\ttP)$. 
\end{definition}

\begin{definition}
A \emph{restricted simplicial set over $\ttP$} is a functor $\Psi_{K}\colon (\Delta[\ttP]^+)^{\op}\to \set$.
The category of restricted  simplicial sets, with morphisms the natural transformations, is denoted $\ssetp^+$.
As $\Delta[\ttP]^+\subset \Delta[\ttP]$, there is a \emph{restriction functor} $$\cR\colon \ssetp\to \ssetp^+.$$
\end{definition}

We also adapt the  notion of regular stratum of a stratified topological space  to the simplicial paradigm.
We call it non singular to make a clear distinction from the previous notion of regular simplex.

\begin{definition}\label{def:nonsingular}
Let $\Psi_{K}\in\ssetp$. 
The \emph{non singular part} of $\Psi_{K}$ is the restriction of $\Psi_{K}$
to the simplicial subset of $K$ formed of simplices $\sigma$
such that $\Psi_{K}\circ \sigma$ is a maximal element of $\ttP$.
If each simplex of $\Psi_{K}$ is non singular, we say that $\Psi_{K}$ is a \emph{completely regular simplicial set.} 
\end{definition}

 \begin{example}
Let $\ttP=[n]$. An object $\Psi_{K}$ of $\ssetp$ can be described as a family of sets $K_{i_{0},\dots,i_{n}}$, with 
$$i_{j}\in \ \{-1\}\cup  \N.$$ 
The elements of $K_{i_{0},\dots,i_{n}}$ are the simplices $\Delta[i_{0}]\ast\dots\ast\Delta[i_{n}]$,
the value -1 corresponding to an empty subset.
The objects of $\ssetp^+$ correspond to the families  $K_{i_{0},\dots,i_{n}}$ with $i_{n}\neq -1$.
The non singular part of $\Psi_{K}$ is the union of 
$K_{-1,\dots,-1,i_{n}}$ with $i_{n}\neq -1$.
\end{example}

We now introduce two functors from $\ssetp^+$ to $\ssetp$,

$$\xymatrix{
\ssetp^+
\ar@<1ex>[rr]^-{\ttn}
\ar@/_1pc/[rr]_{\tti}
&&
\ssetp
\ar[ll]^-{\cR}
}$$

\begin{definition}\label{def:functorn}
The functor $\ttn\colon \ssetp^+\to\ssetp$ is defined for any $\sigma\in\ttN(\ttP)$ by,
$$(\ttn(\Psi_{K}))_{\sigma}=\left\{
\begin{array}{cl}
(\Psi_{K})_{\sigma}&\text{if } \sigma \text{ is regular},\\
\tv&\text{otherwise,}
\end{array}\right.
$$
where $\tv$ is an additional vertex.
\end{definition}

\begin{definition}\label{def:functiri}
The functor $\tti\colon \ssetp^+\to\ssetp$ is  defined by left Kan extension.
First, to any restricted  simplicial set over $\ttP$,
$\left(\Delta[i_{0}]\ast\dots\ast\Delta[i_{n}]\right)^+$, 
we associate the simplicial set over $\ttP$,
$\Delta[i_{0}]\ast\dots\ast\Delta[i_{n}]$.
Then, as an object $\Psi_{K}$ of $\ssetp^+$ is a colimit in $\ssetp^+$,
$$\Psi_{K}=\varinjlim_{\left(\Delta[i_{0}]\ast\dots\ast\Delta[i_{n}]\right)^+\to \Psi_{K}} \left(\Delta[i_{0}]\ast\dots\ast\Delta[i_{n}]\right)^+,$$
we set $\tti(\Psi_{K})$ as a colimit in $\ssetp$,
$$\tti\left(\Psi_{K}\right)=\varinjlim_{\left(\Delta[i_{0}]\ast\dots\ast\Delta[i_{n}]\right)^+\to \Psi_{K}} \Delta[i_{0}]\ast\dots\ast\Delta[i_{n}].$$
\end{definition}

We present below examples of the compositions $\tti\circ \cR$ and $\ttn\circ \cR$.

 \begin{example}\label{exam:papillon}[{\it The composition $\tti\circ \cR$.}]
 Let $\ttP= [1]$.
This first example starts with an ordered simplicial complex, constituted of two 2-simplices with a vertex $a$ in common.
The vertex $a$ is  singular and the other ones completely regular:
$$K=\Delta[0]\ast\Delta[1]\cup_{\Delta[0]}\Delta[0]\ast\Delta[1].$$ 
Below, we draw $K$, 
its restriction $\cR(K)\in\ssetp^+$ and the composite $\tti(\cR(K))\in\ssetp$.

\hskip 1.2cm
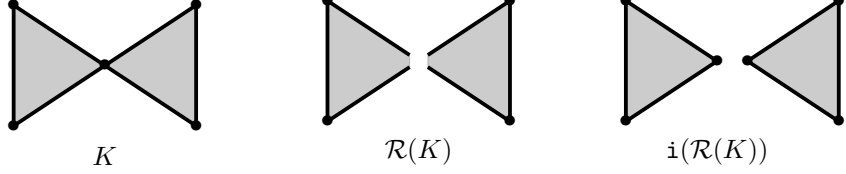
\begin{tikzpicture}[black,scale=.4]
\draw [fill=black, opacity=0.2] (0,0) -- (3,2) -- (0,4) -- cycle;
\draw[line width=0.5mm](0,0) -- (3,2) -- (0,4) -- (0,0);
 
 \draw(0,0) node {$\bullet$} ;
  \draw(3,2) node {$\bullet$} ;
   \draw(0,4) node {$\bullet$} ;
    \draw(6,0) node {$\bullet$} ;
     \draw(6,4) node {$\bullet$} ;

\draw [fill=black, opacity=0.2] (3,2) -- (6,0) -- (6,4) -- cycle;
\draw[line width=0.5mm](3,2) -- (6,0) -- (6,4) -- (3,2);
\draw(3,-1) node {\negro$K$};
\end{tikzpicture}
\hskip 1.2cm
\begin{tikzpicture}[black,scale=.4]

\draw [fill=black, opacity=0.2] (0,0) -- (2.7,1.8) -- (2.7,2.2)  -- (0,4) -- cycle;
\draw[line width=0.5mm] (2.7,2.2)  -- (0,4) -- (0,0) -- (2.7,1.8) ;
 
 \draw(0,0) node {$\bullet$} ;
 
   \draw(0,4) node {$\bullet$} ;
    \draw(6,0) node {$\bullet$} ;
     \draw(6,4) node {$\bullet$} ;

\draw [fill=black, opacity=0.2] (3.3,1.8) -- (6,0) -- (6,4) -- (3.3,2.2) -- cycle;
\draw[line width=0.5mm](3.3,1.8) -- (6,0) -- (6,4) --  (3.3,2.2);
\draw(3,-1) node {$\cR(K)$};
\end{tikzpicture}
\hskip 1cm
\begin{tikzpicture}[black,scale=.4]
\draw [fill=black, opacity=0.2] (0,0) -- (3,2) -- (0,4) -- cycle;
\draw[line width=0.5mm](0,0) -- (3,2) -- (0,4) -- (0,0);
 
 \draw(0,0) node {$\bullet$} ;
  \draw(3,2) node {$\bullet$} ;
   \draw(0,4) node {$\bullet$} ;
    \draw(7,0) node {$\bullet$} ;
     \draw(7,4) node {$\bullet$} ;
      \draw(4,2) node {$\bullet$} ;

\draw [fill=black, opacity=0.2] (4,2) -- (7,0) -- (7,4) -- cycle;
\draw[line width=0.5mm](4,2) -- (7,0) -- (7,4) -- (4,2);
\draw(3,-1) node {$\tti(\cR(K))$};
\end{tikzpicture}

\noindent
The link (\cite[Definition 1.11]{CST1}) of the singular vertex of $K$ is not connected, 
which means by definition that $K$ is not normal (\cite[Definition 1.55]{CST1}).
In contrast, $(\tti\circ\cR)(K)$ is normal.
One can also notice that $\ttn(\cR(K))=K$ for this example.
\end{example}

The composition $\tti\circ \cR$ corresponds (see \propref{prop:normal}) 
to the process of normalization introduced in \cite{GM1} 
and justifies the following definition. 

\begin{definition}\label{def:normalisation}
 The composite $\tti\circ \cR\colon \ssetp\to\ssetp$ is called the \emph{normalization functor.}
 \end{definition}

\begin{example}\label{exam:crac}[{\it The composition $\ttn\circ \cR$.}]
Let $\ttP=[1]$.
We consider the ordered simplicial complex $K=\Delta[1]\ast \Delta[0]$,
with $\Delta[1]$ singular and $\Delta[0]$ completely regular. 
We denote $a,b$ the vertices of $\Delta[1]$ and $c$ the vertex of $\Delta[0]$.
Below, we draw $K$, 
its restriction $\cR(K)\in\ssetp^+$ and the composite $\ttn(\cR(K))\in\ssetp$.
 
 \hskip 2cm
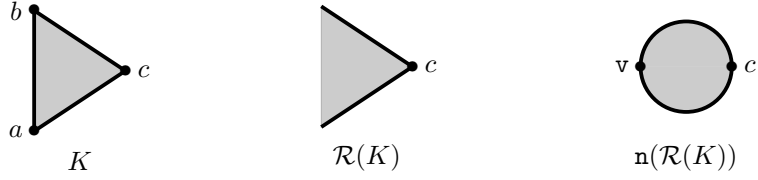
\begin{tikzpicture}[black,scale=.4]
\draw [fill=black, opacity=0.2] (0,0) -- (3,2) -- (0,4) -- cycle;
\draw[line width=0.5mm](0,0) -- (3,2) -- (0,4) -- (0,0);
 
 \draw(0,0) node {$\bullet$} ;
  \draw(-0.6,0) node {$a$} ;
  \draw(3,2) node {$\bullet$} ;
  \draw(3.6,2) node {$c$} ;
   \draw(0,4) node {$\bullet$} ;
    \draw(-0.6,4) node {$b$} ;

\draw(1.5,-1) node {\negro$K$};
\end{tikzpicture}
\hskip 2cm
\begin{tikzpicture}[black,scale=.4]
\draw [fill=black, opacity=0.2] (0,0) -- (3,2) -- (0,4) -- cycle;
\draw[line width=0.5mm](0,0) -- (3,2) -- (0,4) ;
 
  \draw(3,2) node {$\bullet$} ;
   \draw(3.6,2) node {$c$} ;

\draw(1.5,-1) node {$\cR(K)$};
\end{tikzpicture}
\hskip 2cm
\begin{tikzpicture}[black,scale=.4]
\draw [line width=0.5mm] (0,2)  arc (180:0:1.5) ;
\draw [line width=0.5mm] (0,2)  arc (-180:0:1.5) ; 
\draw [fill=black, opacity=0.2] (0,2) arc (180:0:1.5);
\draw [fill=black, opacity=0.2] (0,2) arc (-180:0:1.5);

  \draw(0,2) node {$\bullet$} ;
  \draw(3,2) node {$\bullet$} ;
   \draw(3.6,2) node {$c$} ;
 \draw(-0.6,2) node {$\tv$} ;

\draw(1.5,-1) node {$\ttn(\cR(K))$};
\end{tikzpicture}

\noindent
The image $\ttn(\cR(K))\in \ssetp$ is the quotient of $\Delta[1]\ast \Delta[0]$ by the relation $\Delta[1]=\ttv$.
One can also notice that $\tti
 (\cR(K))=K$ for this example.
\end{example}

The next result comes directly from the definitions of these functors.

\begin{proposition}\label{prop:leftrightR}
The functors $\ttn$ and $\tti$ are respectively the right and the left adjoints of the restriction functor, $\cR$.
They verify $\cR\circ \tti=
\cR\circ \ttn=
\id$.
\end{proposition}

Thus $\cR$ preserves limits and colimits, $\tti$ is compatible with colimits and $\ttn$ with limits. 
Also, the functors $\cR$, $\ttn$, $\tti$ verify, for any 
$\Psi_{K},\Psi_{K_{1}},\Psi_{K_{2}}\in \ssetp^+$ and $\Psi_{L}\in\ssetp$,
\begin{eqnarray}
\Hom_{\ssetp}(\Psi_{L},\ttn(\Psi_{K}))
&\cong&
\Hom_{\ssetp^+}(\cR(\Psi_{L}),\Psi_{K}),\label{equa:leftright}\\
\Hom_{\ssetp}(\tti(\Psi_{K}),\Psi_{L})
&\cong&
\Hom_{\ssetp^+}(\Psi_{K},\cR(\Psi_{L})),\label{equa:leftright2}
\end{eqnarray}
which imply,
with $\cR\circ \tti=\cR\circ \ttn=\id$,
\begin{eqnarray}\label{equa:leftright3}
\Hom_{\ssetp^+}(\Psi_{K_{1}},\Psi_{K_{2}}))
&\cong&
\Hom_{\ssetp}(\tti(\Psi_{K_{1}}),\ttn(\Psi_{K_{2}}))\\ 
&\cong&
\Hom_{\ssetp}(\tti(\Psi_{K_{1}}),\tti(\Psi_{K_{2}}))\nonumber \\
&\cong&
\Hom_{\ssetp}(\ttn(\Psi_{K_{1}}),\ttn(\Psi_{K_{2}})).\nonumber
\end{eqnarray}

\subsection{Simplicial structures}
The structure of simplicial category on $\ssetp$
induces a structure of simplicial category on $\ssetp^+$ by
\begin{equation}\label{equa:simplicial+}
\Hom^{\Delta}_{\ssetp^+}(\Psi_{K_{1}},\Psi_{K_{2}})=\Hom^{\Delta}_{\ssetp}(\tti(\Psi_{K_{1}}),\ttn(\Psi_{K_{2}})).
\end{equation}
Let $\Psi_{K}\in \sset^+$. The  products on $\ssetp$ and $\ssetp^+$ are linked by
$$\Psi_{K}\otimes \Delta[n]=\cR(\tti(\Psi_{K})\otimes \Delta[n]).$$
One can also observe from the definition of $\tti$ that 
\begin{equation}\label{equa:producti}
\tti(\Psi_{K}\otimes \Delta[n])=\tti(\Psi_{K})\otimes \Delta[n].
\end{equation}

\begin{lemma}\label{lem:simplicial+}
The set of $n$-simplices of the simplicial set 
$\Hom^{\Delta}_{\ssetp^+}(\Psi_{K_{1}},\Psi_{K_{2}})$
is given by
$$(\Hom^{\Delta}_{\ssetp^+}(\Psi_{K_{1}},\Psi_{K_{2}}))_{n}=
\Hom_{\ssetp^+}(\Psi_{K_{1}}\otimes \Delta[n],\Psi_{K_{2}}).$$
\end{lemma}

\begin{proof}
Using \eqref{equa:leftright3}, \eqref{equa:simplicial+} and the adjunction, we have
\begin{eqnarray*}
\left(\Hom^{\Delta}_{\ssetp^+}(\Psi_{K_{1}},\Psi_{K_{2}})\right)_{n}
&=&
\left(\Hom^{\Delta}_{\ssetp}(\tti(\Psi_{K_{1}}),\ttn(\Psi_{K_{2}}))\right)_{n}\\
&=&
\Hom_{\ssetp}(\tti(\Psi_{K_{1}})\otimes \Delta[n],\ttn(\Psi_{K_{2}}))\\
&=&
\Hom_{\ssetp}(\tti(\Psi_{K_{1}}\otimes \Delta[n]),\ttn(\Psi_{K_{2}}))\\
&=&
\Hom_{\ssetp^+}(\Psi_{K_{1}}\otimes \Delta[n], \Psi_{K_{2}}).
\end{eqnarray*}
\end{proof}

\begin{proposition}\label{prop:inadjoint}
The functors $(\cR,\ttn)$ and $(\tti,\cR)$ are  adjoint pairs of simplicial functors; i.e.,
for $\Psi_{L}\in \ssetp$ and $\Psi_{K}\in \ssetp^+$, we have
$$\HomD_{\ssetp}(\Psi_{L},\ttn(\Psi_{K}))\cong
\HomD_{\ssetp^+}(\cR(\Psi_{L}),\Psi_{K})
$$
and
$$\HomD_{\ssetp^+}(\Psi_{K},\cR(\Psi_{L}))
\cong
\HomD_{\ssetp}(\tti(\Psi_{K}),\Psi_{L}).
$$
\end{proposition}

\begin{proof}
Using $\cR(\Psi_{L}\otimes \Delta[n])=\cR(\Psi_{L})\otimes \Delta[n]$ and 
$\tti(\Psi_{K}\otimes \Delta[n])=\tti(\Psi_{K})\otimes \Delta[n]$,  
the proof is a consequence of the properties of adjunctions and \lemref{lem:simplicial+}.
\end{proof}

\section{Blown up cochains on a poset}\label{sec:boum}

\begin{quote}
In this section, we recall the notion of perversity and 
present the blown up cohomology associated to a poset, $\ttP$,  which corresponds to a local situation.
\end{quote}

\subsection{Perversity on a poset}
First comes the principal tool for the Goresky and MacPherson theory: the notion of perversity.

\begin{definition}\label{def:poset}
A \emph{perversity on a poset $\ttP$} is a map $\ov{p}\colon \ttP\to \ov{\Z}=\Z\cup\{-\infty,\infty\}$
taking the value 0 on the maximal elements.
If $\ov{p}$ and $\ov{q}$ are two perversities on $\ttP$, we write  $\ov{p}\leq \ov{q}$ when
$\ov{p}(\tts)\leq\ov{q}(\tts)$, for each $\tts\in \ttP$.  

Given two perversities, $\ov{p}$ and $\ov{q}$, we define a new perversity, $\ov{p}+\ov{q}$,
by $(\ov{p}+\ov{q})(\tts)=\ov{p}(\tts)+\ov{q}(\tts)$, with the conventions
$k + (-\infty)=-\infty+k=-\infty$, for any $k\in\ov{\Z}$, and
$\ell+(+\infty)=+\infty +\ell=+\infty$, for any $\ell\in\Z$.
(The first  convention is required in  the definition of the cup product in blown up cohomology, 
see \cite[Proposition 4.2]{CST4}.)
\end{definition} 

For stratified topological spaces, perversities defined on the set of strata
already appeared in \cite{MacPherson90} but, 
historically, the first ones in \cite{GM1, GM2} correspond to perversities defined on the poset
$\N^\op$. In fact, in the case of geometrical data,  the strata come with an intrinsic notion of 
dimension and in the two seminal papers quoted above, perversities retain only this information.
Let us give an illustration  with  pseudomanifolds.
Being the spaces with singularities satisfying  a Poincar\'e duality (see \cite{GM1, CST2, ST1})
they play a central role in the theory. 

 \begin{definition}\label{def:pseudomanifold}
A \emph{topological pseudomanifold of dimension $n$} (or a pseudomanifold) is 
a Hausdorff  space
together with a filtration by closed subsets,
$$
X_{-1}=\emptyset\subseteq X_0 \subseteq X_1 \subseteq \dots \subseteq X_{n-2} \subseteq X_{n-1} \subsetneqq X_n =X,
$$
such that, for each $i\in\{0,\dots,n\}$, 
$X_i\backslash X_{i-1}$ is a topological manifold of dimension~$i$ or the empty set. 
The subspace $X_{n-1}$ is called \emph{the singular set} and each 
point $x \in X_i \backslash X_{i-1}$ with $i\neq n$ admits
\begin{enumerate}[(i)]
\item an open neighborhood $V$ of $x$ in $X$, endowed with the induced filtration,
\item an open neighborhood $U$ of $x$ in  $X_i\backslash X_{i-1}$, 
\item a  compact pseudomanifold $L$  of dimension  $n-i-1$, whose cone $\rc L$ is endowed with the conic filtration, 
$(\rc L)_{i}=\rc L_{i-1}$,
\item a homeomorphism, $\varphi \colon U \times \rc L\to V$, 
such that
\begin{enumerate}[(a)]
\item $\varphi(u,\tv)=u$, for any $u\in U$, where $\tv$  is the apex of $\rc L$,
\item $\varphi(U\times \rc L_{j})=V\cap X_{i+j+1}$, for any $j\in \{0,\dots,n-i-1\}$.
\end{enumerate}
\end{enumerate}
The pseudomanifold $L$ is called the \emph{link} of $x$. 
\end{definition}

By taking the partition $X=\sqcup \,X_{i}\backslash X_{i-1}$, we get a stratified space 
(\cite[Theorem G]{CST1})
as in \defref{subsec:stratifiedspacesdef} and thus a stratification over its poset of strata.
As it  appears in the next sections, the study of intersection cohomology of stratified spaces 
does not need any notion of dimension or codimension. But, as quoted above,  a
feature of pseudomanifolds is the existence of a geometrical notion of dimension.
Moreover, in the seminal work of Goresky and MacPherson, the perversities in use need this notion of dimension.
Not only, perversities need it but they have a geometrical meaning, as  they control
 the tangential component of the singular simplices  relatively to the strata of $X$. 

To keep this peculiarity of pseudomanifolds, we can give preference to the opposite poset of the natural integers
instead of the poset of strata. More precisely, to any $n$-dimensional pseudomanifold $X$ we associate
the continous map $\varphi_{X}\colon X\to [n]^{\op}$, sending a point
$x\in X$ to the codimension of the stratum $S$ with $x\in S$.
Let us observe that $\varphi_{X}^{-1}(k)=X_{k}\menos X_{k-1}$.
With these observations, let us recall how perversities appears in \cite{GM1}.

\begin{definition}\label{def:GMperversite}
A \emph{{\rm GM}-perversity} is a map $\ov{p}\colon [n]^{\op}\to\Z$ such that
$\ov{p}(0)=\ov{p}(1)=\ov{p}(2)=0$ and
$\ov{p}(i)\leq\ov{p}(i+1)\leq\ov{p}(i)+1$, for all $i\geq 2$.
As particular case, we have the null perversity~$\ov{0}$ constant with value~0 and the \emph{top perversity} 
$\ov{t}$ defined by
$\ov{t}(i)=i-2$ if $i\geq 2$. 
For any perversity, $\ov{p}$, the perversity $D\ov{p}:=\ov{t}-\ov{p}$ is called the
\emph{complementary perversity} of $\ov{p}$.
\end{definition}

We complete this paragraph with the transfer of perversities through a map of posets.
An exhaustive topological study of these operations on perversities
is done in \cite{saralegiaranguren2019refinement}. 
This transfer has been also used in \cite{CST3} in relation with the analysis
of the topological invariance of intersection homology.

\begin{definition}\label{def:perversitypullback}
Let $\ttf\colon \ttP\to \ttQ$ be a morphism of  posets. 
 If $\ov{q}\colon \ttQ\to \ov{\Z}$
is a perversity on $\ttQ$, \emph{the pullback perversity} $\ttf^*\ov{q}$  on $\ttP$ is defined by
$\ttf^*\ov{q}(\tts)=\ov{q}\circ \ttf(\tts)$ if $\tts$ is not maximal.
 If $\ov{p}\colon \ttP\to \ov{\Z}$ is a perversity on $\ttP$, \emph{ the pushforward perversity}  $\ttf_{*}\ov{p}$  on $\ttQ$ is defined on a not maximal element $\ttt$ by
$$\ttf_{*}\ov{p}(\ttt)=\inf_{\ttf(\tts)=\ttt}\ov{p}(\tts).$$
\end{definition}

\subsection{Perverse degree}\label{subsec:perversedegree}
Recall some notations and conventions, already used in previous works as \cite[Subsection 2.1]{CST2}.
The cone on the simplicial set $\Delta[n]=[e_{0},\dots,e_{n}]$ is the simplicial set 
$\tc \Delta[n]=[e_{0},\dots,e_{n},\tv]$. 
The apex $\tv$ is called a \emph{virtual vertex} and can be considered as the cone on the empty set,
$\tc \emptyset =[\tv]$.
Given a simplex $F$  of  $\Delta[n]$, we denote by $(F,0)$ the same simplex viewed as a face of 
$\tc \Delta[n]$ and by $(F,1)$  the face $\tc F$ of $\tc \Delta[n]$.

\medskip
The blown up cochains associated to a perversity need the introduction of an extra degree  that we detail now.
To any regular simplex 
$\Delta[m]=\Delta[q_{0}]\ast\dots\ast\Delta[q_{\ell}]$ of $\ttN(\ttP)$,
we associate the prism
$\tDelta[m]=\tc\Delta[q_{0}]\times \dots \times \tc \Delta[q_{\ell-1}]\times \Delta[q_{\ell}]$,
called \emph{the blow up of $\Delta[m]$.}
A face of the blow up $\tDelta[m]$ is a product 
\begin{equation}\label{equa:fepsilon}
(F,\varepsilon)=(F_{0},\varepsilon_{0})\times\dots\times (F_{\ell-1},\varepsilon_{\ell-1})\times F_{\ell},
\end{equation}
where, following the previous conventions,
\begin{itemize}
\item if $\varepsilon_{i}=0$ or $i=\ell$, then $F_{i}$ is a face of $\Delta[q_{i}]$,
\item  if $\varepsilon_{i}=1$ and $F_{i}\neq \emptyset$, then $(F_{i},1)$ is the cone $\tc F_{i}$ on the face 
$F_{i}$ of $\Delta[q_{i}]$,
\item  if $F_{i}=\emptyset$, then $\varepsilon_{i}=1$ and $(F_{i},1)=\tv_{i}$.
\end{itemize}
For any $i\in\{0,\dots,\ell-1\}$, we denote
$$\|(F,\varepsilon)\|_{i}=
 \dim ((F_{i+1},\varepsilon_{i+1})\times\dots\times (F_{\ell-1},\varepsilon_{\ell -1})\times F_{\ell}).
$$

\begin{definition}\label{def:degreperversboum}
Let  $\Delta[m]=\Delta[q_{0}]\ast\dots\ast\Delta[q_{\ell}]=\tts_{0}^{[q_{0}]}\prec\dots\prec \tts_{\ell}^{[q_{\ell}]}$
be a regular simplex of $\ttN(\ttP)$ and
 $(F,\varepsilon)=(F_{0},\varepsilon_{0})\times\dots\times (F_{\ell-1},\varepsilon_{\ell-1})\times F_{\ell}$
be a face  of the blow up $\tDelta[m]$.
The  \emph{perverse degree} of 
$(F,\varepsilon)$  along the element $\tts\in \ttP$ is
$$
\|(F,\varepsilon)\|_{\tts}=\left\{
\begin{array}{ccl}
-\infty&\text{if} &\tts\notin \{\tts_{0},\dots,\tts_{\ell}\} \text{ or } (\tts=\tts_{i} \text{ and } \varepsilon_{i}=1),\\
\|(F,\varepsilon)\|_{i}&\text{if} & \tts=\tts_{i} \text{ and } \varepsilon_{i}=0.
\end{array}\right.$$
\end{definition}

\begin{remark}
The perverse degree of \defref{def:degreperversboum} coincides with the perverse degree associated to a weight decomposition
introduced in \cite[Definition 3.3]{CST4}.
\end{remark}

We fix a commutative ring with unit  $R$. Let $j\in\N$  and
 $N^*(\Delta[j])=\Hom(N_{*}(R\Delta[j]),R)$  the dual of the Moore complex associated to $R\Delta[j]$. 
For each simplex $F \in \Delta[j]$, we write $\1_{F}$ the element of $N^*(\Delta[j])$ 
taking the value 1 on $F$ and 0 otherwise.
Let
$\Delta[m]=\Delta[q_{0}]\ast\dots\ast\Delta[q_{\ell}]$, with $q_{i}\geq 0$ for all $i$.
We first define the blown up complex on $\Delta[m]$ by
\begin{equation}\label{equa:boumcochain}
\tN^*(\Delta[m])=N^*(\tc\Delta[q_{0}])\otimes\dots\otimes N^*(\tc\Delta[q_{\ell-1}])\otimes N^*(\Delta[q_{\ell}]).
\end{equation}
The elements 
$\1_{(F,\varepsilon)}=
\1_{(F_{0},\varepsilon_{0})}\otimes\dots\otimes \1_{(F_{\ell-1},\varepsilon_{\ell-1})}\otimes \1_{F_{\ell}}$
form a basis of $\tN^*(\Delta[m])$. (By convention, we also set $\varepsilon_{\ell}=0)$.


\medskip
We describe the maps induced by the morphisms of $\ttN(\ttP)^+$ between the blown up complexes.
Let us begin with the
\emph{regular} 
face operators of $\Delta[m]=\Delta[q_{0}]\ast\dots\ast\Delta[q_{\ell}]$ with $q_{i}\geq 0$ for all $i$.
Let  $\alpha\colon\nabla\to \Delta[m]$ be a face map with $\nabla$ and $\Delta[m]$ regular.
The induced filtration on $\nabla$ gives a decomposition
$$\nabla=\ov{\nabla}_{0}\ast\dots\ast\ov{\nabla}_{\ell},\;\text{ with } \ov{\nabla}_{i}=\nabla\cap \Delta[q_{i}],$$
in which some $\ov{\nabla}_{i}$ can be the empty set. 
We  get rid of these empty factors 
to obtain what we call \emph{the solid $\Delta[m]$-decomposition} of $\nabla$,
$$\nabla=\Delta[p_{0}]\ast\dots\ast\Delta[p_{k}], \;\text{ with } p_{i}\geq 0 \text{ for } 0\leq i\leq k.$$
More precisely, as $\nabla$ is regular, there is a strictly increasing map,
$\eta\colon \{0,\dots,k\}\to \{0,\dots,\ell\}$, with $\eta(k)=\ell$, defined by
$$\ov{\nabla}_{j}=\left\{
\begin{array}{ccl}
\emptyset&\text{if}&j\notin \im(\eta),\\
\Delta[p_{i}]&\text{if}&j=\eta(i).
\end{array}\right.
$$

Let
$\tN^*(\nabla)=N^*(\tc\Delta[p_{0}])\otimes\dots\otimes N^*(\tc\Delta[p_{k-1}])\otimes N^*(\Delta[p_{k}])$
be the blown up complex of $\nabla$ endowed with its solid $\Delta[m]$-decomposition.
The face map
$\alpha\colon\nabla\to \Delta[m]$
induces a cochain map,
$$\alpha^*\colon \tN^*(\Delta[m])\to\tN^*(\nabla),$$
defined as follows. Let
$\1_{(F,\varepsilon)}=
\1_{(F_{0},\varepsilon_{0})}\otimes\dots\otimes \1_{(F_{\ell-1},\varepsilon_{\ell-1})}\otimes \1_{F_{\ell}}\in \tN^*(\Delta[m])$. We say that $F$ is \emph{$\nabla$-compatible} if $(F_{i},\varepsilon_{i})=(\emptyset,1)$ for all 
$i\notin \im(\eta)$.
(The $\nabla$-compatibility ensures that $F=F_{0}\ast\dots\ast F_{\ell}$ is a face of $\nabla$.)
We have:
$$\alpha^*(\1_{(F,\varepsilon)})=\left\{
\begin{array}{cl}
0&  \text{if } F \text{ is not } \nabla\text{-compatible},\\
\1_{(H,\varepsilon)}&\text{otherwise},
\end{array}\right.
$$
where
$\1_{(H,\varepsilon)}=\1_{(H_{0},\varepsilon_{0})}\otimes\dots\otimes \1_{(H_{k-1},\varepsilon_{k-1})}\otimes \1_{H_{k}}$,
with $(H_{i},\varepsilon_{i})=(F_{\eta(i)},\varepsilon_{\eta(i)})$ for all $i\in\{0,\dots,k\}$.

\medskip
Let us consider  a degeneracy map, $\beta\colon \Delta[m]\to\Delta[m+1]$.
Such map $\beta$ being the repetition of a vertex  in some $\Delta[q_{i}]$, %
we  have a chain map, $N_{*}(\Delta[q_{i}])\to N_{*}(\Delta[q_{i}+1])$, which  gives
cochain maps
$N^{*}(\Delta[q_{i}+1])\to N^{*}(\Delta[q_{i}])$ and
$N^{*}(\tc \Delta[q_{i}+1])\to N^{*}(\tc \Delta[q_{i}])$. 
Tensoring with the identity map
on the other components of the tensor product gives 
the  cochain map
$\beta^*\colon \tN^*(\Delta[m+1])\to \tN^*(\Delta[m])$.

\medskip
Let  $\ov{p}$ be a perversity on the poset $\ttP$.
Let $R$ be a commutative ring and $\mdg$  be the category of positively graded differential graded $R$-modules,
 with a differential of degree +1.
 We have defined a functor
 $\tN^*\colon \Delta[\ttP]^+\to \mdg$,
 sending $\Delta[m]\in \Delta[\ttP]^+$ on the differential complex $\tN^*(\Delta[m])$.

 \begin{definition}\label{def:degrepervers}
Let $\Delta[m]\in \Delta[\ttP]^+$.
\begin{enumerate}[1)]
\item The  \emph{perverse degree} of 
$\1_{(F,\varepsilon)}\in \tN^*(\Delta[m])$ along  an element $\tts\in \ttP$ is
the perverse degree of $(F,\varepsilon)$ along $\tts$.
For a cochain $\omega = \sum_b\lambda_b \ \1_{(F_b,\varepsilon_b) }\in\tN^*(\Delta[m])$ with 
$\lambda_{b}\neq 0$ for all $b$,
the \emph{perverse degree along $\tts$} is
$$\|\omega \|_{\tts}=\max_{b}\|{(F_b,\varepsilon_b)}\|_{\tts}.$$
By convention, we set $\|0\|_{\tts}=-\infty$.
We denote $\|\omega\|\colon \ttP\to \ov{\Z}$ the map which associates $\|\omega\|_{\tts}$
 to any $\tts\in\ttP$. 
 \item The cochain $\omega$
 is \emph{$\ov{p}$-allowable} if $\| \omega\|\leq \ov{p}$ 
 and of \emph{$\ov{p}$-intersection} if $\omega$ and 
 its differential $\delta\omega$ are $\ov{p}$-allowable. 
  We denote $\tN^*_{\ov{p}}(\Delta[m];R)$ (or $\tN^*_{\ov{p}}(\Delta[m])$ if there is no ambiguity)
  the complex of $\ov{p}$-intersection cochains on $\Delta[m]$ and by
  $$\tN^*_{\ov{p}}\colon \Delta[\ttP]^+\to \mdg$$
  the associated functor.
 Finally, we extend it in a functor
 $$\tN^*_{\ov{p}}\colon \Delta[\ttP]\to \mdg,$$
 by setting $\tN^*(\Delta[m])=0$ if $\Delta[m]$ is a not regular simplex.
  \end{enumerate}
\end{definition}

\section{Blown up cochains of simplicial sets  over a poset} 
\label{sec:functors}

\begin{quote}
Let $R$ be a commutative ring with unit and $\mdg$  be the category of positively graded differential graded $R$-modules,
 with a differential of degree +1.
Let  $\ttP$  be a poset and $\ov{p}\colon \ttP\to\ov{\Z}$ a perversity.  
We define a pair of adjoint functors
$$\xymatrix{\ssetp&& \mdg \ar@<1ex>[ll]^(.46){\langle-\rangle_{\ov{p}}}
\ar@<1ex>[ll];[]^(.50){\tN^*_{\ov{p}}   }\\}
$$
between the categories of simplicial sets over  $\ttP$ and  $\mdg$ and prove
the existence of an extension of this adjunction to homotopy classes.
\end{quote}

\subsection{Construction of the two functors}\label{subsec:twofunctors}
Let $\ttP$ be a poset and $\ov{p}$ be a perversity on $\ttP$. 
 We undertake the presentation made by Bousfield and Gugenheim 
 for the Sullivan theory of  rational homotopy type, see \cite[Chapter 8]{MR0425956}.
 For any $\sigma\in\ttN(\ttP)$ and $k\in\N$, we set
 $$\ttM_{\ov{p}}(\sigma,k) =\tN^k_{\ov{p}}(\Delta[\sigma]).$$
We observe that $\ttM_{\ov{p}}(\bullet,\ast)$ is a simplicial differential graded module over $\ttN(\ttP)$.
Let $\Psi_{L}\in \ssetp$ and $M\in\mdg$.
First, we define a functor  $\tN^*_{\ov{p}}(-)\colon \ssetp\to \mdg$  by
$$\tN_{\ov{p}}^k(\Psi_{L})=\Hom_{\ssetp}(\Psi_{L},\ttM_{\ov{p}}(\bullet,k)).$$
In particular, if  $\Psi_{L}$ is the identity map on $\ttN(\ttP)$, denoted $\ttN(\ttP)$, we have
\begin{equation}\label{equa:tNNP}
\tN^*_{\ov{p}}(\ttN(\ttP))=\Hom_{\ssetp}(\ttN(\ttP),\tN^*_{\ov{p}}(\Delta[\bullet]),
\end{equation}
which associates to any $\sigma\in\ttN(\ttP)$ an element of $\tN^*_{\ov{p}}(\Delta[\sigma])$.

\medskip
Let $M\in\mdg$. The element  $\langle M\rangle_{\ov{p}}\in \ssetp$ is defined by the images of the 
$\sigma\in \Delta[\ttP]$ and we set
$$\langle M\rangle_{\ov{p}}(\sigma)=\Hom_{\mdg}(M,\ttM_{\ov{p}}(\sigma,*)).$$

\begin{proposition}\label{prop:adjoint}
The functors $\langle -\rangle_{\ov{p}}$ and $\tN^*_{\ov{p}}(-)$ are adjoint; i.e., for any 
 $\Psi_{L} \in \ssetp$ and $M\in\mdg$, there are bijections
$$\xymatrix@1{
 \Hom_{\ssetp}(\Psi_{L},\langle M\rangle_{\ov{p}})
  \ar@<1ex>[rr]^-{\beta}
  &&
  \Hom_{\mdg}(M,\tN^*_{\ov{p}}(\Psi_{L})).
 \ar@<1ex>[ll]^-{\alpha}_-{\cong}
 }
$$
\end{proposition}

\begin{proof}The two bijections can be written down explicitly as in \cite{MR0425956}.
Let $\sigma\in\ttN(\ttP)$, $k\in\N$, $x\in{L}_{\sigma}$, $w\in M^k$. 
We set
$\alpha(g)(x)(w)=g(w)(x)\in \ttM_{\ov{p}}(\sigma,k)$ and $\beta(f)(w)(x)=f(x)(w)$.
 \end{proof}
 
 This is an adjunction of contravariant functors. If we replace $\mdg$ by
 the opposite category $(\mdg)^{\op}$, these two functors give a pair of adjoint covariant functors, the functor
 corresponding to $\tN^*_{\ov{p}}$ being the left adjoint. Therefore, 
 \emph{$\tN^*_{\ov{p}}(-)$ 
 transforms inductive limits in $\ssetp$ in projective limits in $\mdg$.} This can also be seen directly  from the definition
 and the compatibility of the $\Hom$-functors with limits.
 
 \medskip
 Similarly, we define a functor  $\tN^{+,*}_{\ov{p}}\colon \ssetp^+\to \mdg$  by
$$\tN^{+,k}_{\ov{p}}(\Psi_{K})=\Hom_{\ssetp^+}(\Psi_{K},\ttM_{\ov{p}}(\bullet,k))$$
and a functor
$\langle -\rangle^+_{\ov{p}}\colon \mdg\to \ssetp^+$ by
$$\langle M\rangle^+_{\ov{p}}(\sigma)=\Hom_{\mdg}(M,\ttM_{\ov{p}}(\sigma,*)),$$
for each $\sigma\in \Delta[\ttP]^+$, $\Psi_{K}\in \ssetp^+$, $M\in\mdg$. We check easily 
\begin{equation}\label{equa:+or+}
\langle M\rangle_{\ov{p}}=\ttn\left(\langle M\rangle_{\ov{p}}^+\right)
\quad \text{and}\quad
\tN^*_{\ov{p}}(\Psi_{L})= (\tN^{+,*}_{\ov{p}}\circ \cR)(\Psi_{L}).
\end{equation}
 As above, 
these two functors are adjoint.

\begin{definition}\label{def:allowableboum}
Let $\Psi_{K}\in \ssetp^+$,  
the homology of $\tN^{+,*}_{\ov{p}}(\Psi_{K})$ is called the
 \emph{blown up intersection cohomology}  of $\Psi_{K}$, 
 for the perversity~$\ov{p}$, with coefficients in $R$,
 and denoted  $\crH_{\ov{p}}^{+,*}({\Psi_{K}};R)$.
 (If there is no ambiguity, we also denote it $\crH_{\ov{p}}^*({{\Psi_{K}}})$ or $\crH^*_{\ov{p}}(K)$.)

Similarly, for $\Psi_{L}\in \ssetp$, the homology of
$\tN^*_{\ov{p}}(\Psi_{L})$ is called the   \emph{blown up intersection cohomology}  of $\Psi_{L}$ 
and denoted $\crH_{\ov{p}}^*({\Psi_{L}};R)$.
 \end{definition}

Let $\Psi_{L}\in\ssetp$ and $\Psi_{K}\in\ssetp^+$. From
$\tN^*_{\ov{p}}(\Psi_{L};R)=\tN^{+,*}_{\ov{p}}(\cR(\Psi_{L});R)$
and
$\cR\circ \tti=\cR\circ \ttn=\id$, we deduce
\begin{equation}\label{equa:blowupcohomologyin}
\crH^*_{\ov{p}}(\Psi_{L})=\crH^{+,*}_{\ov{p}}(\cR(\Psi_{L}))\;\text{ and }\;
\crH^{+,*}_{\ov{p}}(\Psi_{K})\cong
\crH^*_{\ov{p}}(\tti(\Psi_{K}))\cong
\crH^*_{\ov{p}}(\ttn(\Psi_{K})).
\end{equation}

\subsection{Simplicial maps and blown up cohomology}\label{subsec:boomlesmaps}

Simplicial maps have a nice behavior with respect to blown up cohomology. 
We prove it in the general context of simplicial maps between simplicial sets over 
possibly different posets.
The following result uses the notion of pullback perversity of a perversity,
introduced in \defref{def:perversitypullback}.

 \begin{proposition}\label{prop:mapinterseccohomology}
 Let $\ttP$ and $\ttQ$ be two posets,
 $\ov{p}$, $\ov{q}$ be two perversities defined on $\ttP$ and $\ttQ$ respectively.
  We consider a commutative diagram of simplicial maps, 
 $$\xymatrix{
K\ar[r]^-{f}\ar[d]_{\Psi_{K}}&
L\ar[d]^{\Psi_{L}}\\
\ttN(\ttP)\ar[r]^-{\ttf}&\ttN(\ttQ).
}$$
We denote by $\tN^*_{\ov{q}}(\Psi_{L},\ttQ;R)$ and
$\tN^*_{\ov{p}}(\Psi_{K},\ttP;R)$
the blown up cochains corresponding to the perversities $\ov{q}$, $\ov{p}$ and the posets $\ttQ$, $\ttP$,
respectively.
If $\ov{p}\geq \ttf^*\ov{q}$, then $f$ induces a cochain map
$\tN^*_{\ov{q}}(\Psi_{L},\ttQ;R)
\to
\tN^*_{\ov{p}}(\Psi_{K},\ttP;R)$ defined by
$(f^*\omega)_{\sigma}=\omega_{f\circ\sigma}$.
Therefore, there is an induced homomorphism between the associated blown up cohomology.
\end{proposition}

\begin{proof}
The association $\omega\mapsto f^*\omega$ is compatible with the face operators,
$(d_{i}^*f^*(\omega))_{\sigma}=
d_{i}^*\omega_{f\circ \sigma}=\omega_{f\circ\sigma\circ d_{i}}=(f^*(\omega))_{\sigma\circ d_{i}}$
and similarly with the degeneracy operators. Moreover, if $\delta$ denotes the differentials, we have
$(\delta f^*(\omega))_{\sigma}=\delta( f^*(\omega)_{\sigma})=\delta\omega_{f\circ \sigma}
=(f^*(\delta\omega))_{\sigma}$. 
Compatibility with perversities remains to be taken into account.

\medskip
Let $\sigma\colon \Delta[n]\to K$ and $\omega\in \tN^*_{\ov{q}}(\Psi_{L},\ttQ;R)$.
The images of the simplex $\Delta[n]$ can be written
\begin{itemize}
\item $(\Psi_{K}\circ \sigma)(\Delta[n])=\tts_{1}\prec\dots\prec \tts_{k_{1}}\prec \tts_{k_{1}+1}\prec\dots\prec \tts_{k_{p}}$,
\item $(\Psi_{L}\circ f \circ \sigma)(\Delta[n])=\ttt_{1}\prec \ttt_{2}\prec \dots\prec \ttt_{p}$,
\end{itemize}
with $\ttf(\tts_{k_{i}+j})=\ttt_{i+1}$, for all $j\in\{1,\dots,k_{i+1}-k_{i}\}$. Therefore, by definition, we have
$$\|f^*(\omega_{\sigma})\|_{\tts_{k_{i}}}
=
\|\omega_{\sigma}\|_{\ttt_{i}}.
$$
We fix $i$ and $j$, and denote $\tts=\tts_{k_{i}+j}$ and $\ttt=\ttt_{i}$. Recall that, by hypothesis, we have
$\|\omega_{\sigma}\|_{\ttt}\leq \ov{q}(\ttt)$. This inegality and the hypotheses  imply,
$$
\|f^*(\omega_{\sigma})\|_{\tts}\leq \|f^*(\omega_{\sigma})\|_{\tts_{k_{i}}}=\|\omega_{\sigma}\|_{\ttt}
\leq \ov{q}(\ttt) =\ov{q}(\ttf(\tts))
\leq
(\ttf^*\ov{q})(\tts)\leq \ov{p}(\tts),
$$
and the $\ov{p}$-allowability of $f^*\omega$. We have proved that
$f^*\colon \tN^*_{\ov{q}}(\Psi_{L},\ttQ;R) \to \tN^*_{\ov{p}}(\Psi_{K},\ttP;R)$
is a cochain map.
\end{proof}

\subsection{Compatibility with the homotopy classes}
Recall that $\mdg$ has a closed model structure \cite[Section 2.3]{Hov} 
where weak-equivalences are quasi-isomorphisms and 
fibrations are surjective chain maps. The cofibrant objects are the 
cochain complexes of projective $R$-modules.
The rest of this section is devoted to the development of properties contributing to the proof of the 
following statement, which extends
the adjunction between $\langle -\rangle_{\ov{p}}$
and
$\tN^*_{\ov{p}}(-)$  to homotopy classes.

\begin{theoremb}\label{thm:isohomotopyclassesmdg}
Let $\Psi_{L}$ be an $s$-fibrant object of $\ssetp$, $\ov{p}\colon \ttP\to\ov{\Z}$  a perversity 
and $M\in \mdg$. Then, $\langle M\rangle_{\ov{p}}$ is s-fibrant and
the adjunction induces  an isomorphism between the homotopy classes,
$$[\Psi_{L},\langle M\rangle_{\ov{p}}]_{\ssetp} 
\cong [M,\tN^*_{\ov{p}}(\Psi_{L})]_{\mdg}.$$
\end{theoremb}

We keep the notation of this statement all along the rest of this section. 
The tensorisation $\Psi_{L} \otimes -$ makes reference to the tensor product 
 in $\ssetp$ of $\Psi_{L}$ with a simplicial set.
The proof consists  of a construction  of an ad'hoc path object in $\mdg$, see
\corref{cor:pathobjectmdg}.

\medskip
For $j=0\,,1$, we denote by $\iota_{j}\colon  \Psi_{L}\cong  \Psi_{L}\otimes \{j\}\to  \Psi_{L}\otimes\Delta[1]$  the canonical inclusion.

 \begin{proposition}\label{prop:epimorphism}
The following restriction morphism is a surjection,
$$\tN^*_{\ov{p}}(\iota_{0})+\tN^*_{\ov{p}}(\iota_{1})\colon \tN^*_{\ov{p}}( \Psi_{L}\otimes \Delta[1]) 
\rightarrow \tN^*_{\ov{p}}( \Psi_{L})\oplus \tN^*_{\ov{p}}( \Psi_{L}).$$
\end{proposition}

\begin{proof}
For constructing a section, we use the existence of a cup product at the cochain level 
(\cite[Proposition 4.2]{CST4}).
Let $\Delta[1]=[e_{0},e_{1}]$ and $\1_{e_{i}}\in N^0(\Delta[1])$ the 0-cochain taking the value 1 on $e_{i}$
and 0 otherwise, for $i=0,\,1$.
We denote by $\pi\colon  \Psi_{L}\otimes\Delta[1]\to  \Psi_{L}$ the canonical projection
and consider the morphism
$$\Phi\colon \tN^*_{\ov{p}}( \Psi_{L})\oplus \tN^*_{\ov{p}}( \Psi_{L})\to \tN^*_{\ov{p}}( \Psi_{L}\otimes \Delta[1]),$$
defined by 
$$\Phi(\omega_{0},\omega_{1})=\1_{e_{0}}\smile \tN^*_{\ov{p}}(\pi)(\omega_{0})+
\1_{e_{1}}\smile  \tN^*_{\ov{p}}(\pi)(\omega_{1}).
$$
We observe that 
$(\tN^*_{\ov{p}}(\iota_{0})+\tN^*_{\ov{p}}(\iota_{1})(\Phi(\omega_{0},\omega_{1}))=
(\tN^*_{\ov{p}}(\pi\circ \iota_{0})(\omega_{0}),\tN^*_{\ov{p}}(\pi\circ\iota_{1})(\omega_{1}))=
(\omega_{0},\omega_{1})$.  
\end{proof}

\begin{proposition}\label{prop:productwithDelta1}
 For $j=0,\,1$, the morphism
$\iota_{j}^*=\tN^*_{\ov{p}}(\iota_{j})\colon\tN^*_{\ov{p}}(\Psi_{L}\otimes \Delta[1])\longrightarrow \tN^*_{\ov{p}}(\Psi_{L})
$
is a trivial fibration in $\mdg$ and the morphisms $\iota_{0}$ and $\iota_{1}$
induce the same map in blown up cohomology. 
\end{proposition}

\begin{proof}
The surjectivity of $\tN^*_{\ov{p}}(\iota_{j})$ comes from \propref{prop:epimorphism}.
 Let $\pi \colon  \Psi_{L}\otimes\Delta[1]\to \Psi_{L}$ be the canonical projection. The equality $\pi\circ\iota_{j}=\id$ implies $\crH^*_{\ov{p}}(\iota_{j})\circ \crH^*_{\ov{p}}(\pi)=\id$. 
 Set $\psi_{j}=\iota_{j}\circ \pi\colon \Psi_{L}\otimes\Delta[1]\to \Psi_{L}\otimes\Delta[1]$. 
 (We denote also $\psi_{j}$ the underlying map from  ${L}\otimes\Delta[1]$ to
 $ {L}\otimes\Delta[1]$.)
 The first part of the statement is therefore reduced to the existence of a cochain homotopy, $\tG$, 
 between  the identity map and $\tN^*_{\ov{p}}(\psi_{j})$. 
 Such homotopy and the previous equality imply
$\crH^*_{\ov{p}}(\iota_{0})=(\crH^*_{\ov{p}}(\pi))^{-1}=\crH^*_{\ov{p}}(\iota_{1})$. 

We suppose $j=0$,  the case $j=1$ being similar. For simplicity, we denote $\psi=\psi_{0}$.
Let $\sigma\colon\Delta[{\bullet}]\to \Psi_{L}$ be a simplex. 
The map $\psi\colon  \Psi_{L}\otimes\Delta[1]\to  \Psi_{L}\otimes\Delta[1]$ is a collection of maps,
 $\psi_{\Delta[k]}\colon \Delta[k]\to\Delta[k]$, for any simplex $\Delta[k]$ of the product $\Delta[{\bullet}]\otimes \Delta[1]$. 
 Such map can be extending in,
$c\psi_{\Delta[k]}\colon c\Delta[k]\to c\Delta[k]$, by taking the identity on the cone point. 
We denote by $c\psi_{\Delta[k]}^*$ and $\psi_{\Delta[k]}^*$ the induced cochain morphisms. 

The blow up of a simplex of  $\Psi_{L}\otimes\Delta[1]$ is a face of the product 
$\ti{\Delta}=c\Delta[{j_{0}]}\times\cdots\times c\Delta[{j_{\ell-1}}]\times \Delta[{j_{\ell}}]$ 
where each $\Delta[{j_{i}}]$ is a simplex of a product $\Delta[{\bullet}]\times\Delta[1]$. 
By naturality, it is sufficient to define the homotopy $\tG$ at the level of $\ti{\Delta}$.
Denote by $G\colon N^{*}( {L}\otimes\Delta[1])\to N^{*}( {L}\otimes \Delta[1])$ 
the homotopy between the identity map and $N^{*}(\psi)$, 
i.e., we have $\delta G+G\delta=N^{*}(\psi)-\id$. 
The restriction of $G$ to a simplex $\Delta[j]$ of  $ {L}\otimes\Delta[1]$ is denoted $G_{\Delta[j]}$ 
and its extension to $c\Delta[j]$ by $cG_{\Delta[j]}$. 
Let us also denote  $\id_{\Delta[j]}$ the identity map on $N^*(\Delta[j])$.
For
$$\1_{(F,\varepsilon)}=
\1_{(F_{0},\varepsilon_{0})}\otimes\dots\otimes \1_{(F_{\ell-1},\varepsilon_{\ell-1})}\otimes \1_{F_{\ell}}\in \tN^*(\Delta[m]),
$$
we set $|(F,\varepsilon)|_{<j}=\sum_{i=0}^{j-1}|(F,\varepsilon_{i})|
$ and we define $\tG(\1_{(F,\varepsilon)})$ as
$$
\sum_{\ell=0}^m (-1)^{|(F,\varepsilon)|_{<j}} 
\1_{(\psi(F_{0}),\varepsilon_{0})}
\otimes\dots\otimes
\1_{(\psi(F_{j-1}),\varepsilon_{j-1})}
\otimes
G(\1_{(F_{j}),\varepsilon_{j}})
\otimes
\1_{(F_{j},\varepsilon_{j})}
\otimes\dots\otimes
\1_{F_{\ell}}
.$$
We prove that $\tG\delta+\delta \tG=\tN^{*}(\psi)-\id $ by induction on $\ell$, the perverse degree being taken in account
at the end of the proof.
It is true for $\ell=0$, by choice of $G$. 
Let us suppose it is true for $\ell-1$ and we prove it for $\ell$.
The element $\1_{(F,\varepsilon)}$ can be written as 
$\1_{(F_{0},\varepsilon_{0})}\otimes B$, where $B$ is a tensor product on which the induction hypothesis can be applied. 
By definition of $\tG$, we have
$$\tG(\1_{(F_{0},\varepsilon_{0})}\otimes B) =
G(\1_{(F_{0},\varepsilon_{0})})\otimes B+(-1)^{|(F_{0},\varepsilon_{0})|}\1_{(F_{0},\varepsilon_{0})}\otimes \tG(B).
$$
A computation, using
$G\delta+\delta G=N^{*}(\psi)-\id $ and the induction, gives
$\tG\delta+\delta \tG=\tN^{*}(\psi)-\id $.
Finally, by construction, the homotopy $\tG$ respects the perverse degree and is the required homotopy, $\tG\colon
\tN^{*}_{\ov{p}}( \Psi_{L}\otimes\Delta[1])\to \tN^{*}_{\ov{p}}( \Psi_{L}\otimes \Delta[1])$.
\end{proof}

\begin{corollary}\label{cor:pathobjectmdg}
The two injections $\iota_0,\,\iota_1$,
and the projection $\pi$
generate a path object in the category $\mdg$,
$$\tN^*_{\ov{p}}(\Psi_{L}){\xrightarrow[]{\pi^*}}
\tN^*_{\ov{p}}(\Psi_{L}\otimes \Delta[1])
\xrightarrow[]{\iota_0^*+\iota_1^*}
\tN^*_{\ov{p}}(\Psi_{L})\oplus \tN^*_{\ov{p}}(\Psi_{L}).$$
Thus if $\Phi_{1}\sim \Phi_{2}$ in $\ssetp$, then
$\tN_{\ov{p}}^*(\Phi_{1})$ and $\tN_{\ov{p}}^*(\Phi_{2})$ are homotopic in $\mdg$
and two homotopic maps in $\ssetp$ induce the same
map in blown up cohomology. In particular,  a homotopy equivalence induces an isomorphism. 
\end{corollary}

\begin{proof}
The composition $(\iota_{0}^*+\iota_{1}^*)\circ \pi^*$ is  the diagonal map
$\tN_{\ov{p}}(\Psi_{L})\to \tN_{\ov{p}}(\Psi_{L})\oplus \tN_{\ov{p}}(\Psi_{L})$. The statement is
thus a direct consequence of the definition of a path object and of Propositions 
\ref{prop:mapinterseccohomology}, \ref{prop:epimorphism}
and \ref{prop:productwithDelta1}.
\end{proof}

\begin{proof}[Proof of \thmref{thm:isohomotopyclassesmdg}]
For any $\Psi_{L}\in\ssetp$, the simplicial set
$\HomD_{\ssetp}(\Psi_{L},\langle M\rangle_{\ov{p}})$
is a simplicial group, therefore $\langle M\rangle_{\ov{p}}\in\ssetp$
is s-fibrant in the sense of \defref{def:sfibrant}.
Denote by 
$\cT\colon 
\Hom_{\ssetp}(\Psi_{L},\langle M\rangle_{\ov{p}})
\to
\Hom_{\mdg}(M,\tN^*_{\ov{p}}(\Psi_{L}))$
the bijection given by the adjunction. 
Let $\Phi_{1}$, $\Phi_{2}$ be two elements of  $\Hom_{\ssetp}(\Psi_{L},\langle M\rangle_{\ov{p}})$. We have to prove
$$\Phi_{1}\sim \Phi_{2}\text{ if, and only if, } \cT(\Phi_{1})\sim \cT(\Phi_{2}).$$
Suppose first $\Phi_{1}\sim \Phi_{2}$ and recall that, for $i=1,\,2$,  $\cT(\Phi_{i})$ is the composition
$$\xymatrix@1{
M\ar[r]&
\tN^*_{\ov{p}}(\langle M\rangle_{\ov{p}})\ar[rr]^-{\tN^*_{\ov{p}}(\Phi_{i})}&&
\tN^*_{\ov{p}}(\Psi_{L}).
}$$
Thus $\cT(\Phi_{1})\sim \cT(\Phi_{2})$ is a consequence of \corref{cor:pathobjectmdg}.

\medskip
Suppose now $\cT(\Phi_{1})\sim \cT(\Phi_{2})$. Then, for $i=1,\,2$, the map $\Phi_{i}$ is the composition
$$\xymatrix@1{
\Psi_{L}\ar[r]&
\langle \tN^*_{\ov{p}}(\Psi_{L})\rangle_{\ov{p}}
\ar[rr]^-{\langle \cT(\Phi_{i})\rangle_{\ov{p}}}&&
\langle M\rangle_{\ov{p}}.
}$$
The homotopy  $\cT(\Phi_{1})\sim \cT(\Phi_{2})$ consists of a map
$H\colon M\to \tN^*_{\ov{p}}(\Psi_{L}\otimes \Delta[1])\in\mdg$ whose projection to 
$\tN^*_{\ov{p}}(\Psi_{L})\oplus \tN^*_{\ov{p}}(\Psi_{L})$ is $\cT(\Phi_{1})+\cT(\Phi_{2})$.
Thus, by adjunction, we get a morphism in $\ssetp$,
$$\Psi_{L}\otimes \Delta[1]\to \langle \tN^*_{\ov{p}}(\Psi_{L}\otimes \Delta[1])\rangle_{\ov{p}}
\to \langle M\rangle_{\ov{p}},$$
which is a homotopy between $\Phi_{1}$ and $\Phi_{2}$.
\end{proof}

 We denote by $\Lambda^{k}[m]$ the $k^{\text{th}}$-horn, which is the subcomplex of 
 $\Delta[m]$ generated by all faces except the $k^{\text{th}}$ face.

 \begin{proposition}\label{prop:trivialfibration}
 Let $\Psi_{L}$ be an object of $\ssetp$ and $\ov{p}$  a perversity on $\ttP$.
 For any $m\geq 1$ and any $k$, $0\leq k\leq m$, the canonical inclusion 
 $\Lambda^{k}[m]\hookrightarrow \Delta[m]$ 
 induces a trivial fibration,
 $$\varphi\colon\tN_{\ov{p}}^*(\Psi_{L}\otimes \Delta[m])\to \tN_{\ov{p}}^*(\Psi_{L}\otimes \Lambda^{k}[m]).$$
 \end{proposition}

  \begin{proof}
  Let $M\in\mdg$ be cofibrant. We have to prove the existence of a dotted arrow
making commutative the following diagram in $\mdg$,
$$
\xymatrix{
M\ar[r]\ar@{-->}[rd]&
\tN^*_{\ov{p}}(\Psi_{L}\otimes\Lambda^{k}[m])\\
&\tN^*_{\ov{p}}(\Psi_{L}\otimes \Delta[m]).\ar[u]
} $$
 By adjunction, this is equivalent to the existence of a dotted arrow making commutative the following diagram in $\ssetp$,
 $$
\xymatrix{
\Psi_{L}\otimes \Lambda^{k}[m]\ar[d]\ar[r]&
\langle M\rangle_{\ov{p}}\\
\Psi_{L}\otimes \Delta[m].\ar@{-->}[ru]
} $$
From \defref{def:simplicialcat},  this last property amounts the existence of a dotted arrow, $g$,  
making commutative the following diagram in $\sset$
 $$
\xymatrix{
\Lambda^{k}[m]\ar[d]\ar[r]&
\HomD_{\ssetp}(\Psi_{L},\langle M\rangle_{\ov{p}})\\
\Delta[m]\ar@{-->}[ru]_-{g}
} $$
The morphism $g$ exists since $\HomD_{\ssetp}(\Psi_{L},\langle M\rangle_{\ov{p}})$ is a simplicial group, therefore
a Kan simplicial set.
  \end{proof}

\section{Perverse Eilenberg-MacLane  simplicial sets}\label{sec:perversEML}

\begin{quote}
Let $R$ be a commutative ring, $\ttP$  a poset and $\ov{p}\colon \ttP\to \ov{\Z}$  a perversity.
In this section, we show that the blown up cohomology is a representable functor on $\ssetp$.
If $\tkpn$  is the simplicial set over $\ttP$ representing
$\crH^n_{\ov{p}}(-;R)$, we  prove that the family 
$(\tkpn)_{n})$ is an infinite loop space in $\ssetp$.
We also introduce the cohomological operations on  the blown up cohomology and present some
basic properties. In the case of a GM-perversity on a pseudomanifold, they are 
cohomological operations on the Goresky-MacPherson hypercohomology of the Deligne's sheaves.
\end{quote}

\subsection{Representability}\label{subsec:representable}

In the classical case, a simplicial model of the Eilenberg-MacLane space $K(R,n)$ 
has for set of $k$-simplices the  $R$-module of cocycles of degree $n$ in $N^*(\Delta[k];R)$.
We follow a similar treatment remplacing the category of $\Delta[n]$'s by $\Delta[\ttP]$.

\medskip
Denote by $R(n)$ the object of $\mdg$, reduced to a free $R$-module  generated by one cocycle 
of degree $n$. Recall from\eqref{equa:tNNP}, that, for any ${\sigma}\in\ttN(\ttP)$,
we have
$$\langle R(n)\rangle_{\ov{p}}({\sigma})=\Hom_{\mdg}(R(n),\tN^*_{\ov{p}}(\Delta[{\sigma}];R))=
\cZ^n\tN^*_{\ov{p}}(\Delta[{\sigma}];R),$$
where $\cZ^n$ denotes the subspace of cocycles in degree $n$.
Therefore, we have
$$\langle R(n)\rangle_{\ov{p}}=
\cZ^n\tN_{\ov{p}}({\ttN(\ttP)};R).$$
This definition fits the case of restricted  simplicial sets over $\ttP$ and gives
$$\langle R(n)\rangle^+_{\ov{p}}=
\cZ^n\tN^+_{\ov{p}}(\cR({\ttN(\ttP)});R).$$

\begin{proposition}\label{prop:emlsset}
The functor  sending $\Psi_{L}\in \ssetp$ to  $\crH^n_{\ov{p}}({\Psi_{L}};R)$ is representable; i.e.,
for any s-fibrant $\Psi_{L}\in\ssetp$, we have an isomorphism,
$$\crH^n_{\ov{p}}(\Psi_{L};R)
\cong
[\Psi_{L},\langle R(n)\rangle_{\ov{p}}]_{\ssetp}.$$
Similarly, the functor  sending $\Psi_{K}\in \ssetp^+$ to  $\crH^n_{\ov{p}}({\Psi_{K}};R)$ is representable; i.e.,
for any s-fibrant $\Psi_{K}\in\ssetp^+$, we have an isomorphism,
$$\crH^n_{\ov{p}}(\Psi_{K};R)
\cong
[\Psi_{L},\langle R(n)\rangle^+_{\ov{p}}]_{\ssetp^+}.$$
\end{proposition}

\begin{proof}
This comes from the definition of $\crH^*_{\ov{p}}(-)$ and \thmref{thm:isohomotopyclassesmdg},
$$\crH^n_{\ov{p}}(\Psi_{L};R)=
H^n(\tN^*_{\ov{p}}(\Psi_{L}))=
[R(n),\tN^*_{\ov{p}}(\Psi_{L})]_{\mdg}
\cong
[\Psi_{L},\langle R(n)\rangle_{\ov{p}}]_{\ssetp}.$$
The second assertion follows by a similar argument.
\end{proof}

\begin{definition}\label{def:EMLsimplicialperverse}
The simplicial set over $\ttP$, $\langle R(n)\rangle_{\ov{p}}$, is called a
\emph{$\ov{p}$-perverse Eilenberg-MacLane  simplicial set}, henceforth $\emlp$-space,
and denoted $\tkpn$.
The \emph{$\ov{p}$-perverse Eilenberg-MacLane  restricted  simplicial set} (henceforth $\emlp^+$-space)
corresponds to $\langle R(n)\rangle^+_{\ov{p}}$ and is denoted $\tkpn^+$.
\end{definition}

The $\emlp$-spaces and their restricted  analogs are connected as follows.

\begin{proposition} The two $\ov{p}$-perverse Eilenberg-MacLane  spaces are connected by
$$\tkpn=\ttn(\tkpn^+)
\quad \text{and}\quad
\tkpn^+=\cR(\tkpn).$$
\end{proposition}

\begin{proof}
Recall from \eqref{equa:+or+}, the equality $\langle M\rangle_{\ov{p}}=\ttn\left(\langle M\rangle_{\ov{p}}^+\right)$, 
for any $M\in\mdg$. This gives immediatly $\tkpn=\ttn(\tkpn^+)$ from their definition. 
The second equality follows from $\cR\circ \ttn=\id$, see \propref{prop:leftrightR}.
\end{proof}

We now specify the $(n-1)$-skeleton of $\tkpn$.

\begin{proposition}\label{prop:skeletonEML}
For any perversity $\ov{p}$, the $(n-1)$-skeleta of $\tkpn$  
and $\ttN(\ttP)$ coincide.
\end{proposition}

\begin{proof} 
 If $\ell <n$, the $n$-simplices of $\Delta[\ell]$ are degenerate. As
 $N^*(\Delta[\ell])$ is the normalized complex, we have 
 $\cZ^n\tN_{\ov{p}}([\Delta[\ell])=0$.
Thus 
$\langle R(n)\rangle_{\ov{p}}({\sigma})=\Hom_{\mdg}(R(n),\tN^*_{\ov{p}}(\Delta[{\sigma}]))=\{0\}$,
for any $\sigma$ of dimension $<n$,
and
$\langle R(n)\rangle_{\ov{p}}^{(\ell)}=
\ttN(\ttP)^{(\ell)}$ if $\ell<n$.
\end{proof}

As in the classical algebraic topology situation, the family of $\ov{p}$-perverse Eilenberg-MacLane  spaces is
an infinite loop space. 
(We refer to  Subsection~\ref{subsec:infiniloop} for a reminder on infinite loop spaces
in a simplicial category.)
 First, we have to select a base point in $\tkpn=\langle R(n)\rangle_{\ov{p}}\in \ssetp$. 
 The final object of $\ssetp$ being the identity map on $\ttN(\ttP)$, such a basepoint  
is an element of
$$\Hom_{\ssetp}({\ttN(\ttP)},\langle R(n)\rangle_{\ov{p}})
\cong
\Hom_{\mdg}(R(n),\tN_{\ov{p}}({\ttN(\ttP)}))
\cong
\cZ^n \tN_{\ov{p}}({\ttN(\ttP)}).
$$
Therefore, we can choose  the map 
$\epsilon\colon \ttN(\ttP) \to \cZ^n \tN_{\ov{p}}({\ttN(\ttP)})$
constant on 0. 

\begin{remark}
The zero cocycle is a natural basepoint. Let us notice that, depending on the combinatorics of the poset $\ttP$, 
the Kan simplicial set $\cZ^n \tN_{\ov{p}}({\ttN(\ttP)})$ is not necessarily connected. But, as a simplicial group,
all its connected components are homotopy equivalent, in fact isomorphic.
\end{remark}

\begin{theoremb}\label{thm:EMLinfinitloop}
Let $\ov{p}\colon \ttP\to\ov{\Z}$ be a perversity on a poset $\ttP$ and $n\geq 0$.
The families of 
$\emlp$ and $\emlp^+$ simplicial sets are infinite loop spaces in the categories $\ssetp$ and $\ssetp^+$, respectively; 
i.e., there are weak $s$-equivalences,
$$\tkpp
\simeq
\Omega_{\epsilon} \tkpn
\quad \text{and}\quad
\tkpp^+
\simeq
\Omega_{\epsilon} \tkpn^+.
$$
\end{theoremb}

\begin{proof}
Let  $\Psi_{A}\colon A\to \ttN(\ttP)$ in $\ssetp$, pointed by $\epsilon$.
The loop space associated to $\epsilon$ being 
defined as a pullback (see \defref{def:pointed}), for any
$\Psi_{L}\in\ssetp$,
we have the following pullback, 
$$\xymatrix{
\Hom_{\ssetp}(\Psi_{L},\Omega_{\epsilon}\Psi_{A})
\ar[r]\ar[d]&
\Hom_{\ssetp}(\Psi_{L},\Psi_{A}^{\Delta[1]})
\ar[d]\\
\Hom_{\ssetp}(\Psi_{L},\ttN(\ttP))
\ar[r]&
\Hom_{\ssetp}(\Psi_{L},\Psi_{A}^{\partial \Delta[1]}).
}$$
If $\Psi_{A}=\tkpn$ and $\epsilon=0$, we deduce from
\defref{def:simplicialcat} that
$$\Hom_{\ssetp}(\Psi_{L},\Omega_{\epsilon}\,\tkpn=
\Ker (\cZ^n\tN^*_{\ov{p}}(\Psi_{L}\otimes \Delta[1])\to \cZ^n\tN^*_{\ov{p}}(\Psi_{L}\otimes\partial \Delta[1]))
$$
and
\begin{equation}\label{equa:petitesuite}
[\Psi_{L},\Omega_{\epsilon}\,\tkpn]\cong
\crH^n(\Psi_{L}\otimes \Delta[1],\Psi_{L}\otimes\partial \Delta[1]).
\end{equation}
For the determination of this relative cohomology, we consider the Ker-Coker exact sequence applied to
the following morphism of short exact sequences,
$$
\scriptsize{\xymatrix{
0\ar[r]&
\tN^*_{\ov{p}}(\Psi_{L}\otimes \Delta[1], \Psi_{L}\otimes \partial \Delta[1])
\ar[r]\ar[d]_{{\mu}}&
\tN^*_{\ov{p}}(\Psi_{L}\otimes \Delta[1])
\ar[r]^-{\iota_{0}^*+\iota_{1}^*}\ar@{=}[d]&
\tN^*_{\ov{p}}(\Psi_{L})\oplus \tN^*_{\ov{p}}(\Psi_{L})
\ar[r]\ar@{->>}[d]^{\nu}&0\\
0\ar[r]&
\Ker\ar[r]&
\tN^*_{\ov{p}}(\Psi_{L}\otimes \Delta[1])\ar[r]&
\tN^*_{\ov{p}}(\Psi_{L})\ar[r]^-{\iota_{0}^*}&0,
}}$$
where $\nu$ is the projection on the first factor.
We obtain an isomorphism between $\tN^*_{\ov{p}}(\Psi_{L})$
and the cokernel of $\mu$. 
The acyclicity of $\Ker$  (see \propref{prop:productwithDelta1}) gives
an isomorphism
\begin{equation}\label{equa:boumrelatif}
\crH^*_{\ov{p}}(\Psi_{L})
\xrightarrow{\cong}
 H^{*+1}(\tN^*_{\ov{p}}(\Psi_{L}\otimes \Delta[1], \Psi_{L}\otimes \partial \Delta[1]))
=\crH^{*+1}_{\ov{p}}(\Psi_{L}\otimes \Delta[1], \Psi_{L}\otimes \partial \Delta[1]).
\end{equation}
From \propref{prop:emlsset} and the isomorphisms \eqref{equa:boumrelatif}, \eqref{equa:petitesuite}, we deduce
\begin{eqnarray*}
[\Psi_{L},\tkpp]
&\cong&
\crH^{n-1}_{\ov{p}}(\Psi_{L})\\
&\cong&
\crH^{n}_{\ov{p}}(\Psi_{L}\otimes \Delta[1], \Psi_{L}\otimes \partial \Delta[1]))\\
&\cong&
[\Psi_{L},\Omega_{\epsilon}\,\tkpn].
\end{eqnarray*}
The result follows from the Yoneda lemma. The proof of the second assertion is similar.
\end{proof}

We deduce from \thmref{thm:EMLinfinitloop} an isomorphism involving classical 
Eilenberg-MacLane spaces and perverse ones.

\begin{corollary}\label{cor:productEML}
Let $\Psi_{L}\in \ssetp$, $\ov{p}\colon \ttP\to\ov{\Z}$ be a perversity and $n\geq 0$. Then the
simplicial set $\Hom_{\ssetp}^{\Delta}(\Psi_{L},\tkpn)$ is a product of 
Eilenberg-MacLane spaces,
$$\Hom_{\ssetp}^{\Delta}(\Psi_{L},\tkpn)\cong \prod_{j}K(\crH^{n-j}_{\ov{p}}(\Psi_{L};R),j).$$
\end{corollary}

\begin{proof}
The simplicial set
 $Z=\Hom_{\ssetp}^{\Delta}(\Psi_{L},\tkpn)$
 is a  simplicial abelian group therefore a product of Eilenberg-MacLane spaces. 
 We are thus reduced to the determination of its homotopy groups. 
 We first have
$$\pi_{0}(Z)= [\Psi_{L},\tkpn]_{\ssetp}\cong \crH^{n}_{\ov{p}}(\Psi_{L};R).$$
The rest follows by induction from
$\tkpp 
\simeq
\,\Omega_{\epsilon}\tkpn$.
\end{proof}

Let $\psi_{X}\colon X\to \ttP$ be a topological pseudomanifold of poset of strata, $\ttP$.
For any perversity, $\ov{p}\colon \ttP\to \ov{\Z}$, and any commutative ring $R$, 
we denote  $\bQ_{\ov{p}}$ the Deligne's sheaf 
 introduced in \cite{GM2} and $\H^*(X;\bQ_{\ov{p}})$ its associated hypercohomology groups,
 which coincide with the original groups introduced in \cite{GM1}.
 (For general perversities as those we are using here, we refer to \cite{FR2}.)
 Let  $\singp \psi_{X}$ be the simplicial set over the poset of strata, $\ttP$,  introduced in \eqref{equa:filteredpullback}.
 In \propref{prop:deltaornot}, we prove the existence of an isomorphism,
 $$
 \crH_{\ov{p}}^*(\singp \psi_{X};R)
 \cong
 \crH_{\ov{p}}^{\Delta,*}(\cO_{\ttP}(\singp \psi_{X});R).
 $$
In \cite{CST5}, we denote $ \crH_{\ov{p}}^{\Delta,*}(\cO_{\ttP}(\singp \psi_{X});R)$ simply by
$\crH_{\ov{p}}^*(X;R)$ and prove in \cite[Theorem A]{CST5}
that it is isomorphic to Deligne's hypercohomology,
$$
\H^*(X;\bQ_{\ov{p}})
\cong
\crH_{\ov{p}}^*(X;R).
$$
We thus have an isomorphism,
\begin{equation}\label{equa:deligneforever}
 \crH_{\ov{p}}^*(\singp \psi_{X};R)
 \cong
\H^*(X;\bQ_{\ov{p}}),
\end{equation}
and \corref{cor:productEML} implies the following result.
\begin{corollary}
Let $\psi_{X}\colon X\to \ttP$ be a topological pseudomanifold of poset of strata, $\ttP$, and
 $\ov{p}\colon \ttP\to \ov{\Z}$ be a perversity. Then there is an isomorphism of simplicial sets,
$$\Hom_{\ssetp}^{\Delta}(\singp \psi_{X},\tkpn)\cong \prod_{j}K(\H^*(X;\bQ_{\ov{p}}),j).$$
\end{corollary}

Moreover, in the case of pseudomanifolds, the isomorphism \eqref{equa:deligneforever} identifies the cohomological
operations, developed in the next subsection and in \secref{sec:operations},
with  operations on Deligne's hypercohomology.

\subsection{Perverse cohomological operations}\label{subsec:operations}

Let us  define cohomological operations on intersection cohomology, reproducing  the classical case.

\begin{definition}\label{def:operationperverse}
Let $\ov{p},\,\ov{q}\colon \ttP\to\ov{\Z}$ be two perversities on a poset $\ttP$ and $n$, $m$ be two integers.
A \emph{perverse cohomological operation} of type $(\ov{p},n,\ov{q},m)$
is a natural transformation
between the  functors
$\crH_{\ov{p}}^n(-;R)$ and $\crH_{\ov{q}}^m(-;R)$,
from $\ssetp$ to the category of $R$-modules.
We denote $\natr(\crH^n_{\ov{p}},\crH^m_{\ov{q}})$
the set of perverse cohomological operations of this type. 
\end{definition}
 
Cohomological operations on intersection cohomology are also determined by the
perverse Eilenberg-MacLane  spaces.

\begin{proposition}\label{prop:emloperation}
Let $\ov{p},\,\ov{q}\colon \ttP\to\ov{\Z}$ be two perversities on a poset $\ttP$ and $n$, $m$ be two integers.
There is an isomorphism
$$\natr(\crH^n_{\ov{p}},\crH^m_{\ov{q}})=
\left[ \tkpn, \tkqm \right]_{\ssetp}
=\crH_{\ov{q}}^m(\tkpn;R).$$
\end{proposition}

\begin{proof}
This is a direct consequence of the  Yoneda lemma
and the representablity of $\crH_{\bullet}^*(-;R)$ 
established in \propref{prop:emlsset} .
\end{proof}

The following result has to be compared with the classical fact that $K(R,n)$
is $(n-1)$-connected.

 \begin{proposition}\label{prop:operationmpetit}
 Let
 $\ov{p}$, $\ov{q}$ be two perversities on the poset $\ttP=\N^\op$. If $0<m<n$, we have
$$\natr(\crH^n_{\ov{q}},\crH^{m}_{\ov{p}})=
\crH_{\ov{p}}^m(\tkqn;R)=0.$$
\end{proposition}

\begin{proof}
The simplicial set  $\ttN(\N^{\op})$ is a cone on an acyclic simplicial set.
Therefore, its reduced intersection cohomology is 0 for any perversity
and the result follows from \propref{prop:skeletonEML}.
\end{proof}

\begin{example} From \cite{CST6,CST4,MR761809,MR1014465}, 
we already know some perverse cohomological operations.
Let ${\sigma}=\tts_{0}^{[q_{0}]}\prec\dots\prec \tts_{\ell}^{[q_{\ell}]}\in \ttN(\ttP)$ and recall   that 
$\langle R(n)\rangle_{\ov{p}}[{\sigma}]
 =
 \cZ^n\tN_{\ov{p}}({\sigma})=\cZ^n\tN_{\ov{p}}(\Delta[q_0]\ast\cdots\ast\Delta[q_\ell])$,
 where $\cZ$ denotes the subspace of cocycles.
\begin{enumerate}
\item For any commutative ring, $R$, there exists a square map,
$\tN_{\ov{p}}^i({\sigma})
\to
\tN_{2\ov{p}}^{2i}({\sigma})$,
coming from the cup product established in \cite[Proposition 4.2]{CST4}.
This  map gives an element in 
$\natr(\crH^n_{\ov{p}},\crH^{2n}_{2\ov{p}})$.
\item Let $R=\Z_{2}$ and $\cE(2)$ be the normalized homogeneous bar resolution of the symmetric group $\Sigma_{2}$.
In \cite[Theorem A]{CST6}, 
we establish the existence of an $\cE(2)$-algebra structure on $\tN^*({\sigma})$,
inducing $\Sigma_{2}$-equivariant cochain maps,
$\Phi\colon\cE(2)\otimes \tN^*_{\ov{p}}({\sigma})\otimes \tN^*_{\ov{q}}({\sigma})  \to 
\tN^*_{\ov{p}+\ov{q}}({\sigma})$.
In particular, setting $\psi(e_{i}\otimes x\otimes y)=x\smile_{i}y$, the square maps verify
$\|x\smile_{|x|-i}x\|\leq \ov{p}+i$, see \cite[Proof of Theorem B]{CST6}. 
With the notation $\cL(\ov{p})=\min(2\ov{p},\ov{p}+i)$, we prove in \cite{CST6}
that the Steenrod perverse squares
$\sq^i\in \nat_{\Z_{2}}(\crH^n_{\ov{p}},\crH^{n+i}_{\cL(\ov{p})})$,
as conjectured in \cite{MR761809} and \cite{MR1014465}.
\end{enumerate}
\end{example}

\section{Examples of operations in blown up cohomology}\label{sec:operations}

\begin{quote}
In this section, we suppose that $R$ is a Dedekind domain and we   focus on the  case 
of Goresky and MacPherson perversities  for singular spaces with one singular stratum.
We show that  the  perverse Eilenberg-MacLane  spaces are Joyal's projective cone over the classical Eilenberg-MacLane
spaces. We also determine them for the perversities $\ov{\infty}$ and $\ov{0}$,
as well as  some sets of perverse cohomological operations.
\end{quote}

Depth one singular spaces are singular spaces with one regular stratum and one singular stratum.
They  correspond to the poset $\ttP=[1]=\{0,1\}$ and  $\ttN(\ttP)=\Delta[1]$.
An object of $\ssets$ can be described as a family of sets $K_{k,\ell}$, with 
\begin{equation}\label{equa:emldelta1}
(k,\ell)\in \N^2\cup \left(\N\times\{-1\}\right)\cup \left(\{-1\}\times \N\right),
\end{equation}
the value -1 corresponding to an empty subset.
The objects of $\ssetp^+$ are the objects of $\ssetp$ such that $\ell\neq -1$; i.e.,
with
$(k,\ell)\in \N^2\cup \left(\{-1\}\times \N\right)$.
A \emph{completely regular simplicial set} over $[1]$ verifies $K_{k,\ell}=\emptyset$ if $k\geq 0$.
A perversity on $\ttP=[1]$ reduces to an element $\lambda\in \ov{\Z}$ corresponding to its value on 0, and
 is  denoted $\ov{\lambda}$.

\subsection{Perverse Eilenberg-MacLane  spaces as Joyal cylinders}\label{subsec:joyal}
Simplicial sets over $[1]$ have been studied by A. Joyal, see \cite[Section 7]{joyalbarcelona}. 
Let us first recall his presentation.

\medskip
A simplicial map $\Psi_{K}\colon K\to \Delta[1]$ is called \emph{a simplicial cylinder.} The simplicial sets
$K(0)$ and $K(1)$, defined by the following pullback
$$\xymatrix{
K(0)\sqcup K(1)\ar[r]\ar[d]
&
K\ar[d]^-{\Psi_{K}}\\
\partial\Delta[1]\ar[r]
&
\Delta[1],
}$$
generate a functor $i^*\colon \ssets\to \sset\times\sset$, sending $\Psi_{K}$ to $(K(0),K(1))$.
The simplicial sets $K(1)$ and $K(0)$ are respectively called the \emph{base} and the \emph{cobase} of 
the cylinder $\Psi_{K}$.

\medskip
The functor $i^*$ admits a left adjoint  $i_{*}\colon \sset\times\sset\to \ssets$ 
which sends $(A,B)$ to the composite
$\Psi_{A\sqcup B}\colon A\sqcup B\to \partial\Delta[1]\to\Delta[1]$, 
\begin{equation}\label{equa:agauche}
\Hom_{\sset\times\sset}((A,B),i^*(\Psi_{K}))\cong \Hom_{\sset_{[1]}}(\Psi_{A\sqcup B},\Psi_{K}).
\end{equation}

\medskip
For any pair $(A,B)\in \sset\times\sset$, we can define their join, $A\ast B$
(see \cite[\& 3.1]{joyalbarcelona}) as a simplicial set
whose sets of simplices are given by 
 $$(A\ast B)_{k}=\left\{
 \begin{array}{ccl}
 \sigma&\text{with}&\sigma\in A_{k},  \quad \text{denoted}\quad (\sigma,\emptyset),\\
 \tau&\text{with}&\tau\in B_{k},  \quad \text{denoted}\quad (\emptyset,\tau),\\
 (\sigma,\tau)&\text{with}&\sigma\in A_{i},\;\tau\in B_{j}\quad \text{and}\quad i+j+1=k.
 \end{array}\right.$$
 Then, we have two canonical maps, the injection $\iota_{A\sqcup B}\colon A\sqcup B\to A\ast B$, and
 the surjection
 $\Psi_{A\ast B}\colon A\ast B\to \Delta[1]$ which is the join of $A\to\Delta[0]$ and
 $B\to\Delta[0]$. The join construction gives a right adjoint $i_{!}\colon \sset\times\sset\to \ssets$ 
 sending  $(A,B)$ to  $\Psi_{A\ast B}$, 
\begin{equation}\label{equa:adroite}
\Hom_{\sset\times\sset}(i^*(\Psi_{\psi_{K}}),(A,B))
\cong
\Hom_{\sset_{[1]}}(\Psi_{K},\Psi_{A\ast B})
.
\end{equation}
In particular, for each cylinder $\Psi_{K}$, there are canonical maps
$$\Psi_{K(0)\sqcup K(1)}\to \Psi_{K}\to \Psi_{K(0)\ast K(1)}.$$
Denote by $\cC(A,B)$ the set of cylinders $\Psi_{K}$ with cobase $K(0)=A$ and base $K(1)=B$.
Any $\Psi_{K}\in \cC(A,B)$ is characterized by a simplicial map
$\Psi_{K}\to \Psi_{A\ast B}$ such that $A\sqcup B\to K\to A\ast B$ is the canonical inclusion.
Joyal also proves that cylinders coincide with presheaves over the category product of the category 
$\Delta[A]$ of simplices of 
$A$ by the category $\Delta[B]$ of simplices of $B$. 
A cylinder whose base is $B$ and cobase is a point is called a \emph{projective cone over $B$} in \cite{joyalbarcelona}. 

\medskip
If $\ttP=[1]$, any perverse Eilenberg-MacLane  space is a cylinder \`a la Joyal.
We first determine their base and cobase. By \defref{def:nonsingular}, the base
is the non singular part.

 \begin{proposition}\label{prop:regularpartEML}
 For any perversity, a perverse Eilenberg-MacLane  space over the poset $\ttP=[1]$ is a projective cone 
 over the classical Eilenberg-MacLane space.
 \end{proposition}
 
 \begin{proof}
 With the notation of \eqref{equa:emldelta1}, the non singular part of $\ttK(R,n,[1],\ov{p})$ is obtained from the simplices 
 $(-1, i_{1})$, with $i_{1}\neq -1$, and thus
 $$
 \ttK(R,n,[1],\ov{p})( \emptyset\ast\Delta[i_{1}]) =
  \cZ^n(N(\tc\emptyset)\otimes N(\Delta[i_{1}])=
  \cZ^n(N(\Delta[i_{1}])).
  $$
The result follows from the simplicial definition of  Eilenberg-MacLane spaces.

The cobase consists of non regular simplices in the sense of \defref{def:simplexregular}. 
The result follows from the
equality $\tN^*(\Delta[m])=0$ for non regular simplices, see \subsecref{subsec:perversedegree}.
 \end{proof}

Base and cobase do not depend on the perversity; we show now that it is the same for the $n$-skeleta
of perverse Eilenberg-MacLane  spaces.

\begin{proposition}\label{prop:nskeleton} 
The $n$-skeleton of $\ttK(R,n,[1],\ov{p})$ 
does not depend on the perversity $\ov{p}$; i.e., for any perversity $\ov{p}$, we have
$$(\ttK(R,n,[1],\ov{p}))^{(n)}=(\ttK(R,n,[1],\ov{\infty}))^{(n)}.$$
\end{proposition}

\begin{proof} 
By definition,
an $n$-simplex of $\ttK(R,n,[1],\ov{p})$ over $\Delta[a]\ast\Delta[b]$ is  a linear combination of $n$-cochains
of the shape $\1_{\tc\Delta[a]\otimes\Delta[b]}$ with $a+b+1=n$. For dimensional reasons, these cochains must
``contain'' the apex of $\tc \Delta[a]$. 
Thus, by \defref{def:degreperversboum}, they are of perverse degree $-\infty$
and the result follows.
\end{proof}

 \subsection{The perversity $\ov{\infty}$}\label{subsec:loinloin}
 We show that the classifying space for the  infinite perversity is the final element of the 
 Joyal cylinder $\cC(\Delta[0],K(R,n))$.

\begin{proposition}\label{prop:EMLinfini}
There is a homotopy equivalence,
$\ttK(R,n,[1],\ov{\infty})\simeq\Delta[0]\ast K(R,n)$.
\end{proposition}

\begin{proof}
Let $\Psi_{K}\in\ssets$.
The isomorphism \eqref{equa:adroite} with $A=\Delta[0]$ and $B=K(R,n)$ implies
$$
\Hom_{\sset}(K(1),K(R,n))
\cong
\Hom_{\sset_{[1]}}(\Psi_{K},\Psi_{\Delta[0]\ast K(R,n)}).
$$
From
$(\Psi_{K}\otimes \Delta[1])(1)= K(1)\otimes \Delta[1]$,
we deduce also an isomorphism between the sets of homotopy classes,
$$[K(1), K(R,n)]_{\sset}\cong 
[\Psi_{K},\Psi_{\Delta[0]\ast K(R,n)}]_{\ssets}.$$
Recall that $K(1)$ is the completely regular part of $\Psi_{K}$. 
Thus, together with \corref{cor:normal} and the representability of cohomology, we have:
$$\crH^n_{\ov{\infty}}(\Psi_{K})\cong H^n(K(1))\cong [K(1),K(R,n)]_{\sset}
\cong
[\Psi_{K},\Psi_{\Delta[0]\ast K(R,n)}]_{\ssets}.
$$
From the uniqueness up to homotopy equivalence of the representing object, we deduce
$$\ttK(R,n,[1],\ov{\infty})\simeq \Psi_{\Delta[0]\ast K(R,n)}.$$
\end{proof}

\begin{corollary}\label{cor:ntkpn}
For any perversity $\ov{p}$, we have 
$\crH_{\ov{\infty}}^n(\ttK(R,n,[1],\ov{p});R)=R
$.
\end{corollary}

\begin{proof}
From Propositions \ref{prop:regularpartEML} and \ref{prop:EMLinfini}, we deduce:
$$\crH_{\ov{\infty}}^n(\ttK(R,n,[1],\ov{p});R)\cong H^n(\ttK(R,n,[1],\ov{p})(1))\cong H^n(K(R,n))=R.$$
\end{proof}

In \propref{prop:operationmpetit}, we prove the nullity of
$\crH^m_{\ov{p}}(\tkqn;R)$
for $m<n$. Let us study now the $n$-cohomology of the $n$-skeleton of $\ttK(R,n,[1],\ov{q})$.

\begin{proposition}\label{prop:kskeletonofcone}
Let  $\ov{p}$, $\ov{q}$ be two perversities on $\ttP=[1]$.
The $\ov{p}$-blown up cohomology in degree $n$ of the $n$-th skeleton, 
$(\ttK(R,n,[1],\ov{q}))^{(n)}$, does not depend on the perversities $\ov{p}$ and $\ov{q}$.
More precisely,   we have
$$\crH^n_{\ov{p}}((\ttK(R,n,[1],\ov{q}))^{(n)};R)
\cong
H^n(K(R,n)^{(n)};R).
$$
\end{proposition}

\begin{proof}
With \propref{prop:nskeleton}, we may choose $\ov{q}=\ov{\infty}$.
We know that $\ttK(R,n,[1],\ov{\infty})= \Delta[0]\ast K(R,n)$ and that the 
$(n-1)$-skeleton of $K(R,n)$ is contractible. By using the exact sequence associated to the pair
$((\Delta[0]\ast K(R,n))^{(n)}, \Delta[0]\ast(K(R,n)^{(n-1)}))$, we get
$$\crH^n_{\ov{p}}((\Delta[0]\ast K(R,n))^{(n)})
\cong 
\crH^n_{\ov{p}}((\Delta[0]\ast K(R,n))^{(n)}, \Delta[0]\ast (K(R,n)^{(n-1)})).
$$
 The excision of the apex implies
$$\crH^n_{\ov{p}}((\Delta[0]\ast K(R,n))^{(n)})\cong H^n(K(R,n)^{(n)},K(R,n)^{(n-1)})\cong H^n(K(R,n)^{(n)}).$$
\end{proof}

We determine the set of perverse cohomology operations keeping the same cohomological degree 
for two perversities, $\ov{p}$ and $\ov{q}$.
The result depends on the respective situation of $\ov{p}$ and $\ov{q}$ in $\ov{\Z}$.
We first establish a lemma.

\begin{lemma}\label{lem:twoinclusions}
Let   $\ov{p},\,\ov{q}\in \ov{\Z}$ be two perversities on $\ttP=[1]$. 
If  $\ov{q}\leq \ov{p}$, then the inclusion 
 $\ttK(R,n,[1],\ov{q})
 \hookrightarrow 
\ttK(R,n,[1],\ov{p})$
 and the natural transformation $\crH_{\ov{p}}^*\to \crH_{\infty}^*$
 induce injective maps:
\begin{equation*} 
\crH_{\ov{p}}^n(\ttK(R,n,[1],\ov{p});R)
\hookrightarrow 
\crH_{\ov{p}}^n(\ttK(R,n,[1],\ov{q});R)
\hookrightarrow
\crH_{\ov{\infty}}^n(\ttK(R,n,[1],\ov{q});R)= R.
\end{equation*}
\end{lemma}

\begin{proof}
 This is  a consequence of \corref{cor:ntkpn} and of the  commutative diagram,
\begin{equation}\label{equa:tkpntkqn}
\xymatrix{
\crH_{\ov{p}}^n(\ttK(R,n,[1],\ov{p});R)
\ar@{^{(}->}[r]
\ar@{-->}[d]
&
\crH_{\ov{p}}^n((\ttK(R,n,[1],\ov{p}))^{(n)};R)
\ar[d]^-{\cong}
\\
\crH_{\ov{p}}^n(\ttK(R,n,[1],\ov{q});R)
\ar@{^{(}->}[r]
\ar@{-->}[d]
&
\crH_{\ov{p}}^n((\ttK(R,n,[1],\ov{q}))^{(n)}\,;R)
\ar[d]^-{\cong}\\
\crH_{\ov{\infty}}^n(\ttK(R,n,[1],\ov{q});R)=R
\ar@{^{(}->}[r]
&
\crH_{\ov{\infty}}^n((\ttK(R,n,[1],\ov{q}))^{(n)}\,;R),\\
}\end{equation}
where the isomorphisms of the right-hand column come from \propref{prop:kskeletonofcone} .
\end{proof}

\begin{proposition}\label{prop:kskeleton}
Let  $\ov{p},\,\ov{q}\in \ov{\Z}$ be two perversities on $\ttP=[1]$. 
If $\ov{q}\leq \ov{p}$, we have, 
$$\natr(\crH^n_{\ov{q}},\crH^{n}_{\ov{p}})=
[\ttK(R,n,[1],\ov{q}), \ttK(R,n,[1],\ov{p})] = 
\crH^{n}_{\ov{p}}(\ttK(R,n,[1],\ov{q});R)=R.$$ 
\end{proposition}

\begin{proof}
In  the left-hand column of \eqref{equa:tkpntkqn} with  $\ov{p}=\ov{q}$, the
 unit $1\in  \crH_{\ov{\infty}}^n(\ttK(R,n,[1],\ov{p});R)=R$ 
 corresponds to the homotopy class of the canonical inclusion 
 $\ttK(R,n,[1],\ov{p}) 
 \hookrightarrow
 \ttK(R,n,[1],\ov{\infty})$. 
 This is the image of the homotopy class of the identity map on 
 $\ttK(R,n,[1],\ov{p})$  
 by the natural transformation
 $
 \crH^n_{\ov{p}}(-)=[-,\ttK(R,n,[1],\ov{p})] 
 \to
 \crH^n_{\ov{\infty}}(-)=[-,\ttK(R,n,[1],\ov{\infty})] 
 $. Thus the unit $1$ is reached and
 $\crH^{n}_{\ov{p}}(\ttK(R,n,[1],\ov{p});R)=R$. 
 The result follows now from \lemref{lem:twoinclusions}.
\end{proof}

In contrast with \propref{prop:kskeleton}, we have:

\begin{proposition}\label{prop:emlk}
Let  $\ov{p},\,\ov{q}\in \ov{\Z}$ be two perversities on $\ttP=[1]$. 
If $\ov{p}< \ov{q}<\ov{\infty}$, then, for any $n>0$, we have
$$\natr(\crH^n_{\ov{q}},\crH^{n}_{\ov{p}})=
[\ttK(R,n,[1],\ov{q}),\ttK(R,n,[1],\ov{p})]= 
\crH^{n}_{\ov{p}}(\ttK(R,n,[1],\ov{q});R)=0.$$ 
\end{proposition}

\begin{proof}
Let   $[\omega]\in \crH^n_{\ov{p}}(\ttK(R,n,[1],\ov{q});R)$.  
Its image in 
$\crH_{\ov{\infty}}^k(\ttK(R,n,[1],\ov{q});R)=R$  
is an element $\lambda\in R$.
Consider $X=S^{n-{q}}\times \tc S^{{q}}$ and translate $[\omega]$ and $\lambda$ in terms of maps between the
cohomology groups. The commutativity of the diagram
$$\xymatrix{
\crH^n_{\ov{q}}(X)=R\ar[rr]^-{[\omega]_{X}}\ar[d]^{\cong}&&
\crH^n_{\ov{p}}(X)=0\ar[d]\\
\crH_{\ov{\infty}}^n(X)=R\ar[rr]^-{\lambda}&&
\crH^n_{\ov{\infty}}(X)=R
}$$
implies $\lambda=0$. The  injectivity of
$\crH^n_{\ov{p}}(\ttK(R,n,[1],\ov{q});R) 
\to
\crH_{\ov{\infty}}^n(\ttK(R,n,[1],\ov{q});R)$ 
established in \lemref{lem:twoinclusions} gives the conclusion.
\end{proof}

\subsection{The perversity $\ov{0}$}\label{subsec:nul}

We determine the classifying space associated to the zero perversity.

 \begin{proposition}\label{prop:eml0}
 The perverse Eilenberg-MacLane  space $\mathtt{K}(R,n,[1],\ov{0})$ 
 is homotopically equivalent to
 $$(\ttn\circ\cR)(\Delta[1]\times  K(R,n) )=(\Delta[1]\times K(R,n))/(\Delta[0]\times  K(R,n)).$$
 \end{proposition}
 
 On the subcategory of $\ssets$ formed of normal simplicial sets over $[1]$, the
 classifying space of the blown up cohomology reduces to  $\Delta[1]\times K(R,n)$.
 
 \begin{proof}
 Let $\Psi_{K}\in\ssets$.
 From \eqref{equa:blowupcohomologyin}, \propref{cor:normal} and \cite[Propositions 1.57 and 1.60]{CST1}, we get
 $$\crH_{\ov{0}}^n(\Psi_{K};R)\cong 
 \crH_{\ov{0}}^n((\tti\circ\cR)(\Psi_{K});R)\cong H^n(\cU(\tti\circ\cR)(\Psi_{K});R),$$
 where $\cU\colon \ssetp\to \sset$ is the forgetful functor. 
 Therefore, from the representability of the cohomology in $\sset$, \propref{prop:forgetpair}
 and \propref{prop:inadjoint}, we deduce
 \begin{eqnarray*}
 \crH^n_{\ov{0}}(\Psi_{L})
 &\cong&
 [\cU(\tti\circ\cR)(\Psi_{L}), K(R,n)]_{\sset}
 \cong
 [(\tti\circ\cR)(\Psi_{L}),\Delta[1]\times K(R,n)]_{\ssets}\\
 &\cong&
 [\cR(\Psi_{L}),\cR(\Delta[1]\times K(R,n))]_{\ssets}
 \cong
 [\Psi_{L},\ttn(\cR(\Delta[1]\times K(R,n)))]_{\ssets}.
 \end{eqnarray*}
 \end{proof}

The next result is a first step in the direction of \conjref{conjecture1}.

\begin{theoremb}\label{thm:conjen0}
For any positive perversity, $\ov{p}$, there are isomorphisms,
$$\crH^k_{\ov{p}}(\ttK(R,n,[1],\ov{0});R)
\cong
\nat_{R}(\crH_{\ov{0}}^n,\crH_{\ov{p}}^k)
\cong
H^k(K(R,n);R).
$$
\end{theoremb}

Before giving the proof, we need two lemmas. For sake of simplicity, we denote 
$$Q(R,n)=((\Delta[0]\ast\Delta[0])\times K(R,n))/((\Delta[0]\ast\emptyset)\times K(R,n)),$$
the simplicial set given by the pushout,
\begin{equation}\label{equa:Qpi}
\xymatrix{
(\Delta[0]\ast\emptyset)\times K(R,n)\ar[r]\ar[d]_-{pr}
&
(\Delta[0]\ast\Delta[0])\times K(R,n)\ar[d]^-{\pi}\\
\Delta[0]\ast\emptyset\ar[r]
&
Q(R,n).
}
\end{equation}
In \propref{prop:eml0}, we have proven that $Q(R,n)$ is homotopically equivalent to the Eilenberg-MacLane space,
$\ttK(R,n,[1],\ov{0})$. The next lemma  expresses the fact that the identification of non regular simplices to a
point has no influence on the blown up cohomology, due to the equality $\tN^*(-)=0$ on these elements.

\begin{lemma}\label{lem:fucknotregular}
The canonical surjection $\pi$ in the diagram \eqref{equa:Qpi} induces an isomorphism of cochain complexes
$$\pi^*\colon
\tN^*_{\ov{p}}(Q(R,n);R)
\xrightarrow{\cong}
\tN^*_{\ov{p}}((\Delta[0]\ast\Delta[0])\times K(R,n);R).
$$
\end{lemma}

\begin{proof}
As the functor $\tN^*_{\ov{p}}(-)$ takes the value 0 on the non regular simplices and sends
inductive limits on projective limits, the result follows from the following pullback,
$$\xymatrix{
0=\tN^*_{\ov{p}}((\Delta[0]\ast\emptyset)\times K(R,n))
\ar[r]
&
\tN^*_{\ov{p}}((\Delta[0]\ast\Delta[0])\times K(R,n))\\
0=\tN^*_{\ov{p}}(\Delta[0]\ast\emptyset)
\ar[r]\ar[u]
&
\tN^*_{\ov{p}}(Q(R,n)).
\ar[u]_-{\pi^*}
}$$
\end{proof}

The second lemma introduces contractible simplicial $R$-modules.

\begin{lemma}\label{lem:secondcomplex}
Let $k$ be a positive integer, $\ov{p}$ a perversity on a poset $\ttP$ and $\Psi_{L}\in\ssetp$.
Then the simplicial $R$-module, defined by $n\mapsto \tN^k_{\ov{p}}(\Psi_{L}\otimes \Delta[n])$,
is contractible.
\end{lemma}

\begin{proof}
The simplicial set 
$\tN^k_{\ov{p}}(\Psi_{L}\otimes \Delta[\bullet])$, 
being a simplicial abelian group,
is Kan. Thus it is sufficient to prove that its homotopy groups $\pi_{i}(-)$ are trivial. For that,
we use \cite[Proposition 1]{MR431137}, where two cases are considered.

\medskip
$\bullet$ For $n=0$, we consider $x\in\tN^k_{\ov{p}}(\Psi_{L}\otimes \Delta[0])$. We have to prove
that there exists $y\in \tN^k_{\ov{p}}(\Psi_{L}\otimes \Delta[1])$ with $d_{0}y=x$ and $d_{1}y=0$, where
$d_{0},\,d_{1}\colon \tN^k_{\ov{p}}(\Psi_{L}\otimes \Delta[1])
\to
\tN^k_{\ov{p}}(\Psi_{L}\otimes \Delta[0])$
are the face operators. Let $\Delta[1]=[e_{0},e_{1}]$ and $\alpha=1_{e_{1}}\in N^0(\Delta[1])$ the cochain with values
0 on $e_{0}$ and $1$ on $e_{1}$. We denote
$\omega=\pi^*({\alpha})\in N^0(\Psi_{L}\otimes \Delta[1])$ its pullback by the projection
$\pi\colon \Psi_{L}\otimes \Delta[1]\to \Delta[1]$. We set
$$y=\omega\smile s_{0}x\in \tN^k_{\ov{p}}(\Psi_{L}\otimes \Delta[1]),$$
where $s_{0}\colon \tN^k_{\ov{p}}(\Psi_{L}\otimes \Delta[0])\to \tN^k_{\ov{p}}(\Psi_{L}\otimes \Delta[1])$ 
is the 0-degeneracy map. With these choices, from $d_{1}\omega=\pi^*(d_{1}\alpha)=0$,
 $d_{0}\omega=\pi^*(d_{0}\alpha)=1$ and the simplicial identities, 
we have $d_{1}y=0$ and $d_{0}y=x$, as expected.

\medskip
$\bullet$ If $n>0$, we have to prove that the homotopy groups 
are trivial. Let $x\in \tN^k_{\ov{p}}(\Psi_{L}\otimes \Delta[n])$ with all face restrictions 
$d_{i}x\in\tN^k_{\ov{p}}(\Psi_{L}\otimes \Delta[n-1])$ trivial.
Set $\Delta[n]=[e_{1},\dots,e_{n+1}]$ and $\Delta[n+1]=[e_{0},e_{1},\dots,e_{n+1}]$.
We choose $\alpha=\sum_{i>0} 1_{e_{i}}\in N^*(\Delta[n+1])$
and $\omega=\pi^*(\alpha)$ where $\pi\colon \Psi_{L}\otimes \Delta[n+1]\to \Delta[n+1]$ is the canonical projection.
We set
$$y=\omega\smile s_{0}x\in \tN^k_{\ov{p}}(\Psi_{L}\otimes \Delta[n+1]).$$
From $d_{0}s_{0}=\id$, $d_{0}\omega=\pi^*(d_{0}\alpha)$ and $d_{0}\alpha=1$, we deduce
$d_{0}y=x$. We now have to prove
\begin{equation}\label{equa:di}
d_{i}y=d_{i}\omega\smile d_{i}s_{0}(x)=0 \text{ for } i>0.
\end{equation}
If $i>1$, then we have $d_{i}s_{0}x=s_{0}d_{i-1}x=0$, by hypothesis on $x$. 
We still have to consider the case $i=1$ which corresponds to
$$d_{1}y=d_{1}\omega\smile d_{1}s_{0}x=\pi^*(d_{1}\alpha)\smile x.$$
Let us notice that the 0-cochain $d_{1}\alpha$ is the restriction of $\alpha$ to the face
$[e_{0},e_{2},\dots,e_{n+1}]$ and that $\alpha(e_{0})=0$. In the cup product
$\pi^*(d_{1}\alpha)\smile x$, the first term is evaluated on the first vertex, thus $d_{1}y=0$.
We have established \eqref{equa:di}.
\end{proof}

\begin{proof}[Proof of \thmref{thm:conjen0}]
Let $Z\in\sset$.
The two cochain maps, $p_{1}^*$ and $p_{2}^*$, induced by the projections,
$p_{1}\colon (\Delta[0]\ast\Delta[0])\times Z \to \Delta[0]\ast\Delta[0]$
and $p_{2}\colon (\Delta[0]\ast\Delta[0])\times Z \to Z$,
and the cup product, $\smile$, on $\tN^*_{\bullet}(-)$ give a cochain map
$\chi=\smile \circ (p_{1}^*
\otimes
p_{2}^*)$.
We compose it with the canonical injection to obtain a natural transformation,
$$\theta_{-}\colon \xymatrix@1{
 N^*(-)
 \ar[r]^-{1\otimes -}
 &
 N^*_{\ov{p}}(\Delta[0]\ast\Delta[0])\otimes_{R}N^*(-)
 \ar[r]^-{\chi}&
 N^*_{\ov{p}}((\Delta[0]\ast\Delta[0])\otimes (-)),
 }
 $$
between the functors $N^*(-)$ and $G(-)=N^*_{\ov{p}}((\Delta[0]\ast\Delta[0])\otimes (-))$,
from $\sset^\op$ to $\mdg$.
The theorem is proven if we show that $\theta_{Z}$ is a quasi-isomorphism for any $Z\in\sset$.
For that, we use \cite[Proposition 3.1.14]{MR3931682} as ``a theorem of acyclic models.'' 

\medskip
-- A first step is to establish that the two functors send inductive limits on limits and cofibrations on epimorphisms.
The only point which needs a proof is that $G$ sends cofibrations on epimorphisms. 
For that, we consider the 
cochain complex, $D(n)^*$, defined by
$$D(n)^k=\left\{
\begin{array}{cl}
R&\text{if}\; k=n,\,n+1,\\
0&\text{otherwise,}
\end{array}\right.
$$
and an isomorphism $d\colon D(n)^n\to D(n)^{n+1}$ as differential.
We first notice that, for any cochain complex, $C^*$, one has
$\Hom_{\mdg}(D(n)^*,C^*)=C^n$.
Let $\Psi_{L}\in\ssetp$, $Y\in \sset$. With the notations of
\lemref{lem:secondcomplex}, there are isomorphisms,
\begin{eqnarray}
\Hom_{\sset}(Y,\tN^n_{\ov{p}}(\Psi_{L}\otimes \Delta[\bullet]))
&\cong&
\Hom_{\sset}(Y,\Hom_{\mdg}(D({n}),\tN^*_{\ov{p}}(\Psi_{L}\otimes \Delta[\bullet]))) \nonumber \\ 
&\cong&
\Hom_{\sset}(Y,\Hom_{\ssetp}(\Psi_{L}\otimes \Delta[\bullet],\langle D(n) \rangle_{\ov{p}})) \label{equa:lados}\\
&\cong&
\Hom_{\sset}(Y,\HomD_{\ssetp}(\Psi_{L},\langle D(n) \rangle_{\ov{p}})) \label{equa:latres}\\
&\cong&
\HomD_{\ssetp}(\Psi_{L}\otimes Y,\langle D(n) \rangle_{\ov{p}})), \label{equa:laquatro}
\end{eqnarray}
where \eqref{equa:lados} comes from \propref{prop:adjoint}, \eqref{equa:latres} from  \remref{rem:ssetpsimplicial}
and \eqref{equa:laquatro} from \defref{def:simplicialcat}. The last isomorphism allows 
the determination of the set of 0-simplices,
\begin{eqnarray}
(\Hom_{\sset}(Y,\tN^n_{\ov{p}}(\Psi_{L}\otimes \Delta[\bullet])))_{0}
&\cong&
\Hom_{\ssetp}(\Psi_{L}\otimes Y, \langle D(n)\rangle_{\ov{p}})\nonumber \\
&\cong&
\Hom_{\mdg}(D(n),\tN^*_{\ov{p}}(\Psi_{L}\otimes Y))\nonumber\\
&\cong&
\tN^n_{\ov{p}}(\Psi_{L}\otimes Y).\label{equa:una}
\end{eqnarray}
Let $j\colon X\to Y$ be a cofibration in $\sset$. As $\tN^n_{\ov{p}}(\Psi_{L}\otimes \Delta[\bullet])$ is a Kan, contractible simplicial set, 
any diagram as below admits a dot extension,
$$\xymatrix{
X\ar[d]_{j}\ar[r]
&
\tN^n_{\ov{p}}(\Psi_{L}\otimes \Delta[\bullet])\\
Y\ar@{-->}[ru]
}$$
This implies the surjectivity of the  map obtained by composition with $j$,
$$j^{\sharp}\colon
\Hom_{\sset}(Y,\tN^n_{\ov{p}}(\Psi_{L}\otimes \Delta[\bullet]))
\to
\Hom_{\sset}(X,\tN^n_{\ov{p}}(\Psi_{L}\otimes \Delta[\bullet])).
$$
From the surjectivity at the level of the 0-simplices and \eqref{equa:una}, we deduce the surjectivity of
$$(\id\otimes j)^*\colon \tN^n_{\ov{p}}(\Psi_{L}\otimes Y)\to \tN^n_{\ov{p}}(\Psi_{L}\otimes X)$$
for any $\Psi_{L}\in\ssetp$. Setting $\Psi_{L}=\Delta[0]\ast\Delta[0]$ gives the desired result for the functor $G$.

\medskip
-- After this first step, we can apply \cite[Proposition 3.1.14]{MR3931682}.
The proof is thus reduced to the verification that $\theta_{\Delta[m]}$ is a quasi-isomorphism for any $m$. More specifically, we have to verify that
$$N^*(\Delta[m])\to \tN^*_{\ov{p}}((\Delta[0]\ast\Delta[0])\otimes \Delta[m])$$
is a quasi-isomorphism, which is clearly the case.
(Let us recall that $\crH^*_{\ov{p}}(\Delta[0]\ast\Delta[0])$
is isomorphic to
$H^*(\mathrm{pt})$.)
\end{proof}
%

%
\appendix

\section{Filtered face sets}\label{sec:ffsareback}

\begin{quote}
In \cite{CST1}, we study the blown up cohomology (called TW-cohomology) by using filtered face sets, in the
spirit of the $\Delta$-sets of Rourke and Sanderson (\cite{MR0300281}).
In this section, we compare the blown up cohomology  thus obtained with the blown up cohomology of \secref{sec:functors}.
\end{quote}

Let $\ttP$ be a poset.
Let us denote by $\Delta[\ttP]^{\face}$ the subcategory of $\Delta[\ttP]$ whose morphisms
\eqref{equa:morphisms}
come from injective maps $\Delta[k]\to\Delta[\ell]$.
\begin{definition}\label{def:ffs}
A \emph{filtered face set over $\ttP$} is a presheaf on the category $\Delta[\ttP]^{\face}$; i.e.,
 a functor $\Psi_{T}\colon \left(\Delta[\ttP]^{\face}\right)^{\op}\to \set$. We denote $\FFS$ the 
category of natural transformations between \ffss~over $\ttP$.
\end{definition}

Let $\Psi_{T}\in \FFS$ and $\Psi_{L}\in \ssetp$. We define a \ffs~over $\ttP$, $\cO_{\ttP}(\Psi_{L})\in\FFS$, by restriction
and a simplicial set over $\ttP$, $F_{\ttP}(\Psi_{T})\in\ssetp$, by left Kan extension as we do in \defref{def:functiri}. 
These constructions are extended in functors and the functor $F_{\ttP}$ is left-adjoint to $\cO_{\ttP}$.

\medskip
We take over the presentation of \subsecref{subsec:twofunctors}
for the blown up cochains on \ffss~introduced in \cite{CST1}.
For any $\sigma\in\ttN(\ttP)$ and $k\in\N$, we have set
 $\ttM_{\ov{p}}(\sigma,k) =\tN^k_{\ov{p}}(\Delta[\sigma])$
 and $\tN_{\ov{p}}^k(\Psi_{L})=\Hom_{\ssetp}(\Psi_{L},\ttM_{\ov{p}}(\bullet,k))$
 for $\Psi_{L}\in\ssetp$.
Let $\Psi_{T}\in \FFS$ 
and $M\in\mdg$, we set
$$\tN^{\Delta,k}_{\ov{p}}(\Psi_{T})=\Hom_{\FFS}(\Psi_{T},\ttM_{\ov{p}}(\bullet,k))$$
and
$$\langle M\rangle_{\ov{p}}^{\Delta}=\varinjlim_{\sigma\in\Delta[P]^{\face}}\Hom_{\mdg}(M,\ttM_{\ov{p}}(\sigma,*)).$$
We  summarize these data in the following diagram composed of three pairs of adjoint functors:
$$\xymatrix{
\ssetp
\ar@<1ex>[rr]^-{\cO_{\ttP}}
\ar@<1ex>[rdd]^-{\tN_{\ov{p}}}
&&\FFS
\ar@<1ex>[ll]^-{F_{\ttP}}
\ar@<-1ex>[ldd]_{\tN_{\ttP}^\Delta}\\
&&\\
&\mdg
\ar@<-1ex>[ruu]_{\langle-\rangle_{\ttP}^\Delta}
\ar@<1ex>[luu]^{\langle-\rangle_{\ttP}}
&
}$$

\begin{remark}\label{rem:adrouette}
Let $\Psi_{T}\in\FFS$ and $\Delta[J]\in\ttN(\ttP)$. We set
$$\cR_{\ttP}(\Psi_{T})(\Delta[J])=\Hom_{\FFS}(\cO_{\ttP}(\Delta[J]),\Psi_{T}).$$
This definition extends in a functor $\cR_{\ttP}\colon \FFS\to \ssetp$, which is a right adjoint to $\cO_{\ttP}$. 
We do not use the functor $\cR_{\ttP}$ in this work.
\end{remark}

\begin{proposition}\label{prop:deltaornot}
Let $\ttP$ be a poset, $\ov{p}$ a perversity on $\ttP$, $\Psi_{L}\in \ssetp$ and $\Psi_{T}\in\FFS$. 
By denoting $\crH_{\ov{p}}^{\Delta,*}(-)$ the homology of $\tN^{\Delta,*}_{\ov{p}}(-)$, there are natural isomorphisms,
$$\crH^{\Delta,*}_{\ov{p}}(\Psi_{T};R)\cong
 \crH^*_{\ov{p}}(F_{\ttP}(\Psi_{T});R)
 \;\text{ and }\;
 \crH_{\ov{p}}^*(\Psi_{L};R)\cong
 \crH^{\Delta,*}_{\ov{p}}(\cO_{\ttP}(\Psi_{L});R).$$
\end{proposition}

\begin{proof}
Let $\Psi_{T}\in\FFS$. 
The functor $F_{\ttP}$ is defined as a direct limit,
$F_{\ttP}(\Psi_{T})=\varinjlim_{\Delta[J]\to \Psi_{T}} \Delta[J]$,
and we have a natural isomorphism at the  level of the complexes,
$\tN^{\Delta,*}_{\ov{p}}(\Psi_{T})\cong \tN^*_{\ov{p}}(F_{\ttP}(\Psi_{T}))$.

\medskip
Let $\Psi_{L}\in\ssetp$. The second isomorphism is not so direct since  
$\tN^*_{\ov{p}}(\Psi_{L})\neq \tN_{\ov{p}}^{\Delta,*}(\cO_{\ttP}(\Psi_{L}))$:
for instance, if $\ttP=\{0\}$ and $L=\Delta[0]$, then
$\tN^*_{\ov{p}}(\Delta[0])$ is the complex of normalized cochains and
$\tN^{\Delta,*}_{\ov{p}}(\cO_{\ttP}(\Delta[0]))$ is the complex of non normalized cochains.

\medskip
Denote $\cS(\Psi_{L})=\left\{
\sigma\colon \Delta[J]\to \Psi_{L}\mid \Delta[J]\in \ttN(\ttP)\right\}$. 
The functor $\cO_{\ttP}$ being compatible with direct limits, we have
$$\cO_{\ttP}(\Psi_{L})=\varinjlim_{\sigma\in\cS(\Psi_{L})} \cO_{\ttP}(\Delta[J]).$$
Applying the functor $\tN^{\Delta,*}_{\ov{p}}(-)$ which sends direct limits to inverse limits, we get
$$\tN^{\Delta,*}(\cO_{\ttP}(\Psi_{L}))=\varprojlim_{\sigma\in\cS(\Psi_{L})} \tN^{\Delta,*}_{\ov{p}}(\cO_{\ttP}(\Delta[J])).$$
On the other hand, we have
$$\tN^*_{\ov{p}}(\Psi_{L})=\varprojlim_{\sigma\in\cS(\Psi_{L})}\tN^*_{\ov{p}}(\Delta[J]).$$
Thus, we are reduced to prove the existence of a natural homotopy equivalence between
$\tN^*_{\ov{p}}(\Delta[J])$ 
and 
$\tN^{\Delta,*}_{\ov{p}}(\cO_{\ttP}(\Delta[J]))$,
for $ \Delta[J]\in \ttN(\ttP)$.

\medskip
If $L\in \sset$, we denote $N^*(L)$ the normalized  cochain complex and $C^*(L)$ the non normalized one. There
exists a natural homotopy equivalence between them, $(f,g,H)$, 
$$\xymatrix@1{ 
N^*(L) 
\ar@<.4ex>[r]^g
&
C^*(L),  
\ar@<.4ex>[l]^f
}\;\;
f\circ g=\id_{N^*(L)},
$$
and  $H$ a natural homotopy between $g\circ f$ and $\id_{C^*(L)}$.

First, we consider the global complexes $\tN^{\Delta,*}$ and $\tN^*$ without referring to a perversity. By definition,
for $\Delta[J]=\Delta[j_{0}]\ast\dots\ast\Delta[j_{n-1}]\ast \Delta[j_{n}]$, we have
$$\tN^{\Delta,*}(\cO_{\ttP}(\Delta[J]))=
\ov{N}^*(\Delta[j_{0}])\otimes\dots\otimes \ov{N}^*(\Delta[j_{n-1}])\otimes C^*(\Delta[j_{n}]),$$
with 
$\ov{N}^*(\Delta[j_{k}])=\varprojlim_{\{\Delta[i]\to\Delta[j_{k}]\}} N^*(\tc \Delta[i])$.
As the apex of $\tc \Delta[i])$ does not appear in the limit, we do \emph{not} have
$\ov{N}^*(\Delta[j_{k}])=C^*(\tc \Delta[j_{k}])$.
To manage with the apex, we consider the cone
$\tc \Delta[i]$ as the direct limit of
$$\xymatrix@1{
\ast&
\Delta[i]\ar[r]\ar[l]&
\Delta[i]\ast\Delta[1].
}$$
Applying the normalized functor, we obtain $N^*(\tc \Delta[i])$ as the inverse limit of
$$\xymatrix@1{
R\ar[r]&
N^*(\Delta[i])&
N^*(\Delta[i]\times \Delta[1]).\ar[l]
}$$
Using the commutation of direct and inverse  limits, we can write $\ov{N}^*(\Delta[j_{k}])$ as the pullback of
$$\xymatrix@1{
R\ar[r]&
N^*(\Delta[j_{k}])&
\ar[l]
\varprojlim_{\{\Delta[i]\to\Delta[j_{k}]\}}N^*(\Delta[i]\times \Delta[1]).
}$$
By applying the Eilenberg-Zilber theorem and the Alexander-Whitney map, we get a natural homotopy equivalence
between $\ov{N}^*(\Delta[j_{k}])$ and the pullback of
$$\xymatrix@1{
R\ar[r]&
N^*(\Delta[j_{k}])&
\ar[l]
\varprojlim_{\{\Delta[i]\to \Delta[j_{k}\}}N^*(\Delta[i])\otimes N^*(\Delta[1]).
}$$
Let us notice that the right hand expression is the tensor product $C^*(\Delta[j_{k}])\otimes N^*(\Delta[1])$.
Using the natural homotopy equivalence, $(f,g,H)$, the Eilenberg-Zilber theorem and the Alexander-Whitney map, 
we have a natural homotopy equivalence between $\ov{N}^*(\Delta[j_{k}])$ and the pullback
$$\xymatrix@1{
R\ar[r]&
C^*(\Delta[j_{k}])&
\ar[l]
C^*(\Delta[j_{k}]\times \Delta[1]).
}$$
As non normalized cochains send direct limits on inverse limits, we have 
a natural homotopy equivalence between
$C^*(\tc \Delta[j_{k}])$ and $\ov{N}^*(\Delta[j_{k}])$
and thus  
a natural homotopy equivalence between
$\tN^*(\Delta[J])$ and $\tN^{\Delta,*}(\cO_{\ttP}(\Delta[J]))$.

Finally, as the perverse degree is the sum of the degrees of some factors of the tensor product,
which are preserved all along the previous process, we get a natural homotopy equivalence between
$\tN^*_{\ov{p}}(\Delta[J])$
and
$\tN^{\Delta,*}_{\ov{p}}(\cO_{\ttP}(\Delta[J]))$,
as expected.
\end{proof}

The determination of the $\ov{0}$-intersection cohomology of a filtered space 
as the ordinary singular cohomology of the space
 needs an hypothesis of normality
(see \cite{GM1}). This is also true for \ffss~and we send
the reader to \cite[Subsection 1.5]{CST1} for more details. 
As the reference \cite{CST1} is written for $\ttP=[n]$,
we restrict to this case for the end of this section.

\begin{proposition}\label{prop:normal}
Let $\ttP=[n]$  and $\Psi_{L}\in \ssetp$.
The map $(\cO\circ \tti\circ\cR )(\Psi_{L})\to \cO(\Psi_{L})$,
coming from the adjunction $(\tti,\cR)$,
is a normalization of the filtered face set $\cO(\Psi_{L})$.
\end{proposition}

\begin{proof}
We refer to \cite[Definitions 1.55 and 1.59]{CST1} for the definitions of normal \ffs~and normalization.
Set ${T}=\cO(\Psi_{L})$. 
By definition, the expression $\cO(\cR(\Psi_{L}))$ corresponds to $T_{+}$ in \cite{CST1}.
Recall that the functor $\tti$ is constructed from a left Kan extension. Therefore, any simplex in 
$(\tti\circ \cR)(\Psi_{L})$ which is not in 
$(\cR\circ \tti\circ \cR)(\Psi_{L})=\cR(\Psi_{L})$
is a non regular face of a simplex in $\cR(\Psi_{L})$. This gives  condition (a) of \cite[Definition 1.55]{CST1}.
The unicity condition (b) of \cite[Definition 1.55]{CST1} comes from the fact that all added faces by $\tti$ are distinct.
Finally, the adaptation of the equality $(\cR\circ \tti\circ \cR)(\Psi_{L})=\cR(\Psi_{L})$ to the notations of \cite{CST1} is 
$((\cO\circ \tti\circ \cR)(\Psi_{L}))_{+}=(\cO(\Psi_{L}))_{+}$,
which is the required property of a normalization.
\end{proof}

From the Propositions \ref{prop:deltaornot}, \ref{prop:normal}
and from \cite[Propositions 1.54, 1.57  and 1.60]{CST1},
we deduce immediately the following result.

\begin{corollary}\label{cor:normal}
Let $R$ be a Dedekind domain, $\ov{p}$ a perversity, $\ttP=[n]$ and $\Psi_{L}\in \ssetp$.
Then there exist isomorphisms,
and 
$\crH^*_{\ov{\infty}}(\Psi_{L};R)\cong \mbox{} \crH^{\Delta,*}_{\ov{\infty}}(\cO(\Psi_{L});R)\cong H^*(L^{\reg};R)$.
\end{corollary}

In the statement of \cite[Proposition 1.57]{CST1}, the ring $R$ is required to be principal, but
this hypothesis is only used for the existence of a universal coefficient formula which also
exists for Dedekind domains.

\section{Simplicial category}\label{sec:simplicialcat}

\subsection{Definitions}\label{subsec:defsimplicialcat}
Recall from \cite[Chapter II]{MR0223432} (\cite[Chapter 2]{MR1711612} or \cite[Appendix A.1]{MR2522659})
  the following definition.

\begin{definition}\label{def:simplicialcat}
A category $\cC$ is a \emph{simplicial category} if there is a mapping
space functor 
$\HomD_{\cC}(-,-)\colon \cC^{\op}\times \cC \to \sset$,
satisfying the following  properties  for A and B objects in $\cC$, $K$ and $L$ in $\sset$.
\begin{enumerate}[(i)]
\item $\HomD_{\cC}(A,B)_{0}=\Hom_{\cC}(A,B)$.
\item The functor $\HomD_{\cC}(A,-)\colon \cC\to \sset$ has a left adjoint,
$A\otimes -\colon \sset\to \cC$; i.e.
$$\HomD_{\cC}(A\otimes K,B)\cong \HomD_{\sset}(K,\HomD_{\cC}(A,B)),$$
 which is associative in the sense there is an isomorphism,
$A\otimes (K\otimes L)\cong (A\otimes K)\otimes L$,
natural in $A\in\cC$ and $K,\,L\in\sset$.
\item The functor $\HomD_{\cC}(-,B)\colon \cC^{\op}\to\sset$ has a right adjoint
$B^{-}\colon \sset \to \cC^{\op}$; i.e.,
$$\HomD_{\sset}(K,\HomD_{\cC}(A,B))\cong \HomD_{\cC}(A, B^K),$$
for any $K\in\sset$ and  $A,\,B\in\cC$.
\end{enumerate}
\end{definition}

With the previous notation, for all $n\geq 0$, we have 
$$\HomD_{\cC}(A,B)_{n}=\Hom_{\cC}(A\otimes \Delta[n],B).$$

Let $K$ be a simplicial set. Two elements $x,y\in K_{0}$ are strictly homotopic 
(\cite[Section II.1]{MR0223432}) if there exists $z\in K_{1}$ with $d_{1}z=x$ and $d_{0}z=y$. 
The notion of homotopy is the generated equivalence relation, denoted $\sim$.
Thus, in a simplicial category $\cC$, we have  homotopy classes  defined
by $[A,B]=\pi_{0}\HomD_{\cC}(A,B)$. 
More specifically, let  $f,g\colon A\to B\in \cC$. By definition, we have
 $f\sim g$ if there exists $H$ (resp. $H'$) such that the left-hand (resp. right-hand) following diagram commutes,
$$
\xymatrix{
&B^{\Delta[1]}\ar[d]^{(\ev_{0},\ev_{1})}\\
A\ar[r]_-{f\times g}\ar[ru]^-{H}&
B^{\partial\Delta[1]}\cong B\times B,
}
\quad\quad \quad
\xymatrix{
A\otimes \Delta[1]\ar[dr]^{H'}&\\
A\otimes \partial \Delta[1]\ar[r]_-{(f, g)}\ar[u]&
B.
}
$$

\begin{definition}\label{def:sfibrant}
Let $\cC$ be a simplicial category.
An object $A$ of a simplicial category $\cC$ is s-\emph{fibrant} if, for any object $Z$ of $\cC$, the 
simplicial set $\HomD_{\cC}(Z,A)$ is a Kan complex.
A map $f\colon A\to B\in\cC$ is a \emph{weak s-equivalence} if for any object $Z$ of $\cC$, there is
an isomorphism,
$\pi_{0}\HomD_{\cC}(Z,A)\xrightarrow{\cong}\pi_{0} \HomD_{\cC}(Z,B)$.  
\end{definition}

\begin{definition}\label{def:cathomotopique}
The \emph{homotopy category, ${\rm Ho}$-s-$\cC$, associated to a simplicial category, $\cC$,} has
\begin{itemize}
\item for objects, the s-fibrants objects of $\cC$,
\item for morphisms, the connected components of morphisms of $\cC$; i.e.,
$$[A,B]=\pi_{0}\HomD_{\cC}(A,B).$$
\end{itemize}
Two objects of $\cC$ are \emph{homotopically equivalent}, denoted by $A\simeq_{\cC}B$, if there exist 
two morphisms of $\cC$, $f\colon A\to B$ and $g\colon B\to A$ such that $f\circ g$ is homotopic to $\id_{B}$ and 
$g\circ f$ homotopic to $\id_{A}$.
\end{definition}

In the category $\sset$, the notions of s-fibrant objects and of  
weak s-equivalences coincide with those of the Kan closed model structure.
In particular, ${\rm Ho}$-s-$\cC$ is the localisation at the weak homotopy equivalences.

\begin{proposition}
Let $\cC$ be a simplicial category.
A weak s-equivalence between s-fibrant objects of $\cC$ is a homotopy equivalence in $\cC$.
\end{proposition}

\begin{proof}
Let $f\colon A\to B\in\cC$ be a weak s-equivalence with $A$ and $B$ s-fibrant.
A right inverse up to homotopy of $f$ is the map $g$ given by the surjectivity of $f_{*}\colon [Z,A]\to [Z,B]$
applied to the identity on $B$,
$$\xymatrix{
&A\ar[d]^{f}\\
B\ar@{=}[r]^{\id}\ar[ru]^{g}&B
}$$
Now, the injectivity of $f_{*}\colon [A,A]\to [A,B]$ gives $g\circ f$ homotopic to $\id_{A}$.
\end{proof}

\subsection{Infinite loop space}\label{subsec:infiniloop}
Let $\cC$ be a complete and cocomplete simplicial category. We denote $\ast$ the final  object of $\cC$.
We now introduce the notion of based loop space in $\cC$. %

\begin{definition}\label{def:pointed}
A \emph{pointed object} of $\cC$ is a couple $(X,\epsilon)$ of an object $X$ of $\cC$
and a morphism $\epsilon\colon \ast\to X$.
The \emph{pointed loop space}  $\Omega_{\epsilon} X$ of $(X,\epsilon)$
is the pull-back
\begin{equation}\label{equa:loop}
\xymatrix{
\Omega_{\epsilon}X\ar[rr]\ar[d]&&X^{\Delta^1}\ar[d]^{(\ev_{0},\ev_{1})}\\
\ast\ar[rr]^-{(\epsilon,\epsilon)}&&X^{\partial \Delta^1}\cong X\times X. 
}
\end{equation}
\end{definition}
Let $Z$ be an object of $\cC$, we apply the functor $\HomD_{\cC}(Z,-)$ to the previous diagram and get
the pullback,
\begin{equation}\label{equa:loop2}
\xymatrix{
\HomD_{\cC}(Z,\Omega_{\epsilon}X)\ar[rr]\ar[d]&&
\HomD_{\cC}(Z,X^{\Delta[1]})
\ar[d]^{\HomD_{\cC}(Z,(\ev_{0},\ev_{1}))}\\
\HomD_{\cC}(Z,\ast)\ar[rr]^-{\HomD_{\cC}(Z,(\epsilon,\epsilon))}&&
\HomD_{\cC}(Z,X^{\partial \Delta[1]}).
}
\end{equation}
By using  \defref{def:simplicialcat}.(iii), the right-hand vertical map
is induced by the canonical inclusion $\iota\colon \partial\Delta[1]\hookrightarrow \Delta[1]$, up to isomorphisms,
$$
\xymatrix{
\HomD_{\cC}(Z,X^{\Delta[1]})
\ar[r]^-{\cong}
\ar[d]_{\HomD_{\cC}(Z,(\ev_{0},\ev_{1}))}&
\HomD_{\sset}(\Delta[1],\HomD_{\cC}(Z,X))
\ar[d]^-{\iota^*}\\
\HomD_{\cC}(Z,X^{\partial \Delta^1})
\ar[r]^-{\cong}&
\HomD_{\sset}(\partial\Delta[1],\HomD_{\cC}(Z,X)).
}
$$
Suppose $X$ is s-fibrant, then the simplicial set $\HomD_{\cC}(Z,X)$ is Kan and
$\iota^*$ is a Kan fibration (\cite[Corollary 5.3]{MR1711612}).
As the diagram \eqref{equa:loop2} is a pull-back, with $\HomD_{\cC}(Z,\ast)\cong \Delta[0]$, 
we obtain that 
$\HomD_{\cC}(Z,\Omega_{\epsilon}X)$
is the fiber of a Kan fibration, thus it is a Kan simplicial set.
Finally, we have proven that $\Omega_{\epsilon}X$ is s-fibrant if $X$ is too.

\begin{definition}\label{def:infiniteloop}
An \emph{infinite loop space} in $\cC$ is a sequence of s-fibrant pointed objects of $\cC$,
$\{ (B_i,\epsilon_i) \}_{i\in \N}$,
such that $B_{i}$ is weakly s-equivalent to $\Omega_{\epsilon_i} B_{i+1}$.
\end{definition}

\begin{remark}
Let $\{ (B_i,\epsilon_i) \}_{i\in \N}$ be an infinite loop space in $\cC$. We can define a cohomological functor,
$\B\colon \cC^{\op}\to \mathbf{Ab}$-$\mathbf{gr}$,
with values in the category of graded abelian groups, by
$$A\mapsto \B^i(A)=\pi_{0}\HomD_{\cC}(A,B_{i}).$$
\end{remark}


\begin{thebibliography}{10}

\bibitem{MR2286904}
Markus Banagl, \emph{Topological invariants of stratified spaces}, Springer
  Monographs in Mathematics, Springer, Berlin, 2007. \MR{2286904 (2007j:55007)}

\bibitem{Bor}
A.~Borel and et~al., \emph{Intersection cohomology}, Modern Birkh\"auser
  Classics, Birkh\"auser Boston Inc., Boston, MA, 2008, Notes on the seminar
  held at the University of Bern, Bern, 1983, Reprint of the 1984 edition.
  \MR{2401086 (2009k:14046)}

\bibitem{MR0425956}
A.~K. Bousfield and V.~K. A.~M. Gugenheim, \emph{On {${\rm PL}$} de {R}ham
  theory and rational homotopy type}, Mem. Amer. Math. Soc. \textbf{8} (1976),
  no.~179, ix+94. \MR{0425956}

\bibitem{MR431137}
Henri Cartan, \emph{Th\'{e}ories cohomologiques}, Invent. Math. \textbf{35}
  (1976), 261--271. \MR{431137}

\bibitem{CST6}
David Chataur, Martintxo Saralegi-Aranguren, and Daniel Tanr\'e, \emph{Steenrod
  squares on intersection cohomology and a conjecture of {M} {G}oresky and {W}
  {P}ardon}, Algebr. Geom. Topol. \textbf{16} (2016), no.~4, 1851--1904.
  \MR{3546453}

\bibitem{CST7}
\bysame, \emph{Singular decompositions of a cap product}, Proc. Amer. Math.
  Soc. \textbf{145} (2017), no.~8, 3645--3656. \MR{3652815}

\bibitem{CST4}
\bysame, \emph{Blown-up intersection cohomology}, An alpine bouquet of
  algebraic topology, Contemp. Math., vol. 708, Amer. Math. Soc., Providence,
  RI, 2018, pp.~45--102. \MR{3807751}

\bibitem{CST1}
\bysame, \emph{Intersection cohomology, simplicial blow-up and rational
  homotopy}, Mem. Amer. Math. Soc. \textbf{254} (2018), no.~1214, viii+108.
  \MR{3796432}

\bibitem{CST2}
\bysame, \emph{Poincar\'e duality with cap products in intersection homology},
  Adv. Math. \textbf{326} (2018), 314--351. \MR{3758431}

\bibitem{CST3}
\bysame, \emph{{Intersection Homology. General perversities and topological
  invariance}}, Illinois J. Math. \textbf{63} (2019), no.~1, 127--163.

\bibitem{CST5}
\bysame, \emph{Blown-up intersection cochains and {D}eligne's sheaves},
  Geometriae Dedicata \textbf{204} (2020), no.~1, 315--337.

\bibitem{MR3931682}
Denis-Charles Cisinski, \emph{Higher categories and homotopical algebra},
  Cambridge Studies in Advanced Mathematics, vol. 180, Cambridge University
  Press, Cambridge, 2019. \MR{3931682}

\bibitem{2018arXiv180104797D}
Sylvain Douteau, \emph{A simplicial approach to stratified homotopy theory},
  Trans. Amer. Math. Soc. \textbf{374} (2021), no.~2, 955--1006. \MR{4196384}

\bibitem{FR2}
Greg Friedman, \emph{Intersection homology with general perversities}, Geom.
  Dedicata \textbf{148} (2010), 103--135. \MR{2721621}

\bibitem{Greg}
\bysame, \emph{Singular intersection homology}, New Mathematical Monographs,
  Cambridge University Press, 2020.

\bibitem{MR1711612}
Paul~G. Goerss and John~F. Jardine, \emph{Simplicial homotopy theory}, Progress
  in Mathematics, vol. 174, Birkh\"auser Verlag, Basel, 1999. \MR{1711612
  (2001d:55012)}

\bibitem{MR761809}
Mark Goresky, \emph{Intersection homology operations}, Comment. Math. Helv.
  \textbf{59} (1984), no.~3, 485--505. \MR{761809 (86i:55008)}

\bibitem{MR440533}
Mark Goresky and Robert MacPherson, \emph{La dualit\'{e} de {P}oincar\'{e} pour
  les espaces singuliers}, C. R. Acad. Sci. Paris S\'{e}r. A-B \textbf{284}
  (1977), no.~24, A1549--A1551. \MR{440533}

\bibitem{GM1}
\bysame, \emph{Intersection homology theory}, Topology \textbf{19} (1980),
  no.~2, 135--162. \MR{572580 (82b:57010)}

\bibitem{GM2}
\bysame, \emph{Intersection homology. {II}}, Invent. Math. \textbf{72} (1983),
  no.~1, 77--129. \MR{696691 (84i:57012)}

\bibitem{MR932724}
\bysame, \emph{Stratified {M}orse theory}, Ergebnisse der Mathematik und ihrer
  Grenzgebiete (3) [Results in Mathematics and Related Areas (3)], vol.~14,
  Springer-Verlag, Berlin, 1988. \MR{932724 (90d:57039)}

\bibitem{MR1014465}
Mark Goresky and William Pardon, \emph{Wu numbers of singular spaces}, Topology
  \textbf{28} (1989), no.~3, 325--367. \MR{1014465}

\bibitem{GS}
Mark Goresky and Paul Siegel, \emph{Linking pairings on singular spaces},
  Comment. Math. Helv. \textbf{58} (1983), no.~1, 96--110. \MR{699009
  (84h:55004)}

\bibitem{Hov}
Mark Hovey, \emph{Intersection homological algebra}, New topological contexts
  for {G}alois theory and algebraic geometry ({BIRS} 2008), Geom. Topol.
  Monogr., vol.~16, Geom. Topol. Publ., Coventry, 2009, pp.~133--150.
  \MR{2544388 (2010g:55009)}

\bibitem{joyalbarcelona}
Andr{\'e} {Joyal}, \emph{{The theory of quasi-categories and its
  applications}}, Quadern 45, Vol. II, Centre de Recerca Matem{\`a}tica
  Barcelona, 2008.

\bibitem{zbMATH05219541}
Andr\'e {Joyal} and Myles {Tierney}, \emph{{Quasi-categories vs Segal
  spaces.}}, {Categories in algebra, geometry and mathematical physics.
  Conference and workshop in honor of Ross Street's 60th birthday, Sydney and
  Canberra, Australia, July 11--16/July 18--21, 2005}, Providence, RI: American
  Mathematical Society (AMS), 2007, pp.~277--326 (English).

\bibitem{MR2207421}
Frances Kirwan and Jonathan Woolf, \emph{An introduction to intersection
  homology theory}, second ed., Chapman \& Hall/CRC, Boca Raton, FL, 2006.
  \MR{2207421 (2006k:32061)}

\bibitem{LurieSecondBook}
Jacob Lurie, \emph{Higher algebra}, Available at
  \url{http://www.math.harvard.edu/~lurie/papers/HA.pdf}.

\bibitem{MR2522659}
\bysame, \emph{Higher topos theory}, Annals of Mathematics Studies, vol. 170,
  Princeton University Press, Princeton, NJ, 2009. \MR{2522659}

\bibitem{MacPherson90}
Robert MacPherson, \emph{Intersection homology and perverse sheaves},
  Unpublished AMS Colloquium Lectures, San Francisco, 1991.

\bibitem{MR0196744}
Michael~C. McCord, \emph{Singular homology groups and homotopy groups of finite
  topological spaces}, Duke Math. J. \textbf{33} (1966), 465--474. \MR{0196744
  (33 \#4930)}

\bibitem{MooreInvariant}
J.~C. Moore, \emph{Homotopie des complexes mono\"\i daux, {II}}, S\'eminaire
  Henri Cartan \textbf{7} (1954-1955), no.~2 (fr), talk:19.

\bibitem{MR0223432}
Daniel~G. Quillen, \emph{Homotopical algebra}, Lecture Notes in Mathematics,
  No. 43, Springer-Verlag, Berlin-New York, 1967. \MR{0223432}

\bibitem{Qui}
Frank Quinn, \emph{Homotopically stratified sets}, J. Amer. Math. Soc.
  \textbf{1} (1988), no.~2, 441--499. \MR{928266 (89g:57050)}

\bibitem{MR0300281}
C.~P. Rourke and B.~J. Sanderson, \emph{{$\triangle $}-sets. {I}. {H}omotopy
  theory}, Quart. J. Math. Oxford Ser. (2) \textbf{22} (1971), 321--338.
  \MR{0300281 (45 \#9327)}

\bibitem{saralegiaranguren2019refinement}
Martintxo Saralegi-Aranguren, \emph{Refinement invariance of intersection
  (co)homologies}, Homology Homotopy Appl. \textbf{23} (2021), no.~1, 311--340.
  \MR{4170473}

\bibitem{ST2}
Martintxo {Saralegi-Aranguren} and Daniel {Tanr{\'e}}, \emph{Poincar\'{e}
  duality, cap product and {B}orel-{M}oore intersection homology}, Q. J. Math.
  \textbf{71} (2020), no.~3, 943--958. \MR{4142716}

\bibitem{ST1}
\bysame, \emph{Variations on {P}oincar\'{e} duality for intersection homology},
  Enseign. Math. \textbf{65} (2020), no.~1-2, 117--154. \MR{4057357}

\bibitem{Serrekpin}
Jean-Pierre Serre, \emph{Les espaces $k(\pi ,n)$}, S\'eminaire Henri Cartan
  \textbf{7} (1954-1955), no.~1 (fr), talk:1.

\bibitem{Sieg}
P.~H. Siegel, \emph{Witt spaces: a geometric cycle theory for {$K{\rm
  O}$}-homology at odd primes}, Amer. J. Math. \textbf{105} (1983), no.~5,
  1067--1105. \MR{714770}

\bibitem{2018arXiv180411274T}
Dai {Tamaki} and Hiro~Lee {Tanaka}, \emph{{Stellar Stratifications on
  Classifying Spaces}}, arXiv e-prints (2018), arXiv:1804.11274.

\bibitem{MR2575092}
David Treumann, \emph{Exit paths and constructible stacks}, Compos. Math.
  \textbf{145} (2009), no.~6, 1504--1532. \MR{2575092}

\bibitem{MR2591969}
Jon Woolf, \emph{The fundamental category of a stratified space}, J. Homotopy
  Relat. Struct. \textbf{4} (2009), no.~1, 359--387. \MR{2591969}

\end{thebibliography}
%
\providecommand{\bysame}{\leavevmode\hbox to3em{\hrulefill}\thinspace}
\providecommand{\MR}{\relax\ifhmode\unskip\space\fi MR }
\providecommand{\MRhref}[2]{%
  \href{http://www.ams.org/mathscinet-getitem?mr=#1}{#2}
}
\providecommand{\href}[2]{#2}

 \end{document}